\documentclass{article}

\usepackage[numbers,sort&compress]{natbib}
\usepackage[dvipsnames]{xcolor}

\usepackage{mathtools}
\usepackage{authblk}
\usepackage[title]{appendix}
\usepackage{geometry}
\usepackage{subfiles}
\usepackage{caption}
\usepackage{subcaption}
\usepackage{setspace}
\usepackage{amsthm}
\usepackage{amsmath}
\usepackage{amssymb}
\usepackage{ulem}

\usepackage{mathrsfs}

\mathtoolsset{showonlyrefs}

\numberwithin{equation}{section}

\usepackage{float}
\usepackage{fancyhdr}
\usepackage{graphicx}
\usepackage[colorlinks,            linkcolor=black,            anchorcolor=black,            citecolor=black            ]{hyperref}
\usepackage{listings}

\newtheorem{corollary}{Corollary}[section]
\newtheorem{conjecture}{Conjecture}[section]
\newtheorem{con-proposition}{``Proposition''}
\newtheorem{proposition}{Proposition}[section]
\newtheorem{theorem}{Theorem}[section]
\newtheorem{theorem_sec}{Theorem}[section]

\newtheorem{lemma}{Lemma}[section]

\newtheorem{assumption}{Assumption}[section]
\newtheorem*{theorem*}{Theorem}

\DeclareMathOperator\init{init}
\DeclareMathOperator\myin{in}

\newcommand{\p}{\partial}
\newcommand{\beq}{\begin{equation}}
\newcommand{\eeq}{\end{equation}}
\newcommand{\pvf}{(\frac{2\pi}{V_F})}

\newcommand{\bconj}{\begin{conjecture}}
	\newcommand{\econj}{\end{conjecture}}

\newcommand{\bueq}{\begin{equation*}}
\newcommand{\eueq}{\end{equation*}}

\newcommand{\bthm}{\begin{theorem}}
	\newcommand{\ethm}{\end{theorem}}

\newcommand{\eps}{\epsilon}
\newcommand{\veps}{\varepsilon}

\newcommand{\myBtwo}{\Theta}

\newcommand{\pe}{p^{\varepsilon}}

\providecommand{\keywords}[1]
{
	\small	
	\textbf{{Keywords:}} #1
}
\providecommand{\msc}[1]
{
	\small	
	\textbf{{Mathematics Subject Classification:}} #1
}

\title{A simplified voltage-conductance kinetic model for interacting neurons and its asymptotic limit}

\author{Jos\'e A. Carrillo\thanks{Mathematical Institute, University of Oxford, Oxford OX2 6GG, UK (carrillo@maths.ox.ac.uk)}, \quad Xu'an Dou\thanks{Beijing International Center for Mathematical Research, Peking University, Beijing, 100871, China (dxa@pku.edu.cn)}, \quad and \quad  Zhennan Zhou\thanks{Beijing International Center for Mathematical Research, Peking University, Beijing, 100871, China (zhennan@bicmr.pku.edu.cn).}}

\begin{document}
\maketitle
\begin{abstract}
The voltage-conductance kinetic model for the collective behavior of neurons has been studied by scientists and mathematicians for two decades, but the rigorous analysis of its solution structure has been only partially obtained in spite of plenty of numerical evidence in various scenarios. In this work, we consider a simplified voltage-conductance model in which the velocity field in the voltage variable is in a separable form. The long time behavior of the simplified model is fully investigated leading to the following dichotomy: either the density function converges to the global equilibrium, or the firing rate diverges as time goes to infinity. Besides, the fast conductance asymptotic limit is justified and analyzed, where the solution to the limit model either blows up in finite time, or globally exists leading to time periodic solutions. An important implication of these results is that the non-separable velocity field, or physically the leaky mechanism, is a key element for the emergence of periodic solutions in the original model based on the available numerical evidence.

\end{abstract}

\keywords{integrate-and-fire neurons, voltage-conductance model, kinetic Fokker-Planck equation, long time behavior, asymptotic limit, periodic solution}

\msc{35B10; 35B40; 35Q84; 92B20}
	\section{Introduction}

	Modeling the collective behavior of biological neurons via a mean-field description of the population density has been a successful approach, which leads to nonlinear partial differential equations or stochastic differential equations with new structures  (e.g. \cite{abbott1993asynchronous,brunel1999fast,fusi1999collective,cai2004effective} and  \cite[Chapter 13]{gerstner2014neuronal}). While such mean-field equations have been shown to be useful in neuroscience, in their underlying mechanism  there is still much to be understood. The novel structure of these equations brings unfamiliar challenges as well as intriguing phenomena, which are of both mathematical and scientific interest, and have attracted many mathematicians for diversified studies.

	For example, blow-up of the firing rate which relates to the multi-firing event, has been studied from a PDE point of view \cite{caceres2011analysis,carrillo2013CPDEclassical,Antonio_Carrillo_2015,roux2021towards} and from a SDE point of view \cite{delarue2015particle,delarue2015global,caceres2020understanding,hambly2019mckean}. Another fascinating phenomenon are  periodic solutions, which reflect forced or self-sustained oscillations in neuron networks. Such oscillations widely appear and play crucial roles in many biological functions such as rhythmogenesis \cite{gray1994synchronous,bianchi1995central,blankenship2010mechanisms}. In the population density description, periodic solutions have been studied in  time-elapsed  models \cite{pakdaman2013relaxation},   mean-field SDE models \cite{cormier2021hopf},   time-delayed integrate and fire models  \cite{ikeda2021theoretical}, etc. However, a general analysis framework is far from complete for such models, since the unique equation structure as well as the particular form of nonlinearity needs special investigation.  Even on the convergence to a steady state, which appears to be a simpler issue,  often only the weak interaction case can be treated \cite{roux2021towards,cormier2020long,Antonio_Carrillo_2015}, with a smallness assumption on the nonlinearity.
	
	In this work, we focus on a voltage-conductance kinetic neuron model proposed in \cite{cai2004effective,cai2006kinetic}. In this model, a neuron is characterized by two variables, its voltage $v$ and conductance $g$. The ensemble of neurons are described by $p(t,v,g)$, a probability density function at time $t$ of finding a neuron with voltage $v$ and conductance $g$. $p(t,v,g)$ satisfies a  nonlinear PDE, which is given by
	\begin{equation}\label{eq:original}
	    \p_tp+\p_v(J_v(v,g)p)=\p_g((g-g_{\myin}(t))p)+a(t)\p_{gg}p,\quad t>0,v\in(0,V_F),g>0.
	\end{equation} Here $J_v(v,g)$ denotes the velocity field in $v$ direction, which depends on both voltage $v$ and conductance $g$ as
	\begin{equation}\label{original-velocity-v}
	    J_v(v,g)=-g_Lv+g(V_E-v),
	\end{equation}where 
	\begin{equation}
	    g_L>0,\quad V_E>V_F>0.
	\end{equation} 
	The velocity field in $v$ \eqref{original-velocity-v} consists of two terms. The first term $-g_Lv$ models the leaky effect which derives a neuron to the resting potential $V_R$, set to be $0$ here. And $g_L>0$ is called the leaky conductance. The second term $g(V_E-v)$, derives the voltage to the firing potential $V_F$, and the strength of this velocity field is given by the  conductance variable $g$ and $V_E$ is referred to as the excitatory reversal potential. 
	
	A unique mechanism of a typical neuron is that it spikes when its voltage arrives at the threshold $V_F$, which has two consequences. First, after the spike, the voltage of the spiking neuron is reset to a lower potential, which is also set to be $0$, i.e., equal to the resting potential $V_R$. Note that a neuron can spike only if its conductance $g$ satisfy $J_v(V_F,g)>0$ or equivalently $g>\frac{V_F}{V_E-V_F}=:g_F$. In this case, since the spiking neurons are instantaneously reset at $v=0$ the following boundary condition for $v$ is imposed, matching the flux at $v=0$ and $v=V_F$,
	\begin{equation}\label{original-bc-v-01}
	    J_v(0,g)p(t,0,g)=J_v(V_F,g)p(t,V_F,g),\quad g>g_F,t>0.
	\end{equation}
	While if $0<g\leq g_F$, then $J_v(0,g)>0$ and $J_v(V_F,g)\leq 0$ no neuron can spike therefore no neuron is reset at $v=0$, which corresponds to the following Dirichlet boundary condition
	\begin{equation}\label{original-bc-v-02}
	     p(t,0,g)=p(t,V_F,g)=0,\quad 0<g\leq g_F,t>0.
	\end{equation} Thus we observe that the flux equality $J_v(0,g)p(t,0,g)=J_v(V_F,g)p(t,V_F,g)$ holds for all $g>0$.
	
	The second consequence is that a spiking event of one neuron influences other neurons, which make all neurons coupled in an ensemble. At this macroscopic density description, such influence is measured by the firing rate $N(t)$, the number of spikes per unit time, given by the following boundary flux at $v=V_F$ 
    \begin{equation}\label{original-firingrate}
        N(t)=\int_0^{+\infty}J_v(V_F,g)\,p(t,V_F,g)\,dg=\int_0^{+\infty}\bigl(-g_LV_F+g(V_E-V_F)\bigr)\,p(t,V_F,g)\,dg.
    \end{equation}
    
    The firing rate $N(t)$ is the source of the nonlinearity in this equation, because the coefficients $g_{\myin}(t)$ and $a(t)$ depend on $N(t)$ as
    \begin{equation}\label{original-gin-a}
        g_{\myin}(t)=g_0+g_1N(t),\quad a(t)=a_0+a_1N(t).
    \end{equation}
    Here, we  assume the parameters satisfy
    \begin{equation}
        g_0,g_1>0,\quad a_0,\,a_1>0,
    \end{equation} and  the interested readers may refer to \cite{cai2006kinetic,perthame2013voltage} for physical expressions and interpretations on these parameters. Given $g_{\myin}(t)$ and $a(t)$, in $g$ direction the equation looks like a Fokker-Planck equation for the Ornstein-Uhlenbeck process with the following no-flux boundary condition at $g=0$
\begin{equation}\label{original-bc-g}
    (g-g_{\myin}(t))p+a(t)\p_gp=0,\quad g=0,\,v\in(0,V_F),\,t>0.
\end{equation}  This system is complemented with an initial data which is a probability density as follows,
\begin{equation}
    p(0,v,g)=p_{\init}(v,g),\quad p_{\init}(v,g)\geq 0,\quad \int_{0}^{V_F}\int_0^{+\infty}p_{\init}(v,g)\,dvdg=1.
\end{equation}

This voltage-conductance model \eqref{eq:original} resembles a classical kinetic Fokker-Planck equation. In particular, the voltage $v$ is reminiscent of the position variable, and the conductance $g$ is reminiscent of the velocity variable. For one thing, the ``velocity field'' in $v$ \eqref{original-velocity-v} is influenced by $g$. For another, the spike of other neurons first influences $g$ and then affects $v$ indirectly through $J_v(v,g)$, which is like that the force first influences the velocity through Newton's second law, thereby influences the position. Besides, equation \eqref{eq:original} also has the so-called ``hypoelliptic'' structure shared by kinetic equations, as is pointed out in \cite{perthame2013voltage}. In fact, the diffusion is in $g$ direction only, and to gain dissipation as well as regularity in $v$, one needs to exploit the interplay between the transport in $v$ and the diffusion in $g$, which resembles the hypocoercivity for the kinetic Fokker-Planck equations \cite{villani2009hypocoercivity}. However, despite these similarities, the unique structure  of equation \eqref{eq:original} results in distinct solution structures from other kinetic models and correspondingly difficulties in analysis.
    
Remarkably, periodic solutions of \eqref{eq:original} have been numerically observed in \cite{caceres2011numerical}. However, the understanding on such periodic solutions has been very limited. The first theoretical analysis on \eqref{eq:original} is the pioneering work \cite{perthame2013voltage}, in which the authors analyze the steady states as well as deriving several global bounds for the dynamical problem. On the long time behavior, they derive bounds for the firing rate $N(t)$, which excludes the finite time blow-up and study the convergence to the steady state in the linear case, which is recently improved to a stronger sense \cite{dou:hal-03586715}. However, how a periodic solution arises in \eqref{eq:original} is still not understood. Our knowledge on long time behavior of \eqref{eq:original} is also very limited. Even the convergence to equilibrium when the nonlinearity is weak has not been proved yet. Besides its similarities to the kinetic Fokker-Planck equation, \eqref{eq:original} has its unique difficulties.

The first difficulty is that the firing rate $N(t)$ depends only on the integral in $g$ of the flux at  $v=V_F$ as in \eqref{original-firingrate}. This difficulty is essential, since it originates from the model assumption that neurons fire at a deterministic threshold $V_F$. This singular dependence gives mathematical challenges which motivates modified models. Such a firing mechanism is relaxed by a random discharge rule in \cite{perthame2018derivation,kim2021fast} with the aim to derive and study the macroscopic models. The effect of this relaxation has been studied in \cite{mna:8775} for a different but related model. 

The second difficulty stems from the velocity field in $v$, i.e., $J_v(v,g)=-g_Lv+g(V_E-v)$ as in \eqref{original-velocity-v}. The velocity field can not be written in a separable form $J_v(v,g)=f(v)h(g)$, in contrast to the velocity field for the position variable in classical kinetic models. This non-separable velocity field prohibits the use of many analysis tools. Moreover, it leads to the complicated boundary condition \eqref{original-bc-v-01} and \eqref{original-bc-v-02}, which makes it difficult to analyze even in the steady state case \cite{dou:hal-03586715}. In fact, it is not clear prior to this work whether the non-separable velocity field plays an essential role in the dynamics. 
    
The third difficulty is that $g$ lies in $\mathbb{R}^+$ instead of $\mathbb{R}$, which gives a time-dependent boundary condition \eqref{original-bc-g}. This difficulty is more like a technical one, since typically the density near zero is negligible as in numerical simulations \cite{caceres2011numerical}.

Our motivation is to understand self-organized oscillations in \eqref{eq:original}. Therefore, we make two simplifications on \eqref{eq:original} to isolate the effect of each of these difficulties.  First, to understand the consequences of a non-separable velocity field from a complementary side, we consider a simpler velocity field in $v$ neglecting the leaky mechanism, $J_v=g(V_E-v)$, which is further reduced to  $J_v=g$ for the sake of simplicity (see Section 2.5). Second, we extend the domain of $g$ from $\mathbb{R}^+$ to $\mathbb{R}$. Thus, the simplifications are designed to tackle the second and the third difficulties, while the issues resulting from the firing rate remain unaltered. 
   
    For the simplified voltage-conductance model to be presented in Section 2, we are able to show a clear characterization on its long time behavior. When $0<g_1/V_F<1$, the density function converges to the unique global equilibrium, while when $g_1/V_F\geq 1$, the firing rate $N(t)$ diverges to infinity as time $t$ goes to infinity. In fact, as $g_1/V_F$ approaches $1^{-}$, the steady state moves towards $g=+\infty$ and thus loses tightness. In other words, the magnitude of the parameter $g_1$ with respect to that of $V_F$ determines whether the positive feedback associated with the firing rate is weak or strong, and the precise description of these results are given in Theorem \ref{thm:full-convergence}. Our proof is based on a series of model reductions which can be rigorously justified. The full simplified voltage-conductance model is shown to asymptotically approach its $v$-homogeneous reduced model, and the dynamics of the $v$-homogeneous problem is approximately dominated by its Gaussian solution. Surprisingly, the mean and variance of the Gaussian solution satisfy a closed ODE system, and the analysis for such special solutions lays the cornerstone for analyzing the full problem.   
    
    Moreover, we introduce a  parameter $\veps>0$, which is the timescale ratio between the conductance $g$ and the voltage $v$. By considering the fast conductance limit $\veps\rightarrow0^+$, we derive a limit model governing the $v$ marginal density, whose long time behavior can also be characterized clearly in Theorem \ref{thm:longtime-fcl}. Depending on $g_1$ and the $L^{\infty}$ norm of initial data, the solution either blows up in finite time or globally exists in the form of periodic solutions. 

    To summarize, in this work, we have fully clarified the long time behavior of a simplified voltage-conductance  model and its fast conductance limit. We only find  periodic solutions in the fast conductance limit showing that the other neglected difficulties might be relevant for the appearance of periodic solutions in the full model \eqref{eq:original}. 
    
    Let us elaborate more on this last point. In the limit model the profile of the periodic solution is totally determined by the initial data -- there is no limit-cycle. Moreover, in our simplified model, there is no periodic solution for $\veps>0$, which is in contrast to the numerical observations in \cite{caceres2011numerical} for the original model \eqref{eq:original}. Therefore, there must be some nontrivial changes in the two simplifications. Since extending the domain of $g$ from $\mathbb{R}^+$  is clearly technical, an important implication of our analysis is, that the parameter $g_L>0$ plays a crucial role for the emergence of periodic solutions in the original model \eqref{eq:original}. 
    Our analysis indicates that the non-separable velocity field not only bring challenges at a technical level
    but also makes an essential contribution to more complicated dynamics.  Whereas, the role of the leak conductance $g_L>0$ shall not be further explored in this work.

Finally, we remark that the model reduction procedure for the simplified voltage-conductance model \eqref{eq:nonlinear-toy}, by which a kinetic equation is reduced to its Gaussian solutions, do not rely on the specific form of $g_{\myin}(t)$ and $a(t)$ in \eqref{eq:para-general}. Such a model reduction strategy can be used to analyze more general cases, such as nonlinear dependence on the firing rate $N(t)$ in  $g_{\myin}$ and $a$, or incorporating time-delay effects.

    The rest of this paper is arranged as follows. In Section \ref{sec:2} we present the simplified model, its model reduction framework and summarize its long time behavior, but the proofs of the long time behavior results are given in Section \ref{sc:3longtime}. The  fast conductance limit is derived and  studied in Section \ref{sc:4fcl}.

    \section{Simplified models and long time behavior}\label{sec:2}

In this section, we elaborate on the derivation of several simplified models with the objective of understanding their long time behaviors. Our main goal is to find conditions under which periodic solutions may appear. We start from a direct simplification from \eqref{eq:original}, and then consider two further reduced models. 

\subsection{A simplified  voltage-conductance model}
On the original model \eqref{eq:original}, we make the following two simplifications, as discussed in the introduction,
\begin{enumerate}
	\item We assume that the velocity field in voltage $v$ is given by $J_v(v,g)=g$ instead of $J_v(v,g)=-g_Lv+g(V_E-v)$.
	\item We extend the domain of conductance $g$ from $(0,\infty)$ to $(-\infty,+\infty)$.
\end{enumerate}
Then we get the following equation:
\begin{equation}\label{eq:nonlinear-toy}
\p_t p+g\p_v p+\p_g\bigl((g_{\text{in}}(t)-g)p-a(t)\p_gp\bigr)=0,\quad v\in(0,V_F),\ g\in\mathbb{R},\ t>0,
\end{equation}
and the boundary condition in voltage $v$ is simplified to 
\begin{equation}\label{nonlinear-toy-bc}
p(t,V_F,g)-p(t,0,g)=0,\ \forall t>0,g\in\mathbb{R}.
\end{equation}
For the conductance $g$, we no longer need a boundary condition at $g=0$ thanks to the extension of the domain.
For the initial data, we assume it is a probability density function on $(0,V_F)\times(-\infty,+\infty)$, denoted as $p_{\text{init}}$, so
\begin{equation}\label{nonlinear-toy-IC}
    p(0,v,g)=p_{\text{init}}(v,g),\quad v\in(0,V_F),g\in\mathbb{R}.
\end{equation}
The expressions of $g_{\text{in}}(t)$ and $a(t)$ depend on the firing rate $N(t)$ in the same way as in the original model given by
	\begin{equation}\label{eq:para-general}
	g_{\text{in}}(t)=g_0+g_1N(t),\ a(t)=a_0+a_1N(t),\quad g_0,g_1,a_0,a_1>0.
	\end{equation}
We define the firing rate $N(t)$ in a similar manner as the integration of flux at voltage $v=V_F$, over $g>0$, i.e.,
	\begin{equation}\label{eq:nonliner-toy-fire}
	N(t):=\int_{0}^{\infty}gp(t,V_F,g)dg.
	\end{equation}

	Though simplified, this model \eqref{eq:nonlinear-toy} inherits the nonlinear mechanism of the original model \eqref{eq:original} -- dependence of $g_{\text{in}}(t),a(t)$ on  $N(t)$. And the firing rate $N(t)$ still depends on the flux at one voltage value $v=V_F$ only.
	
    Our two simplifications tackle two difficulties of \eqref{eq:original}, as discussed in the introduction.
    
    Regarding the first simplification, our analysis applies to more general velocity fields $J_v(v,g)=gf(v)$, as long as $f(v)$ is a continuous and positive function on $[0,V_F]$. In fact, by a change of variable given in Section \ref{sc:gen-vel}, we reduce the more general case to the case $J_v(v,g)=g$. The essential difference is that in the original model \eqref{eq:original}, the velocity field $J_v(v,g)=-g_Lv+g(V_E-v)$ \eqref{original-velocity-v} is not in a separable form $h(g)f(v)$, which brings difficulties. However, prior to this work it is unclear whether such a non-separable velocity field also make an essential contribution to the dynamics. In our simplified model, the separable velocity field $J_v(v,g)=gf(v)$ makes a simpler boundary condition in $v$ \eqref{nonlinear-toy-bc}, and allows a separation of variables between $v$ and $g$.
    
	Our second modification, extending the domain of conductance $g$ to $\mathbb{R}$, is a technical one.  It is numerically observed that these neural networks maintain high conductance values away from zero, see \cite{caceres2011numerical}. To be specific, in $g$ direction, the typical profile of solutions is a Gaussian whose center is roughly $g_{\text{in}}(t)$, which is positive and away from $0$. This is the case both in the original model \eqref{eq:original} and our simplified model \eqref{eq:nonlinear-toy}. For such typical profiles, both the density at zero in the original model, and the density for $g\leq 0$ in our simplified model are small. Therefore intuitively this extension from $\mathbb{R}^+$ to $\mathbb{R}$ should not make much difference to the behavior of solutions. Moreover, the no-flux boundary condition \eqref{original-bc-g} imposed in the original model  brings many technical difficulties, and in the present simplification they are avoided.
	
	Throughout this work we assume the initial data $p_{\init}$ is a probability density function \eqref{nonlinear-toy-IC} and $g_{\myin}(t),a(t)$ depends on $N(t)$ in a linear form with $g_0,g_1,a_0,a_1>0$ as in \eqref{eq:para-general}.
\subsection{Asymptotic simplification towards a $v$-homogeneous regime}

Now we analyze the simplified model \eqref{eq:nonlinear-toy}. Let us consider the Fourier expansion in $v$ direction
	\begin{equation}\label{fourier-v}
\begin{aligned}
 p(t,v,g)&=\frac{1}{V_F}\sum_{k=-\infty}^{+\infty}p_k(t,g)e^{ikv\frac{2\pi}{V_F}},\quad \text{where}\\
p_k(t,g):&=\int_{0}^{V_F}p(t,v,g)e^{-ikv\frac{2\pi}{V_F}}dv,\quad k\in\mathbb{Z}.
\end{aligned}
	\end{equation}Plugging the expansion \eqref{fourier-v} in \eqref{eq:nonlinear-toy}, we get that each $p_k$ solves the following equation,
	\begin{equation}\label{eq:fourier-each-k}
	\p_t p_k+ik\frac{2\pi}{V_F}gp_k=\p_g\bigl((-g_{\text{in}}(t)+g)p_k+a(t)\p_gp_k\bigr),\quad g\in\mathbb{R},\ t>0.
	\end{equation}
	Thanks to the simplification of the velocity field in $v$, the Fourier modes are ``separated'': In the equation for the $k$-th mode $p_k$ \eqref{eq:fourier-each-k}, only $p_k$ itself shows up explicitly, and different modes are coupled implicitly through the dependence of $g_{\text{in}}(t),a(t)$ on $N(t)$. 
	
	Our first result on the long time behavior is that inhomogeneous modes in $v$ , i.e., $p_k$ with $k\neq0$ in the expansion \eqref{fourier-v}, diminish exponentially. To show this, we first derive semi-explicit formulas for \eqref{eq:fourier-each-k}, which reveals the solution structure.

	Let us consider the Fourier transform in $g$: $\hat{p}_k(t,\xi):=\frac{1}{\sqrt{2\pi}}\int_{-\infty}^{+\infty}e^{-ig\xi}p_k(t,g)dg$. In terms of the Fourier transform $\hat{p}_k$, \eqref{eq:fourier-each-k} becomes
		\begin{align*}
\p_t\hat{p}_k-k\pvf \p_{\xi}\hat{p}_k=-ig_{\myin}(t) {\xi}\hat{p}_k-\xi\p_{\xi}\hat{p}_k-a(t)\xi^2\hat{p}_k,
		\end{align*} which simplifies to 
		\begin{equation}\label{eq:mode-k-lambda-four}
\p_t\hat{p}_k+[\xi-k\pvf] \p_{\xi}\hat{p}_k=-(a(t)\xi^2+ig_{\myin}(t)\xi)\hat{p}_k,
		\end{equation}
		which is a first order equation in $\xi$ with the following characteristic 
		\begin{equation}\label{eq:char}
\frac{d\xi_k(t)}{dt}=\xi_k(t)-k\pvf,
		\end{equation}\text{whose solution is } 
		\begin{equation}\label{eq:char-2}
		    [\xi_k(t)-k\pvf]=e^{t-s}[\xi_k(s)-k\pvf]
		\end{equation}
		Solving the equation \eqref{eq:mode-k-lambda-four} along its characteristic \eqref{eq:char}, we get
		\begin{equation}\label{solu-fourier}
		    \hat{p}_k(t,\xi_k(t))=\hat{p}_k(0,\xi_k(0))\exp\left(-\int_0^t[a(s)\xi_k(s)^2+ig_{\myin}(s)\xi_k(s)]ds\right).
		\end{equation}
		We can intuitively see why nonzero modes $p_k\,(k\neq0)$ decays from \eqref{solu-fourier}, by looking at the magnitude of $\hat{p}$ at $\xi=k\frac{2\pi}{V_F}$, which is the fixed point of characteristics of \eqref{eq:char},
		\begin{equation*}
		   |\hat{p}_k(t,k\frac{2\pi}{V_F})|=|\hat{p}_k(0,k\frac{2\pi}{V_F})|\left|\exp\left(-(k\frac{2\pi}{V_F})^2\int_0^ta(s)ds\right)\right|.
		\end{equation*}
		When $k=0$, the dissipation degenerates at $\xi=0$ and $|\hat{p}_0(t,0)|$ does not change. While in the case $k\neq0$, $|\hat{p}_k(t,k\frac{2\pi}{V_F})|$ decays exponentially with a rate at least $a_0(k\frac{2\pi}{V_F})^2>0$, since from \eqref{eq:para-general} we have $a(s)\geq a_0>0$. 
		
		Applying the inverse transform on \eqref{solu-fourier}, we can derive the solution formulas in Lemma \ref{lemma-explicit-solu} below. The detailed calculation is given in Appendix \ref{app:deri}. 
\begin{lemma}\label{lemma-explicit-solu}
The solution $p_k(t,g)$ of the equation \eqref{eq:fourier-each-k}, the $k$-th mode of the Fourier expansion in $v$ \eqref{fourier-v} for \eqref{eq:nonlinear-toy}, is given by
    \begin{equation}\label{explicit——k}
	p_k(t,g)=e^{ik(\frac{2\pi}{V_F})g}(p_{t,k}\ast G_{t,k})(g).
	\end{equation} 
	Here $p_{t,k}$ is a shrinkage of $p_{0,k}$, which is the initial data for the $k$-th Fourier mode multiplied a shift in frequency:
	\begin{equation}\label{shrinkage}
	p_{t,k}(y):=e^{t}p_{0,k}(e^ty),\quad p_{0,k}(g):=e^{-ik\frac{2\pi}{V_F}g}\int_{0}^{V_F}p_{\init}(v,g)e^{-ikv\frac{2\pi}{V_F}}dv.
	\end{equation} 
	And $G_{t,k}$ is a modified Gaussian with a phase factor and a decay factor, given by
	\begin{equation}\label{eq:modified-Gaussian}
	G_{t,k}(z)=\frac{1}{\sqrt{2\pi C(t)}}\exp\left(-\frac{(z-B(t))^2}{2C(t)}\right)\exp\left(ik\frac{2\pi}{V_F}\myBtwo(t,z)\right)\exp\left(-k^2(\frac{2\pi}{V_F})^2D(t)\right).
	\end{equation}Here the mean $B(t)$ and the variance $C(t)$ are given by:
	\begin{equation}\label{eq:BC}
	\begin{aligned}
		B(t)=\int_{0}^{t}e^{-(t-s)}g_{\text{in}}(s)ds=\int_{0}^{t}e^{-(t-s)}(g_0+g_1N(s))ds.\\
	C(t)=2\int_{0}^{t}e^{-2(t-s)}a(s)ds=2\int_{0}^{t}e^{-2(t-s)}(a_0+a_1N(s))ds.
	\end{aligned}
	\end{equation}
	Moreover $\myBtwo(t,z)$ and $D(t)$ are given by
	\begin{equation}\label{def-B2}
	\myBtwo(t,z)=-(z-B(t))\frac{\int_{0}^{t}e^{s-t}a(s)ds}{\int_{0}^{t}e^{2(s-t)}a(s)ds}-\int_{0}^{t}g_{\text{in}}(s)ds,
	\end{equation} and
	\begin{equation}\label{decay-D(t)}
	D(t)=\int_{0}^{t}a(s)ds-\frac{(\int_{0}^{t}e^{s-t}a(s)ds)^2}{\int_{0}^{t}e^{2(s-t)}a(s)ds}\geq 0.
	\end{equation}
\end{lemma}
One can see $D(t)\geq 0$, by Cauchy-Schwartz inequality $(\int_{0}^{t}a(s)ds)(\int_{0}^{t}e^{2(s-t)}a(s)ds)\geq(\int_{0}^{t}e^{s-t}a(s)ds)^2$. Recall \eqref{eq:para-general} $a(s)=a_0+a_1N(s)>0$, since $a_0,a_1>0$ and $N(s)\geq 0$.

	We remark that here $B,C,\myBtwo,D$ are independent of $k$, therefore $G_{t,k}$ depends on $k$ only through two explicit coefficients before $\myBtwo,D$. 

	We denote the first two terms in \eqref{eq:modified-Gaussian} as $\bar{G}_{t,k}$, whose $L^1$ norm is $1$,
	\begin{equation}\label{def-barG}
	\bar{G}_{t,k}(z):=\frac{1}{\sqrt{2\pi C(t)}}\exp\left(-\frac{(z-B(t))^2}{2C(t)}+ik\frac{2\pi}{V_F}\myBtwo(t,z)\right),\quad \|\bar{G}_{t,k}(z)\|_{L^1(\mathbb{R})}=1.
	\end{equation}
	Then combing \eqref{fourier-v} and \eqref{explicit——k}, we get a formula for the solution of \eqref{eq:nonlinear-toy}
	\begin{equation}\label{formula-v,g}
	p(t,v,g)=\frac{1}{V_F}\sum_{k=-\infty}^{+\infty}\exp\left(-k^2(\frac{2\pi}{V_F})^2D(t)+i\frac{2\pi}{V_F}kv+i\frac{2\pi}{V_F}kg\right)(p_{t,k}*\bar{G}_{t,k})(g).
	\end{equation}
	 Decay of non-zero modes in $v$ emerges from \eqref{formula-v,g}. The $k$-th mode decays with the factor $\exp\left(-k^2(\frac{2\pi}{V_F})^2D(t)\right)$, which comes from the diffusion in $g$ direction. Through the transport term $g\p_vp$, diffusion in $g$ direction is ``passed to'' $v$ direction, although there is no explicit diffusion in $v$. In literature such effect is called hypoellipticity, in view of the regularizing effect, or hypocoercivity, in view of the convergence to a steady state \cite{villani2009hypocoercivity}. We can also rewrite the formula \eqref{formula-v,g} as the shrinkage of initial data in $g$ direction convoluting a ``Green function'':
	\begin{equation}
	    p(t,v,g)=\int_0^{V_F}\int_{-\infty}^{\infty}e^tp_{\text{init}}(t,\tilde{v},e^t\tilde{g})\left[\sum_{k=-\infty}^{+\infty}\frac{1}{V_F}e^{-k^2(\frac{2\pi}{V_F})^2D(t)+ik\frac{2\pi}{V_F}(v-\tilde{v})+ik\frac{2\pi}{V_F}(g-e^t\tilde{g})}\bar{G}_{t,k}(g-\tilde{g})\right]d\tilde{v}d\tilde{g}.
	\end{equation}

To gain exponential decay, we need a further estimate on $D(t)$ in the following Lemma. 
	\begin{lemma}\label{lemma-D(t)}
	For $D(t)$ defined in \eqref{decay-D(t)}, the following lower bound holds
	\begin{equation}\label{eq:D(t)lower}
	    D(t)\geq a_0(t-2\frac{e^t-1}{e^t+1}),\quad t\geq0.
	\end{equation}
	\end{lemma}
	\begin{proof}[Proof of Lemma \ref{lemma-D(t)}]
		 Recalling $a(s)=a_0+a_1N(s)$ in \eqref{eq:para-general}, we rewrite $D(t)$ in \eqref{decay-D(t)} as
	\begin{align}\label{tmp-Dt}
		D(t)&=\int_{0}^{t}[a_0+a_1N(s)]ds-\frac{(\int_{0}^{t}e^{s-t}[a_0+a_1N(s)]ds)^2}{\int_{0}^{t}e^{2(s-t)}[a_0+a_1N(s)]ds}.
	\end{align}
	Note that $N(s)\geq0$, and let us consider an extreme case first, if $\int_{0}^{t}e^{2(s-t)}a_1N(s)ds=0$ then we also have $\int_0^tN(s)ds=\int_{0}^{t}e^{(s-t)}a_1N(s)ds=0$. In this case, the desired inequality becomes an equality:
	\begin{equation*}
	    	D(t)=\int_{0}^{t}a_0ds-\frac{(\int_{0}^{t}e^{s-t}a_0ds)^2}{\int_{0}^{t}e^{2(s-t)}a_0ds}=a_0(t-2\frac{e^t-1}{e^t+1}).
	\end{equation*}
	For the general case $\int_{0}^{t}e^{2(s-t)}a_1N(s)ds>0$. We shall use the following elementary inequality:	$\frac{b_1^2}{c_1}+\frac{\myBtwo^2}{c_2}\geq \frac{(b_1+\myBtwo)^2}{c_1+c_2}$ for $b_1,\myBtwo\in\mathbb{R},c_1,c_2>0$, which is a consequence of using the Cauchy-Schwartz inequality on $\mathbb{R}^2$ for vectors $(\sqrt{c_1},\sqrt{c_2})$ and $(\frac{b_1}{\sqrt{c_1}},\frac{\myBtwo}{\sqrt{c_2}})$, to get $(c_1+c_2)(\frac{b_1^2}{c_1}+\frac{\myBtwo^2}{c_2})\geq(b_1+\myBtwo)^2$. Taking $b_1=\int_{0}^{t}e^{s-t}a_0ds,\myBtwo=\int_{0}^{t}e^{s-t}a_1N(s)ds$, $c_1=\int_{0}^{t}e^{2(s-t)}a_0ds$ and $c_2=\int_{0}^{t}e^{2(s-t)}a_1N(s)ds$, we apply the inequality to \eqref{tmp-Dt} and derive
	\begin{align*}
		D(t)\geq \int_{0}^{t}[a_0+a_1N(s)]ds-\frac{(\int_{0}^{t}e^{s-t}a_0ds)^2}{\int_{0}^{t}e^{2(s-t)}a_0ds}-\frac{(\int_{0}^{t}e^{s-t}a_1N(s)ds)^2}{\int_{0}^{t}e^{2(s-t)}a_1N(s)ds}.
	\end{align*}
	Therefore we have
	\begin{align}\notag
D(t)&\geq \left[\int_{0}^{t}a_0ds-\frac{(\int_{0}^{t}e^{s-t}a_0ds)^2}{\int_{0}^{t}e^{2(s-t)}a_0ds}\right]+\left[\int_{0}^{t}a_1N(s)ds-\frac{(\int_{0}^{t}e^{s-t}a_1N(s)ds)^2}{\int_{0}^{t}e^{2(s-t)}a_1N(s)ds}\right]\\&\geq \left[\int_{0}^{t}a_0ds-\frac{(\int_{0}^{t}e^{s-t}a_0ds)^2}{\int_{0}^{t}e^{2(s-t)}a_0ds}\right]=a_0(t-2\frac{e^t-1}{e^t+1}).\label{Dt_lower_bound}
	\end{align} In the second inequality above, we use the Cauchy-Schwartz inequality again to obtain $[\int_{0}^{t}a_1N(s)ds-\frac{(\int_{0}^{t}e^{s-t}a_1N(s)ds)^2}{\int_{0}^{t}e^{2(s-t)}a_1N(s)ds}]\geq 0$. The proof is completed.
	\end{proof}
Now we prove our first result on the long time behavior, the asymptotic simplification towards the $v$-homogeneous problem for \eqref{eq:nonlinear-toy}. Precisely, the solution of \eqref{eq:nonlinear-toy} converges to its zeroth mode in the Fourier expansion in $v$ \eqref{fourier-v} as in the following theorem.
	\begin{theorem}\label{thm:decay-v}
		Suppose $p(t,v,g)$ is a solution of \eqref{eq:nonlinear-toy} and $p_0(t,g)$ is the zeroth mode in $v$ defined in \eqref{fourier-v}. Then $p(t,v,g)$ converges to the $v$-homogeneous mode $\frac{1}{V_F}p_0(t,g)$ exponentially. Precisely, we have
		\begin{equation}
		\begin{aligned}
		\|p(t,v,g)-\frac{1}{V_F}p_0(t,g)\|_{L^1((0,V_F)\times\mathbb{R})}&\leq 2\frac{e^{-\pvf^2d(t)}}{1-e^{-\pvf^2d(t)}},
		\end{aligned}
		\end{equation}
		where $d(t)$ is given by 
		\begin{equation}\label{def-dt}
		    d(t):=a_0(t-2\frac{e^t-1}{e^t+1})> 0,\quad t>0.	
		\end{equation}
	\end{theorem}
	Note that, defined in \eqref{def-dt}, $d(t)>0$ when $t>0$ and the leading order is $a_0t$ when $t$ goes to infinity, which implies the exponential convergence. While for $t$ goes to $0^+$, $d(t)=\frac{1}{12}a_0t^3+O(t^4)$.

\begin{proof}[Proof of Theorem \ref{thm:decay-v}]We work with the solution formula \eqref{formula-v,g}. The idea is just to utilize the decay from $e^{-k^2(\frac{2\pi}{V_F})^2D(t)}$ with a uniform bound on $e^{i\frac{2\pi}{V_F}kv}e^{i\frac{2\pi}{V_F}kg}(p_{t,k}*\bar{G}_{t,k})(g)$. First we estimate the $L^1$ norm of each $k$-th order mode $\frac{1}{V_F}e^{ik\frac{2\pi}{V_F}v}p_k(t,g)$,
	\begin{align*}
\|\frac{1}{V_F}e^{ik\frac{2\pi}{V_F}v}p_k(t,g)\|_{L^1((0,V_F)\times\mathbb{R})}&\leq \int_{0}^{V_F}\frac{1}{V_F}|e^{ik\frac{2\pi}{V_F}v}|dv\int_{\mathbb{R}}|p_k(t,g)|dg\\&=\|p_k(t,g)\|_{{L^1(\mathbb{R})}}.
	\end{align*}
	And by the formula for $p_k$ \eqref{explicit——k} and Young's convolution inequality, we obtain
	\begin{align*}
	\|p_k(t,g)\|_{{L^1(\mathbb{R})}}=\|e^{ik(\frac{2\pi}{V_F})g}(p_{t,k}*G_{t,k})(g)\|_{{L^1(\mathbb{R})}}&\leq \|p_{t,k}(g)\|_{{L^1(\mathbb{R})}}\|G_{t,k}\|_{{L^1(\mathbb{R})}}\\&=\|p_{0,k}(g)\|_{{L^1(\mathbb{R})}}e^{-k^2(\frac{2\pi}{V_F})^2D(t)}.
	\end{align*}
	In the last equality we use that 
	$p_{t,k}(y)=e^{t}p_{0,k}(e^ty)$ in \eqref{shrinkage}, and the formula for $G_{t,k}$\eqref{eq:modified-Gaussian}. Then by the second equation in \eqref{shrinkage} for $p_{0,k}(g)$, we get $$\|p_{0,k}(g)\|_{L^1(\mathbb{R})}\leq \int_{\mathbb{R}}|e^{-ik\frac{2\pi}{V_F}e^tg}|dg\int_{0}^{V_F}p_{\text{init}}(v,g)|e^{ikv\frac{2\pi}{V_F}}|dv=\int_{\mathbb{R}}\int_{0}^{V_F}p_{\text{init}}(v,g)dvdg=1,$$ since the initial data is a probability density function.
	Therefore, we have an estimate for each mode
	\begin{equation*}
\|\frac{1}{V_F}e^{ik\frac{2\pi}{V_F}v}p_k(t,g)\|_{L^1((0,V_F)\times\mathbb{R})}\leq e^{-k^2(\frac{2\pi}{V_F})^2D(t)}.
	\end{equation*}
	Then from the Fourier expansion in $v$ \eqref{fourier-v} we deduce
	\begin{align}\notag
	\|p(t,v,g)-\frac{1}{V_F}p_0(t,g)\|_{L^1((0,V_F)\times\mathbb{R})}&\leq \sum_{k\neq 0,k\in\mathbb{Z}}e^{-k^2(\frac{2\pi}{V_F})^2D(t)}= 2\sum_{k=1}^{+\infty}e^{-k(\frac{2\pi}{V_F})^2D(t)}\\&=2\frac{e^{-(\frac{2\pi}{V_F})^2D(t)}}{1-e^{-(\frac{2\pi}{V_F})^2D(t)}}.\label{tmp-pf1}
	\end{align}
	Finally by Lemma \ref{lemma-D(t)}, $D(t)\geq a_0(t-2\frac{e^t-1}{e^t+1})=d(t)$. With a direct calculation we get that $d(t)>0$ for $t>0$, and complete the proof.
\end{proof}

Theorem \ref{thm:decay-v} shows that, despite the nonlinearity in $g_{\myin}(t)$ and $a(t)$, non-homogeneous modes in $v$ always decay. This relies on the separation of variables between $v$ and $g$, which is a consequence from the simplification of the velocity field in $v$. Actually in the original model \eqref{eq:original}, it is difficult if not impossible to perform such a separation of variables, since the velocity field $J_v(v,g)=-g_Lv+g(V_E-v)$ \eqref{original-velocity-v} is not in a separable form $J_v(v,g)=h(g)f(v)$.

By Theorem \ref{thm:decay-v}, a solution of \eqref{eq:nonlinear-toy} tends to be homogeneous in $v$ as time evolves. Therefore if there is a periodic solution, the periodic dynamics should be in $g$ direction. This inspires us to consider a reduced model.

\subsection{The $v$-homogeneous problem and its Gaussian solutions}
By Theorem \ref{thm:decay-v}, for a solution of \eqref{eq:nonlinear-toy}, the $v$-homogeneous mode $\frac{1}{V_F}p_0(t,g)$ dominates in the long time. To further investigate \eqref{eq:nonlinear-toy}, we consider the case when the solution is exactly $v$-homogeneous, i.e., there is only the zeroth mode in $v$
\begin{equation}\label{v-homo-case}
p(t,v,g)=\frac{1}{V_F}p_0(t,g).
\end{equation}By the formula for each mode in Lemma \ref{lemma-explicit-solu}, a solution is $v$-homogeneous if its initial data is $v$-homogeneous. From Theorem \ref{thm:decay-v} we expect that the $v$-homogeneous case reflects the typical behavior of \eqref{eq:nonlinear-toy}.

 Now we investigate the behavior of $p_0$ in this case \eqref{v-homo-case}. As in the general case, $p_0$ satisfies the following PDE, which is the $k=0$ case in \eqref{eq:fourier-each-k}
	\begin{equation}\label{eq:reduce-p}
\p_t p_0=\p_g\bigl((-g_{\text{in}}(t)+g)p_0+a(t)\p_gp_0\bigr),\quad g\in\mathbb{R},\ t>0,
\end{equation} where $g_{\text{in}}(t)$ and $a(t)$ still given by \eqref{eq:para-general}
\begin{equation*}
	g_{\text{in}}(t)=g_0+g_1N(t),\ a(t)=a_0+a_1N(t),\quad g_0,g_1,a_0,a_1>0.
\end{equation*}
The difference is that in this $v$-homogeneous case, the firing rate $N(t)$ is totally determined by $p_0$, since there is no other mode,
\begin{equation}\label{eq:reduce-N}
N(t)=\int_{0}^{+\infty}gp(t,V_F,g)dg=\frac{1}{V_F}\int_{0}^{+\infty}gp_0(t,g)dg.
\end{equation} Therefore \eqref{eq:para-general}, \eqref{eq:reduce-p} and \eqref{eq:reduce-N} give a closed 1+1 dimensional PDE in $g$ direction. Here the firing rate is measured as a ``positive moment'' in $g$. And the nonlinearity comes from that the drift $g_{\text{in}}(t)$ and the diffusion coefficient $a(t)$ depends on the firing rate $N(t)$, as in the full model \eqref{eq:nonlinear-toy}.

Since we will investigate this $v$-homogeneous case as a reduced model from \eqref{eq:nonlinear-toy}, in the following with abuse of notation, we denote \begin{equation}
    p(t,g):=p_0(t,g).
\end{equation} And we denote the initial data as $p_{\text{init}}(g)$
\begin{equation}
    p(0,g)=p_{\text{init}}(g),\quad g\in\mathbb{R},
\end{equation}which is a probability density on $\mathbb{R}$. 

To make a further model reduction, we observe that the $v$-homogeneous system \eqref{eq:reduce-p} admits a Gaussian type special solution, whose mean and variance solve an ODE system. Based on this, we further reduce the PDE model in conductance only \eqref{eq:reduce-p} to an ODE system \eqref{eq:system-1}, defined in Proposition \ref{prop:Gaussian anstaz} below.
 \begin{proposition}\label{prop:Gaussian anstaz}
 	The $v$-homogeneous system \eqref{eq:reduce-p} admits the following Gaussian type solution
 	\begin{equation}\label{eq:Gaussian anstaz}
 	p(t,g)=\frac{1}{\sqrt{2\pi c(t)}}\exp\left(-\frac{(g-b(t))^2}{2c(t)}\right),
 	\end{equation}where the mean $b(t)$ and the variance $c(t)$ satisfy the following autonomous ODE
 	\begin{equation}\label{eq:system-1}
 	\begin{aligned}
 	\frac{db(t)}{dt}&=g_0+{g_1}N(b(t),c(t))-b(t),\\
 	\frac{dc(t)}{dt}&=2a_0+2{a_1}N(b(t),c(t))-2c(t).
 	\end{aligned}
 	\end{equation}
 	Here $N(b,c)$ is a function of $(b,c)$ which denotes the firing rate of such a Gaussian type solution: 
 	\begin{equation}\label{def-N-bc}
 	N(b,c):=\frac{1}{V_F}\int_{0}^{+\infty}g\frac{1}{\sqrt{2\pi c}}\exp\left(-\frac{(g-b)^2}{2c}\right)dg\geq 0,\quad b\in\mathbb{R},\,c>0.
 	\end{equation}
 \end{proposition}
\begin{proof}[Proof of Proposition \ref{prop:Gaussian anstaz}]
Let $\mathcal{G}(g;b,c)=\frac{1}{\sqrt{2\pi c}}e^{-\frac{(g-b)^2}{2c}}$ be the Gaussian with mean $b\in\mathbb{R}$ and variance $c>0$. Then by direct calculation one obtains
\begin{equation*}
\p_g\mathcal{G}=\frac{b-g}{c}\mathcal{G}=-\p_b\mathcal{G},\quad \p_{gg}\mathcal{G}=(\frac{(b-g)^2}{c^2}-\frac{1}{c})\mathcal{G}=2\p_c\mathcal{G}.
\end{equation*}
Therefore for the ansatz $p(t,g)=\mathcal{G}(g;b(t),c(t))$ defined in \eqref{eq:Gaussian anstaz} one writes
\begin{align*}
        (g-g_{\text{in}}(t))p+a(t)\p_gp&=(b(t)-g_{\text{in}}(t))p+(g-b(t))p+a(t)\p_gp\\\notag &=(b(t)-g_{\text{in}}(t))p-c(t)\p_gp+a(t)\p_gp.
\end{align*}
Hence, we have
\begin{align}\notag
    \p_g[(g-g_{\text{in}}(t))p+a(t)\p_gp] &=\p_g[(b(t)-g_{\text{in}}(t))p+(a(t)-c(t))\p_gp]\\&=(g_{\text{in}}(t)-b(t))\p_b\mathcal{G}+2(a(t)-c(t))\p_{c}\mathcal{G}.\label{tmp-Gauss}
\end{align}
On the other hand, by the chain rule we get
\begin{equation*}
    \p_tp=\frac{db(t)}{dt}\p_b\mathcal{G}+\frac{dc(t)}{dt}\p_c\mathcal{G}.
\end{equation*} Compare this with \eqref{tmp-Gauss}, we get that $p$ satisfies the equation \eqref{eq:reduce-p} if $\frac{db(t)}{dt}=g_{\text{in}}(t)-b(t)$ and $\frac{dc(t)}{dt}=2(a(t)-c(t))$. Recall \eqref{eq:para-general} $g_{\text{in}}=g_0+g_1N(t),a(t)=a_0+a_1N(t)$, and note that when $p(t,\cdot)$ is a Gaussian, the firing $N(t)$ can be expressed as a function of its mean and variance as in \eqref{def-N-bc}, we deduce that when $(b(t),c(t))$ solves the ODE \eqref{eq:system-1}, $p$ defined as \eqref{eq:Gaussian anstaz} is a solution of \eqref{eq:reduce-p}.
\end{proof}

In summary, we derive two reduced models from the 1+2 dimensional voltage-conductance PDE \eqref{eq:nonlinear-toy}. First we consider the $v$-homogeneous case of \eqref{eq:nonlinear-toy} and get a 1+1 dimensional PDE \eqref{eq:reduce-p}. Then from a Gaussian type special solution, we deduce an ODE system \eqref{eq:system-1}. As we will show through the long time behavior result in the next section, these reduced models indeed reflect typical behavior of \eqref{eq:nonlinear-toy}.

\subsection{Long time behavior: Main results}

Our main result is a thorough study of the long time behavior of these three models \eqref{eq:nonlinear-toy}, \eqref{eq:reduce-p} and \eqref{eq:system-1}. When $g_1/V_F\geq 1$, the firing $N(t)$ diverges to infinity as $t$ goes to infinity, while otherwise when $0<g_1/V_F<1$, the solution converges to the unique steady state.

The divergence to infinity of the firing rate $N(t)$ reflects a model for excitatory neuron networks with excessive feedback. In an excitatory network, one neuron's firing excites other neurons. Larger firing rate $N(t)$ results in a larger $g_{\text{in}}(t)$, which may in turns makes the firing rate larger. This is a positive feedback. And when ${g_1}/{V_F}\geq 1$, this positive feedback is so strong that $N(t)$ diverges to infinity as time evolves. At the level of the $v$-homogeneous model \eqref{eq:reduce-p}, we can prove the following proposition.
\begin{proposition}\label{prop:g1>=1}
When $g_1/V_F\geq 1$, for a solution $p(t,g)$ of the $v$-homogeneous problem \eqref{eq:reduce-p}, the firing rate diverges to infinity as time evolves, i.e.,
\begin{equation}
N(t)\rightarrow+\infty,\quad \text{as }t\rightarrow+\infty.
\end{equation}
\end{proposition}
\begin{proof}[Proof of Proposition \ref{prop:g1>=1}]
    Without loss of generality we only need to consider the case $V_F=1$, otherwise we use $g_1/V_F$ and $a_1/V_F$ instead of $g_1$,$a_1$. When $V_F=1$ we have $g_1\geq 1$.
    
	Let's consider the first moment $M_1(t):=\int_{-\infty}^{+\infty}gp(t,g)dg$, which is less than the firing rate $N(t)$: $$M_1(t)\leq\int_{0}^{+\infty}gp(t,g)dg=N(t).$$ Then it suffices to show that $M_1(t)$ goes to infinity as time evolves. We compute the time derivative of $M_1(t)$ and integrate by parts,
	\begin{align}\notag
	\frac{d}{dt}M_1(t)&=\int_{-\infty}^{+\infty}g\p_tp(t,g)dg=\int_{-\infty}^{+\infty}g[\p_g((-g_{\text{in}}(t)+g)p+a(t)\p_gp)]dg\\\notag&=\int_{-\infty}^{+\infty}[(g_{\text{in}}(t)-g)p-a(t)\p_gp]dg\\&=g_{\text{in}}(t)-M_1(t)=g_0+g_1N(t)-M_1(t)\geq g_0+(g_1-1)N(t).\label{ddtM1}
	\end{align}In the last inequality we use $N(t)\geq M_1(t)$. When $g_1\geq 1$, we deduce from \eqref{ddtM1}
	\begin{equation*}
\frac{d}{dt}M_1(t)\geq g_0>0.
	\end{equation*}Thus $M_1(t)$ goes to infinity as time evolves.
\end{proof}
Proposition \ref{prop:g1>=1} implies the same result for the ODE \eqref{eq:system-1}, since the ODE represents a special solution of \eqref{eq:reduce-p}. For the full model \eqref{eq:nonlinear-toy} the same result also holds  (see Theorem \ref{thm:full-convergence}), but is not straightforward, due to the singular dependence of $N(t)$, i.e., $N(t)$ depends only on the flux at one point $v=V_F$. 

When $0<g_1/V_F<1$, the density function converges to the unique steady state. This case is more involved and we take a bottom-up approach. First we characterize the long time behavior for the ODE system \eqref{eq:system-1}.
\begin{theorem}\label{thm:ode}
		When $0<g_1/V_F<1$, the ODE system \eqref{eq:system-1} has a unique steady state $(b^*,c^*)$, which is a global attractor. In other words, for any initial data $b(0)\in\mathbb{R},c(0)>0$, the solution $(b(t),c(t))$ converges to $(b^*,c^*)$ as time $t$ goes to infinity.
\end{theorem}
 Theorem \ref{thm:ode} lays the cornerstone for the analysis of the full simplified model. We naturally require $c(0)>0$, since $c(t)$ represents the variance of the Gaussian special solution of \eqref{eq:reduce-p}. To link with the case $g_1/V_F\geq 1$, we show that $b^*$ goes to infinity as $g_1/V_F$ approaches $1^-$ in Corollary \ref{cor:ode-steady}. This indicates that the steady state loses tightness towards $g=+\infty$ in such a limit, see Theorem \ref{thm:full-convergence} below.

Then, we show that a general solution of \eqref{eq:reduce-p} converges to a time-varying Gaussian and that it eventually converges to the unique steady state, which is the Gaussian corresponding to the steady state of the ODE \eqref{eq:system-1}.
\begin{theorem}\label{thm:v-homo-convergence}
When $0<g_1/V_F<1$, for initial data $p_{\init}(g)$ with finite order moment $\int_{\mathbb{R}}|g|p_{\init}(g)dg<\infty$, as time evolves the solution of the $v$-homogeneous problem \eqref{eq:reduce-p} converges to the unique steady state of \eqref{eq:reduce-p}
	\begin{equation}\label{steady-g}
	p^*(g):=\frac{1}{\sqrt{2\pi c^*}}\exp\left(-\frac{(g-b^*)^2}{2c^*}\right),
	\end{equation} in $L^1(\mathbb{R})$. Here $(b^*,c^*)$ is the unique steady state of \eqref{eq:system-1}.
\end{theorem}

With results on two reduced models, we manage to characterize the full model \eqref{eq:nonlinear-toy}.

\begin{theorem}\label{thm:full-convergence}
For initial data $p_{\init}(v,g)$ with $\int_{\mathbb{R}}\int_0^{V_F}|g|p_{\init}(v,g)dvdg<\infty$, we have the following result on the long time behavior of the simplified voltage-conductance  model \eqref{eq:nonlinear-toy}.
\begin{enumerate}
\item When $g_1/V_F\geq 1$, the firing rate $N(t)$ goes to infinity as time evolves, i.e.,
\begin{equation}
    N(t)\rightarrow+\infty,\quad \text{as }t\rightarrow+\infty.
\end{equation}
    \item 	When $0<g_1/V_F<1$, as time evolves the density function $p(t,v,g)$ converges to the unique steady state
	\begin{equation}\label{steady-vgPDE}
	p^*(v,g):=\frac{1}{V_F}\frac{1}{\sqrt{2\pi c^*}}\exp\left(-\frac{(g-b^*)^2}{2c^*}\right),
	\end{equation} in $L^1((0,V_F)\times\mathbb{R})$. Here $(b^*,c^*)$ is the unique steady state of the ODE system \eqref{eq:system-1}.
\item When $g_1/V_F\rightarrow 1^{-}$, we have $b^*\rightarrow+\infty$, therefore the steady state \eqref{steady-vgPDE} moves towards infinity and the system loses tightness.
\end{enumerate}

\end{theorem}

The requirement of a finite first moment of the initial data in Theorem \ref{thm:v-homo-convergence}, \ref{thm:full-convergence} is natural, since we need to define the firing rate $N(t)$.
We remark that part 2 of Theorem \ref{thm:full-convergence} is not a direct consequence of Theorem \ref{thm:decay-v} and \ref{thm:v-homo-convergence}. Roughly speaking, the decay of non-zero mode from Theorem \ref{thm:decay-v} is in the usual $L^1$ norm, which is not strong enough to control the contributions to the firing rate from these non-zero modes. Nevertheless, part 3 of Theorem \ref{thm:full-convergence} is only related with the ODE \eqref{eq:system-1} and its proof is given in Corollary \ref{cor:ode-steady}, Section \ref{sc:ode-bound-steady}.

Proofs of Theorem \ref{thm:ode}, \ref{thm:v-homo-convergence} and \ref{thm:full-convergence} are presented in Section \ref{sc:3longtime}. In theory, we only need to prove Theorem \ref{thm:full-convergence} for \eqref{eq:nonlinear-toy}, since the two reduced models are special cases of it. However, the investigations on reduced models are crucial for the analysis on the full model \eqref{eq:nonlinear-toy}. Roughly speaking, such a global characterization is obtained by revealing that the long time behavior of the PDE \eqref{eq:nonlinear-toy} is indeed dominated by the ODE system \eqref{eq:system-1}, whose long time asymptotics can be fully characterized in Theorem \ref{thm:ode}. 

With the clear characterization in Theorem \ref{thm:full-convergence}, we confirm that there is no periodic solution in our simplified voltage-conductance model \eqref{eq:nonlinear-toy}, in contrast to the numerical evidence for the original model \cite{caceres2011numerical}. Such a result, though maybe disappointing at a first glance, clearly indicates that there is something significant in the two simplifications. Since our second simplification, extending the domain of $g$, is a technical one, the essential change is in the simplification of the velocity field. Therefore, our result reveals that the non-separable velocity field $J_v(v,g)=-g_Lv+g(V_E-v)$ not only brings difficulties to analysis, but also makes an essential contribution to the dynamics. Its scientific implication is that the leaky conductance $g_L>0$ plays a crucial role in the mechanism of self-sustained oscillations. 

\subsection{Extension to more general velocity fields}\label{sc:gen-vel}
Now we explain how to generalize the results for \eqref{eq:nonlinear-toy} to the more general velocity field $gf(v)$ with a positive and continuous $f(v)$ on $[0,V_F]$, by a change of variable.

Precisely, we consider the following model
	\begin{equation}\label{eq:nonlinear-toy-gen}
	\p_t p+g\p_v(f(v) p)+\p_g((g_{\text{in}}(t)-g)p-a(t)\p_gp))=0,\quad v\in(0,V_F),\ g\in\mathbb{R},\ t>0,
	\end{equation}with a boundary condition on flux in $v$, and we still define firing rate as an integration of flux at $V_F$ over $g>0$,
	\begin{equation}\label{eq:nonlinear-toy-gen-bc}
	f(V_F)p(t,V_F,g)-f(0)p(t,0,g)=0,\quad g\in\mathbb{R},t>0,\qquad N(t):=\int_{0}^{\infty}gf(V_F)p(t,V_F,g)dg.
	\end{equation}
		Actually, we can reduce \eqref{eq:nonlinear-toy-gen} to \eqref{eq:nonlinear-toy} by a change of variable.
	Multiplying \eqref{eq:nonlinear-toy-gen} by $f(v)$, we get
	\begin{equation*}
	    	\p_t (f(v)p)+gf(v)\p_v(f(v) p)+\p_g((g_{in}(t)-g)(f(v)p)-a(t)\p_g(f(v)p))=0,\quad v\in(0,V_F),\ g\in\mathbb{R},\ t>0,
	\end{equation*}
	Let $u(v):=\int_0^v\frac{1}{f(v')}dv'$ and $\tilde{p}(t,u(v),g):=f(v)p(t,v,g)$, we have
		\begin{equation}
	    	\p_t \tilde{p}+g\p_u\tilde{p}+\p_g((g_{in}(t)-g)\tilde{p}-a(t)\p_g\tilde{p})=0,\quad u\in(0,U_F),\ g\in\mathbb{R},\ t>0,
	\end{equation}
	where $U_F:=\int_0^{V_F}\frac{1}{f(v)}dv$. And in new variables \eqref{eq:nonlinear-toy-gen-bc} becomes
		\begin{equation}
\tilde{p}(t,U_F,g)-\tilde{p}(t,0,g)=0,\quad g\in\mathbb{R},t>0,\qquad N(t)=\int_{0}^{\infty}g\tilde{p}(t,U_F,g)dg.
	\end{equation}
	Moreover $\tilde{p}(t,u,g)$ is still a probability density by the following calculation
	\begin{equation*}
	    \int_{\mathbb{R}}dg\int_0^{U_F}du\tilde{p}(t,u,g)=\int_{\mathbb{R}}dg\int_0^{V_F}dv\frac{1}{f(v)}\tilde{p}(t,u(v),g)=\int_{\mathbb{R}}dg\int_0^{V_F}dvp(t,v,g)=1.
	\end{equation*}

	Therefore, all results for the system \eqref{eq:nonlinear-toy} holds for the more general model \eqref{eq:nonlinear-toy-gen}. We remark that if we choose 
	\begin{equation*}
	f(v)=V_E-v,\quad \text{for some }V_E>V_F,
	\end{equation*} then this can be seen as a special case of the original model \eqref{eq:original} with $g_L=0$.
	
	\section{Long time behavior: Proofs of main results }\label{sc:3longtime}
	In this section, we characterize the long time behavior of models \eqref{eq:system-1}, \eqref{eq:reduce-p} and \eqref{eq:nonlinear-toy} by proving Theorem \ref{thm:ode}, \ref{thm:v-homo-convergence} and \ref{thm:full-convergence}. The three models satisfy the same dichotomy: when $g_1/V_F\geq 1$, the firing rate $N(t)$ diverges to infinity as time goes to infinity, while when $g_1/V_F<1$, the solution converges to the unique steady state. The two scenarios are connected in Corollary \ref{cor:ode-steady} which shows that the steady state loses tightness as $g_1/V_F$ approaches $1^{-}$.
	
	In Section \ref{sc:ode3.1} we analyze the ODE \eqref{eq:system-1}, which lays the cornerstone for the analysis of the PDEs \eqref{eq:reduce-p} and \eqref{eq:nonlinear-toy} in Section \ref{sc:3.2} and Section \ref{sc:longtime3.3}, respectively. Very loosely speaking, the proof strategy for the PDEs is to reduce the problem to an non-autonomous ODE, which can be viewed as a perturbation of \eqref{eq:system-1}. Of course, such reductions need investigations on the solution structure and careful estimates.

	\subsection{Long time behavior of the Gaussian solutions: Proof of Theorem \ref{thm:ode}}\label{sc:ode3.1}
	In this section, we analyze the ODE system \eqref{eq:system-1}
	\begin{equation*}
	\begin{aligned}
	  	\frac{db(t)}{dt}&=g_0+{g_1}N(b(t),c(t))-b(t),\\
 	\frac{dc(t)}{dt}&=2a_0+2{a_1}N(b(t),c(t))-2c(t),
	\end{aligned}
	\end{equation*} where $N(b,c)$ is given by \eqref{def-N-bc}
 	\begin{equation}
 	N(b,c)=\frac{1}{V_F}\int_{0}^{+\infty}g\frac{1}{\sqrt{2\pi c}}\exp\left(-\frac{(g-b)^2}{2c}\right)dg\geq 0,\quad b\in\mathbb{R},c>0.
 	\end{equation}
 	
	 First we discuss elementary properties of the nonlinear function $N(b,c)$ in Section \ref{sc:ode-basic}. Then in Section \ref{sc:ode-bound-steady}, we give basic characterization of the dynamics \eqref{eq:system-1} such as the boundedness, the existence of a unique steady state. With these preparations, in Section \ref{sc:ode-convergence} we complete the proof of Theorem \ref{thm:ode}, showing the solution of such a system will converge to the unique steady state when $0<g_1/V_F<1$. We recall from Proposition \ref{prop:g1>=1} otherwise when $g_1/V_F\geq 1$ the firing rate will diverge to infinity as time evolves. The two regimes are linked in Corollary \ref{cor:ode-steady}, which shows that as $g_1/V_F$ approaches $1^{-}$, both components of the steady state $b^*$ and $c^*$ diverge to infinity, and that there is no steady state when $g_1/V_F\geq 1$.

	 Before we start, we recall that $b(t),c(t)$ represents the mean and variance of a Gaussian special solution of \eqref{eq:reduce-p}. Therefore we shall only consider the case $c>0$, since $c$ denotes the variance, while we allow the mean $b\in\mathbb{R}$. In fact, the function $N(b,c)$ \eqref{def-N-bc} is not defined for $c\leq 0$. The following lemma that shows $c(t)>0$ for $t>0$ can be ensured by that initially $c(0)>0$. 
	\begin{lemma}\label{lemma-ODE-initial}
	For the ODE system \eqref{eq:system-1}, if the initial data $c(0)>0$ then $c(t)>0$ for all $t\geq0$.
	\end{lemma}
	\begin{proof}[Proof of Lemma \ref{lemma-ODE-initial}]
	Since $N(b,c)\geq 0$, from \eqref{eq:system-1} we derive $\frac{dc(t)}{dt}\geq 2a_0-2c(t)$. Then the result follows from the standard comparison principle.
	\end{proof}
	In the following discussion we are in the scenario that $b\in\mathbb{R},c>0$.

	\subsubsection{Basics properties for the nonlinearity $N(b,c)$}\label{sc:ode-basic}
In this section we analyze the function $N(b,c)$ in \eqref{def-N-bc}. It is helpful to rewrite $N(b,c)$ in terms of standard normal variables as follows,
	\begin{equation}\label{compute-N}
	    V_FN(b,c)=\int_{\mathbb{R}}g_+\frac{1}{\sqrt{2\pi c}}e^{-\frac{(g-b)^2}{2c}}dg=\int_{\mathbb{R}}(g+b)_+\frac{1}{\sqrt{2\pi c}}e^{-\frac{g^2}{2c}}=\int_{\mathbb{R}}(\sqrt{c}g+b)_+\frac{1}{\sqrt{2\pi}}e^{-\frac{g^2}{2}}dg.
	\end{equation}
	Here as usual, for $x\in\mathbb{R}$, we denote its positive part as $x_+=\max(x,0)$. 
	
	First, we summarize basic facts on the derivatives of $N$.
	\begin{lemma}\label{lemma:N-der}
	For $N(b,c)$ defined in \eqref{def-N-bc}, its derivatives are given by
	\begin{equation}\label{derivative-N}
	    V_F\frac{\p N}{\p b}=\int_{-\frac{b}{\sqrt{c}}}^{+\infty}\frac{1}{\sqrt{2\pi}}e^{-\frac{g^2}{2}}dg,\quad V_F\frac{\p N}{\p c}=\frac{1}{2\sqrt{c}}\frac{1}{\sqrt{2\pi}}e^{-\frac{b^2}{2c}}.
	\end{equation}
	And therefore we have the following bounds
	\begin{equation}\label{eq:mono}
\frac{\p N}{\p b}\in(0,\frac{1}{V_F}),\quad  \frac{\p N}{\p c}>0.
\end{equation}
	\end{lemma}
	\begin{proof}[Proof of Lemma \ref{lemma:N-der}]
From \eqref{compute-N} we have $V_FN(b,c)=\int_{\mathbb{R}}(g+b)_+\frac{1}{\sqrt{2\pi c}}e^{-\frac{g^2}{2c}}dg$, then we get 
$$
V_F\frac{\p N}{\p b}=\int_{\mathbb{R}}\frac{\p (g+b)_+}{\p b}\frac{1}{\sqrt{2\pi c}}e^{-\frac{g^2}{2c}}dg=\int_{\mathbb{R}}\mathbb{I}_{g+b\geq0}\frac{1}{\sqrt{2\pi c}}e^{-\frac{g^2}{2c}}=\int_{-\frac{b}{\sqrt{c}}}^{+\infty}\frac{1}{\sqrt{2\pi}}e^{-\frac{g^2}{2}}dg,
$$ since $\frac{\p (g+b)_+}{\p b}=\mathbb{I}_{g+b\geq0}$
almost everywhere.

Similarly from $V_FN(b,c)=\int_{\mathbb{R}}(\sqrt{c}g+b)_+\frac{1}{\sqrt{2\pi}}e^{-\frac{g^2}{2}}dg$ in \eqref{compute-N}, we calculate
$$
V_F\frac{\p N}{\p c}=\int_{\mathbb{R}}\frac{\p (\sqrt{c}g+b)_+}{\p c}\frac{1}{\sqrt{2\pi}}e^{-\frac{g^2}{2}}dg=\frac{1}{2\sqrt{c}}\int_{\mathbb{R}}g\mathbb{I}_{\sqrt{c}g+b\geq0}\frac{1}{\sqrt{2\pi}}e^{-\frac{g^2}{2}}dg=\frac{1}{2\sqrt{c}}(\frac{1}{\sqrt{2\pi}}e^{-\frac{b^2}{2c}}),
$$ since $\frac{\p (\sqrt{c}g+b)_+}{\p c}=\frac{1}{2\sqrt{c}}g\mathbb{I}_{\sqrt{c}g+b\geq0}$  almost everywhere. Now \eqref{derivative-N} is proved and the bounds \eqref{eq:mono} follows directly.
	\end{proof}

Next, we give estimates on $N(b,c)$ itself.
\begin{lemma}\label{lemma-bound-N}
For $N(b,c)$ defined in \eqref{def-N-bc}, the following estimate holds
\begin{equation}\label{eq:key-boundness}
0\leq b_+\leq V_FN(b,c)\leq b_+ +\sqrt{c},\quad \forall b\in\mathbb{R},c>0.
\end{equation}
Here $b_+=\max(0,b)$ as stated before.
\end{lemma}
\begin{proof}[Proof of Lemma \ref{lemma-bound-N}] By definition \eqref{def-N-bc}, $V_FN(b,c)=\int_{0}^{+\infty}g\frac{1}{\sqrt{2\pi c}}e^{-\frac{(g-b)^2}{2c}}dg\geq0$ and $$V_FN(b,c)-b=-\int_{-\infty}^{0}g\frac{1}{\sqrt{2\pi c}}e^{-\frac{(g-b)^2}{2c}}dg\geq0.$$ Therefore $V_FN\geq \max(b,0)= b_+\geq 0$.

For the upper bound we deduce from \eqref{compute-N}, since $(b+g)_+\leq b_++g_+$,
\begin{equation*}
    V_FN=\int_{\mathbb{R}}(b+g)_+\frac{1}{\sqrt{2\pi c}}e^{-\frac{g^2}{2c}}dg\leq \int_{\mathbb{R}}b_+\frac{1}{\sqrt{2\pi c}}e^{-\frac{g^2}{2c}}dg+\int_{\mathbb{R}}g_+\frac{1}{\sqrt{2\pi c}}e^{-\frac{g^2}{2c}}dg \leq b_++\sqrt{c}.
\end{equation*}

\end{proof}

	Lemme \ref{lemma:N-der} is crucial for the long time behavior of \eqref{eq:system-1} because it implies that \eqref{eq:system-1} is a cooperative system \cite{hirsch2006monotone} (see Section \ref{sc:longtime3.3}).
	
    Lemma \ref{lemma-bound-N} is useful for controlling $N$ throughout Section \ref{sc:3longtime}.

	\subsubsection{Steady state: uniqueness and linear stability}\label{sc:ode-bound-steady}

	In this section, we first prove the boundedness of solutions in Proposition \ref{prop:ode-bound}. Then we prove that there exists a unique steady state which is linearly stable in Proposition \ref{prop:ode-steady}. We focus on the case $0<{g_1}/{V_F}<1$ and study the behavior when $g_1/V_F$ approaches $1^{-}$ to link with the case $g_1/V_F\geq 1$ in Corollary \ref{cor:ode-steady}.

	\begin{proposition}\label{prop:ode-bound}
	When $0<g_1/V_F<1$, every solution of \eqref{eq:system-1} has positive lower bounds for large time. Precisely, given initial value $b(0)\in\mathbb{R}$, $c(0)>0$, there exists time $T^*>0$ such that
	\begin{equation}\label{eq:lower}
	    b(t)\geq \frac{1}{2}g_0>0,\ c(t)\geq \frac{1}{2}a_0>0,\quad t\geq T^*.
	\end{equation}
	Furthermore, the solution is uniformly bounded in time, i.e.,
	\begin{equation}
	    |b(t)|+|c(t)|<C_0,\quad \forall t\geq 0,
	\end{equation} where the constant $C_0$ depends on the initial value and parameters.
\end{proposition}
\begin{proof}[Proof of Proposition \ref{prop:ode-bound}]
WLOG we consider the case $V_F=1$, as in Proposition \ref{prop:g1>=1}. Then in the following proof we can assume $0<g_1<1$. We rewrite \eqref{eq:system-1} as
	\begin{equation}\label{eq:system-V_F=1}
	\begin{aligned}
	\frac{db(t)}{dt}&=g_0+g_1N(b(t),c(t))-b(t)=g_0-(1-g_1)b+g_1(N-b),\\
	\frac{dc(t)}{dt}&=2a_0+2a_1N(b(t),c(t))-2c(t)=2a_0+2a_1N-2c.
	\end{aligned}
	\end{equation}
	First we consider the lower bounds \eqref{eq:lower}. If $b(t)\leq\frac{1}{2}g_0$, we obtain $\frac{db(t)}{dt}\geq g_0-(1-g_1)b(t)\geq \frac{1}{2}g_0>0$, since Lemma \ref{lemma-bound-N} implies $N-b\geq 0$ when $V_F=1$. Similarly if $c(t)\leq \frac{1}{2}a_0$ we obtain  $\frac{dc(t)}{dt}\geq 2(a_0-c(t))\geq a_0>0$, since $N\geq 0$ also by Lemma \ref{lemma-bound-N}. Therefore there exists $T^*$ such that for $t\geq T^*$, we have $b(t)\geq \frac{1}{2}g_0>0, c(t)\geq\frac{1}{2}a_0>0$. 
 
The upper bound relies on the control $N\leq b_++\sqrt{c}$ in Lemma \ref{lemma-bound-N}, and the bound $2\sqrt{c}\leq \delta c+\frac{1}{\delta}$ for all $\delta>0$ via Young's inequality. We shall construct a Liapounov functional. Precisely we will prove that there exists some $\delta_0>0$ such that 
\begin{equation}\label{def-Liapounov}
    \frac{d}{dt}[\frac{1}{2}(b^2+\delta_0 c^2)]\leq C-C_1(b^2+\delta_0 c^2),
\end{equation}with some $C,C_1>0$. Then by Gronwall inequality, $b^2+\delta_0 c^2$ will be uniformly bounded in time and the proof is complete. In the following, $C$ denotes a positive constant which may vary from line to line as custom. Recall $N\leq b_++\sqrt{c}$ from Lemma \ref{lemma-bound-N}, we compute
\begin{equation}\label{eq:upbd-system}
    \begin{aligned}
	\frac{1}{2}\frac{d(b^2(t))}{dt}&=b(g_0+g_1N-b)\leq C|b|-(1-g_1)b^2+g_1|b|\sqrt{c},\\
	\frac{1}{2}\frac{d(c^2(t))}{dt}&=c(2a_0+2a_1N-2c)\leq C|c|+2a_1|bc|+2a_1c\sqrt{c}-2c^2.
	\end{aligned}
\end{equation}
By Young's inequality we have $g_1|b|\sqrt{c}\leq C|b|+g_1\delta_1|bc|$ for some $\delta_1>0$ to be determined and $2a_1c\sqrt{c}\leq C|c|+c^2$. Then from \eqref{eq:upbd-system} we deduce
\begin{equation}\label{eq:upbd-system-2}
    \begin{aligned}
	\frac{1}{2}\frac{d(b^2(t))}{dt}&\leq C|b|-(1-g_1)b^2+g_1\delta_1|bc|,\\
	\frac{1}{2}\frac{d(c^2(t))}{dt}&\leq C|c|+2a_1|bc|-c^2.
	\end{aligned}
\end{equation}
Then for $0<\delta_0<1$ to be determined, we calculate
\begin{equation}\label{tmp-deri-ode-1}
    \frac{d}{dt}[\frac{1}{2}(b^2+\delta_0 c^2)]\leq C(|b|+|c|)-(1-g_1)b^2+(g_1\delta_1+2a_1\delta_0)|bc|-\delta_0c^2.
\end{equation}Now we choose $\delta_1$ such that $g_1\delta_1=a_1\delta_0$. And we choose $\delta_0$ small enough such that
\begin{align*}
    (g_1\delta_1+2a_1\delta_0)^2=9a_1^2\delta_0^2<4(1-g_1)\delta_0,
\end{align*} which can be ensured by choosing $\delta_0<\min(4\frac{(1-g_1)}{9a_1^2},1)$.
For such $\delta_0$ and $\delta_1$, by mean-value inequality there exists $\delta_2>0$ such that $(1-g_1)b^2-(g_1\delta_1+2a_1\delta_0)|bc|+\delta_0c^2\geq \delta_2(b^2+c^2)$. Then by \eqref{tmp-deri-ode-1} we get
\begin{equation}\label{tmp-deri-ode-1.5}
    \frac{d}{dt}[\frac{1}{2}(b^2+\delta_0 c^2)]\leq C(|b|+|c|)-\delta_2(b^2+c^2).
\end{equation} Again we use Young's inequality to control the low order terms, $2|b|\leq \frac{1}{\delta}+\delta b^2,2|c|\leq \frac{1}{\delta}+\delta c^2$. In this way we conclude
\begin{equation}\label{tmp-deri-ode-2}
    \frac{d}{dt}[\frac{1}{2}(b^2+\delta_0 c^2)]\leq C-C_0(b^2+c^2),
\end{equation} for some $C,C_0>0$, which implies \eqref{def-Liapounov}. Then the proof of the upper bound is complete.
\end{proof}

Next, we show that in the regime $0<g_1/V_F<1$ the system \eqref{eq:system-1} has a unique steady state $(b^*,c^*)$ which is linearly stable. Moreover, we derive bounds on $(b^*,c^*)$ which allow us to see its behavior as as $g_1/V_F$ approaches $1^{-}$.

\begin{proposition}\label{prop:ode-steady}
When $0<g_1/V_F<1$, then the ODE system \eqref{eq:system-1} has a unique steady state $(b^*,c^*)$ in $\mathbb{R}\times\mathbb{R}^+$. And the steady state is linearly stable. Moreover, the following bounds hold
\begin{equation}\label{bound-b*c*}
    \begin{aligned}
        b^*&\geq \frac{g_0}{1-g_1/V_F}>0,\\
        c^*&\geq a_0+\frac{a_1}{V_F}\frac{g_0}{1-g_1/V_F}>0.
    \end{aligned}
\end{equation}
\end{proposition}
\begin{proof}[Proof of Proposition \ref{prop:ode-steady}]WLOG we consider the case $V_F=1$, otherwise we use $g_1/V_F,a_1/V_F$ instead of $g_1,a_1$. Then in the following we have $0<g_1<1$.

Suppose $(b^*,c^*)$ is a steady state. From \eqref{eq:system-1} we get
\begin{equation}\label{eq:system-1-steady}
\begin{aligned}
0&=g_0+g_1N(b^*,c^*)-b^*,\\
0&=2a_0+2a_1N(b^*,c^*)-2c^*.
\end{aligned}
\end{equation} First, we suppose the steady state exists and derive the bound \eqref{bound-b*c*}. Applying $N(b,c)\geq b_+\geq b$ from Lemma \ref{lemma-bound-N}, we derive from the first equation of \eqref{eq:system-1-steady}
\begin{equation*}
    b^*\geq g_1b^*+g_0.
\end{equation*} Thus $b^*\geq \frac{g_0}{1-g_1}$. Using Lemma \ref{lemma-bound-N} and this bound on $b^*$ in the second equation of \eqref{eq:system-1-steady}, we deduce the bound of $c^*$
\begin{equation*}
    c^*\geq a_0+a_1b^*\geq a_0+a_1\frac{g_0}{1-g_1}.
\end{equation*}

Next, we solve \eqref{eq:system-1-steady} to find the steady state. One can eliminate the firing rate $N$ in \eqref{eq:system-1-steady} to obtain
\begin{equation}\label{steady-tmp-1}
a_1b^*-g_1c^*=a_1g_0-g_1a_0.
\end{equation} 
We define a function $\beta(c):=g_0+\frac{g_1}{a_1}(c-a_0)$, then \eqref{steady-tmp-1} is equivalent to
\begin{equation}\label{function-beta}
b^*=\frac{1}{a_1}(a_1g_0+g_1(c^*-a_0))=g_0+\frac{g_1}{a_1}(c^*-a_0)=\beta(c^*).
\end{equation}Now we have represented $b^*$ as a function of $c^*$. Plugging \eqref{function-beta} back to the second equation in \eqref{eq:system-1-steady}, we get
\begin{equation}
    a_0+a_1N(\beta(c^*),c^*)-c^*=0.
\end{equation}
Therefore finding steady states reduces to finding zeros of the following nonlinear function
\begin{equation}\label{eq:f}
F(c):=a_0+a_1N(\beta(c),c)-c,\quad c\in[a_0,\infty)
\end{equation}
Here we restrict the domain to $c\geq a_0$ because $F(c)\geq a_0-c>0$ if $c<a_0$, since $N\geq0$ by Lemma \ref{lemma-bound-N}. For later reference, we also note that at steady state $c^*\geq a_0$ implies 
\begin{equation}\label{lower-b}
    b^*=\beta(c^*)\geq g_0>0.
\end{equation}

We observe that $F(a_0)=a_1N(\beta(a_0),a_0)>0$ and that from Lemma \ref{lemma-bound-N}
\begin{equation}
    F(c)\leq C+g_1(c-a_0)_+ +a_1\sqrt{c}-c\leq C+a_1\sqrt{c}-(1-g_1)c.
\end{equation}Therefore $F(c)$ goes to $-\infty$ as $c$ goes to infinity. As a result, there exists at least one zero of $F$ on $(a_0,+\infty)$, which implies the existence of a steady state.

If we can show whenever $c$ is a zero of $F$, we have $F'(c)<0$, then we can conclude that $F(c)$ has exactly one zero on $(a_0,+\infty)$, which implies that the steady state is unique. We take the derivative of $F$
\begin{equation}\label{eq:unique-state}
F'(c)=g_1\frac{\p N}{\p b}(\beta(c),c)+a_1\frac{\p N}{\p c}(\beta(c),c)-1.
\end{equation}
The sign of $F'(c)$ for $F(c)=0$ turns out to be related with the sign of the determinant of the linearized matrix at the steady state. Thus we postpone the proof of uniqueness after some preliminary discussion on linear stability.
 
Suppose $(b^*,c^*)$ is a steady state, we write down the linearized matrix of \eqref{eq:system-1} 
\begin{equation}J=
\begin{pmatrix}
g_1\frac{\partial N}{\partial b}-1& g_1\frac{\p N}{\p c}.\\
2a_1\frac{\p N}{\p b}& 2a_1\frac{\p N}{\p c}-2
\end{pmatrix}.
\end{equation}
Let us first look at the trace of $J$. The first diagonal entry is negative since $g_1\frac{\p N}{\p b}=g_1\int_{-\frac{b}{\sqrt{c}}}^{+\infty}\frac{1}{\sqrt{2\pi}}e^{-\frac{g^2}{2}}dg<1$ by Lemma \ref{lemma:N-der} and our assumption $V_F=1$. For the second diagonal entry, substituting $a_1=\frac{c^*-a_0}{N}$ at a steady state, we get
\begin{align}\notag
a_1\frac{\p N}{\p c}&=\frac{c^*-a_0}{N}\frac{\p N}{\p c}\\ &=\frac{c^*-a_0}{c^*}\frac{c^*}{N}\frac{\p N}{\p c}<\frac{c^*}{N}\frac{\p N}{\p c}.\label{compute-diag-2}
\end{align}
To proceed, we introduce the notation $\lambda:=\frac{b^*}{\sqrt{c^*}}$ which is positive by \eqref{lower-b}. Recall the formula of $\frac{\p N}{\p c}$ in Lemma \ref{lemma:N-der} 
\begin{equation}\label{compute-diag-2-1}
\sqrt{c}\frac{\p N}{\p c}=\frac{1}{2\sqrt{2\pi}}\exp\left(-\lambda^2/2\right).
\end{equation}
And we rewrite the expression for $N$ from \eqref{compute-N}
\begin{align}\notag
    \frac{N}{\sqrt{c^*}}=\int_{\mathbb{R}}(g+\frac{b}{\sqrt{c^*}})_+\frac{1}{\sqrt{2\pi}}e^{-\frac{1}{2}g^2}dg&=\int_{-\lambda}^{+\infty}(g+\lambda)\frac{1}{\sqrt{2\pi}}e^{-\frac{1}{2}g^2}dg\\&=\frac{1}{\sqrt{2\pi}}e^{-\lambda^2/2}+\lambda\int_{-\lambda}^{+\infty}\frac{1}{\sqrt{2\pi}}e^{-g^2/2}dg. \label{compute-diag-2-2}
\end{align}
Plugging \eqref{compute-diag-2-1} and \eqref{compute-diag-2-2} into \eqref{compute-diag-2} we obtain
\begin{align*}
a_1\frac{\p N}{\p c}<\frac{\sqrt{c}\frac{\p N}{\p c}}{\frac{N}{\sqrt{c}}}=\frac{\frac{1}{2\sqrt{2\pi}}\exp(-\lambda^2/2)}{\frac{1}{\sqrt{2\pi}}\exp(-\lambda^2/2)+\lambda\int_{-\lambda}^{+\infty}\frac{1}{\sqrt{2\pi}}\exp(-g^2/2)dg}<\frac{1}{2}<1.
\end{align*}
Therefore the second diagonal entry is also negative, which implies the trace of $J$ is negative. For linear stability, it remains to study the sign of the determinant, which is related with the sign of $F'(c^*)$ for $F$ in \eqref{eq:f} as follows
\begin{equation}\frac{1}{2}\det J=1-g_1\frac{\p N}{\p b}-a_1\frac{\p N}{\p c}=-F'(c^*).\label{eq:det}
\end{equation}
To prove linear stability, it remains to show that the determinant is positive. And if the determinant is positive, we have $F'(c^*)<0$ whenever $c^*$ is a zero of $F$, which implies the uniqueness.

Therefore our problem reduces to show that $\det J>0$. Following similar calculations for the diagonal entities, we rewrite the expression of $\det J$ \eqref{eq:det} as follows.
\begin{align}\notag
\frac{1}{2}\det J&=1-g_1\frac{\p N}{\p b}-a_1\frac{\p N}{\p c}\\&=1-g_1(\int_{-\infty}^{\lambda}\frac{1}{\sqrt{2\pi}}\exp(-g^2/2)dg)-\frac{c^*-a_0}{c^*}\frac{\frac{1}{2\sqrt{2\pi}}\exp(-\lambda^2/2)}{\frac{1}{\sqrt{2\pi}}\exp(-\lambda^2/2)+\lambda\int_{-\lambda}^{+\infty}\frac{1}{\sqrt{2\pi}}\exp(-g^2/2)dg}\notag\\
&=1-g_1A_1-\frac{c^*-a_0}{c^*}A_2.\label{tmp-det}
\end{align}

Here we define
\begin{equation}\label{def-A1A2}
\begin{aligned}
A_1&:=\int_{-\infty}^{\lambda}\frac{1}{\sqrt{2\pi}}\exp(-g^2/2)dg,\\
A_2&:=\frac{\frac{1}{2\sqrt{2\pi}}\exp(-\lambda^2/2)}{\frac{1}{\sqrt{2\pi}}\exp(-\lambda^2/2)+\lambda\int_{-\lambda}^{+\infty}\frac{1}{\sqrt{2\pi}}\exp(-g^2/2)dg}.
\end{aligned}
\end{equation}
We need the following technical lemma.
\begin{lemma}\label{lemma-ode-technical} For all $\lambda\geq 0$, the following inequality holds
\begin{equation}\label{eq:A1A2}
A_1+A_2\leq 1,
\end{equation} where $A_1,A_2$ are defined as in \eqref{def-A1A2}.
\end{lemma}We postpone the proof of Lemma \ref{lemma-ode-technical} to the end of this section. Applying Lemma \ref{lemma-ode-technical} to \eqref{tmp-det} with $\lambda=\frac{b^*}{\sqrt{c^*}}>0$, we conclude $$\frac{1}{2}\det J=1-g_1A_1-\frac{c^*-a_0}{c*}A_2>1-A_1-A_2\geq 0.$$ Therefore both the linear stability and the uniqueness are proved.
\end{proof}

The bound \eqref{bound-b*c*} in Proposition \ref{prop:ode-bound} implies that both $b^*$ and $c^*$ go to infinity as $g_1/V_F$ approaches $1^{-}$. This explains why there is no steady state of \eqref{eq:system-1} when $g_1/V_F\geq 1$. Noting that $b^*$ represents the mean of a Gaussian solution of \eqref{eq:reduce-p} and that $c^*$ is the variance, we interpret that the steady state loses tightness  as $g_1/V_F$ approaches $1^{-}$.

\begin{corollary}\label{cor:ode-steady}
When $0<g_1/V_F<1$, the unique steady state of \eqref{prop:ode-steady} $(b^*,c^*)$ satisfy
\begin{equation*}
    b^*\rightarrow +\infty, c^*\rightarrow +\infty,
\end{equation*} as $g_1/V_F$ approaches $1^{-}$. And when $g_1/V_F\geq 1$, there is no steady state of \eqref{eq:system-1}.
\end{corollary}
\begin{proof}[Proof of Corollary \ref{cor:ode-steady}]
    The part when $0<g_1/V_F<1$ follows from the bound \eqref{bound-b*c*} in Proposition \ref{prop:ode-bound}. Non-existence of a steady state when $g_1/V_F\geq 1$ is implied by Proposition \ref{prop:g1>=1}. Indeed, since the ODE \eqref{eq:system-1} represents the special Gaussian solution \eqref{eq:Gaussian anstaz} of \eqref{eq:reduce-p}, Proposition \ref{prop:g1>=1} gives that when $g_1/V_F\geq 1$ for every solution $(b(t),c(t))$,
    \begin{equation*}
        N(b(t),c(t))\rightarrow +\infty,\quad \text{as } t\rightarrow+\infty, 
    \end{equation*} which implies that there is no steady state.
\end{proof}

Now we complete this section by proving Lemma \ref{lemma-ode-technical}.
\begin{proof}[Proof of Lemma \ref{lemma-ode-technical}]
 For convenience we introduce the following notations for Gaussian density and Gaussian tail probability $$\phi(\mu)=\frac{1}{\sqrt{2\pi }}\exp(-\frac{\mu^2}{2}),\Phi(\mu)=\int_{\mu}^{+\infty}\frac{1}{\sqrt{2\pi }}\exp(-\frac{g^2}{2})dg,$$ and rewrite \eqref{eq:A1A2} into
 \begin{equation*}
 1-\Phi(\lambda)+\frac{1}{2}\frac{\phi(\lambda)}{\phi(\lambda)+\lambda(1-\Phi(\lambda))}\leq 1,
 \end{equation*}which is equivalent to
 \begin{equation}\label{eq:ls-final-0}
\frac{1}{2}{\phi(\lambda)}\leq \Phi(\lambda)(\phi(\lambda)+\lambda(1-\Phi(\lambda))).
\end{equation}
First we consider the case $\lambda\geq 1$. We rewrite \eqref{eq:ls-final-0} as
 \begin{equation}\label{eq:ls-final}
\frac{1}{2}\leq \Phi(\lambda)+\lambda(1-\Phi(\lambda))\frac{\Phi(\lambda)}{{\phi(\lambda)}}.
\end{equation}
By the famous Gaussian tail probability inequality $\Phi(\lambda)\geq \frac{\lambda}{1+\lambda^2}\phi(\lambda)$, we have $\frac{\Phi(\lambda)}{\phi(\lambda)}\geq \frac{\lambda}{1+\lambda^2}$, therefore
\begin{align*}
\Phi(\lambda)+\lambda(1-\Phi(\lambda))\frac{\Phi(\lambda)}{{\phi(\lambda)}}&\geq \Phi(\lambda)+\lambda(1-\Phi(\lambda))\frac{\lambda}{1+\lambda^2}\\&=\frac{\lambda^2}{1+\lambda^2}+\Phi(\lambda)\frac{1}{\lambda^2+1}\geq \frac{\lambda^2}{1+\lambda^2}\geq \frac{1}{2}.
\end{align*} In the last inequality we use $\lambda\geq 1$.

It remains to deal with the case $\lambda\in[0,1]$. In this case, we rewrite \eqref{eq:ls-final-0} as
\begin{equation}\label{tmp-Hlambda}
    H(\lambda)\geq 0,\quad \lambda\in[0,1],
\end{equation}where $H(\lambda)$ is defined as
 \begin{align*}
H(\lambda):&=\Phi(\lambda)(\phi(\lambda)+\lambda(1-\Phi(\lambda)))-\frac{1}{2}{\phi(\lambda)}.
\end{align*} 
We claim that $H''(x)\leq 0$ for $x\in[0,1]$. If this is true, we deduce for $\lambda\in[0,1]$, $H(\lambda)\geq\min(H(0),H(1))$. And $\min(H(0),H(1))=0$ since $H(0)=0$ and from the $\lambda\geq 1$ case we know $H(1)\geq 0$. Therefore our final task is to show the desired sign of  $H''(x)$, actually one calculates
\begin{align*}
    H''(x)=\frac{e^{-x^2} (2\sqrt{2} e^{x^2/2} (x^2 - 3)\int_0^{x/\sqrt{2}}e^{-t^2}dt + 2 x)}{4\pi}.
\end{align*}
And by $\int_0^{x/\sqrt{2}}e^{-t^2}dt\geq \frac{x}{\sqrt{2}}e^{-x^2/2}$, we deduce
\begin{align*}
    2\sqrt{2} e^{x^2/2}(x^2-3)\int_0^{x/\sqrt{2}}e^{-t^2}dt+2x&\leq 2\sqrt{2}e^{x^2/2}(x^2-3)\frac{x}{\sqrt{2}}e^{-x^2/2}+2x\\&=2x(x^2-3)+2x\leq -4x+2x\leq 0,\quad \forall x\in[0,1].
\end{align*} Then the proof is complete.

\end{proof}

	\subsubsection{Proof of Theorem \ref{thm:ode}}\label{sc:ode-convergence}
	With preparations in previous sections, now we prove Theorem \ref{thm:ode}, that all solutions of \eqref{eq:system-1} converge to the unique steady state $(b^*,c^*)$ when $g_1/V_F<1$. 
	
	In Proposition \ref{prop:ode-steady} we show the linear stability, which implies local convergence in a neighborhood of $(b^*,c^*)$. For global convergence, the key observation is that the system \eqref{eq:system-1} is a cooperative system \cite{hirsch2006monotone} in the following sense: If we write the ODE system \eqref{eq:system-1} as
\begin{equation*}
\begin{cases}
\frac{db}{dt}=f_1(b,c),\\
\frac{dc}{dt}=f_2(b,c),
\end{cases}
\end{equation*}then by Lemma \ref{lemma:N-der} the following derivatives are both positive,
\begin{equation*}
\frac{\p f_1}{\p c}=g_1\frac{\p N}{\p c}>0,\quad \frac{\p f_2}{\p b}=a_1\frac{\p N}{\p b}>0.
\end{equation*}The long time behavior of cooperative systems, or in a more general framework called monotone systems, has been characterized extensively in literature. We refer to \cite{hirsch2006monotone}, in particular, its section 3.7 for ODE in $\mathbb{R}^2$. 
\begin{proof}[Proof of Theorem \ref{thm:ode}]Now we begin the proof of Theorem \ref{thm:ode}. By Proposition \ref{prop:ode-bound} the solution $(b(t),c(t))$ is uniformly bounded in $\mathbb{R}\times\mathbb{R}^+$ for $t>0$, with a positive lower bound for $c$. We shall use the following result on cooperative systems.
\begin{theorem*}[{\cite[Theorem 3.21, Section 3.7]{hirsch2006monotone}}, in a modified form]
For a solution  $y(t)=(y_1(t),y_2(t))\in\mathbb{R}^2,t>0,$ of a cooperative system. There exists $t^*$ such that for $t>t^*$, $y_1(t)$ and $y_2(t)$ are both monotone in time.
\end{theorem*} By this theorem, $b(t)$ and $c(t)$ will be eventually monotone in time, since \eqref{eq:system-1} is a monotone system. Together with the upper and lower bounds, we deduce that $(b(t),c(t))$ will converge, monotonely in each component, to some point in $\mathbb{R}\times\mathbb{R}^+$ as $t$ goes to infinity. Then it is easy to verify that such point must be a steady state for \eqref{eq:system-1}. Therefore by uniqueness in Proposition \ref{prop:ode-steady}, it is the unique steady state $(b^*,c^*)$.
\end{proof}
An alternative approach to prove Theorem \ref{thm:ode} is to use Theorem 3.22 in \cite{hirsch2006monotone}, which states that for a cooperative system on a domain $D\subset\mathbb{R}^2$, the $\omega$ limit set contains a single equilibrium if the trajectory has a compact closure in $D$. In our case $D=\mathbb{R}\times\mathbb{R}^+$, and Proposition \ref{prop:ode-bound}(bounds for $b,c$ and positive lower bound for $c$) ensures that the closure of the trajectory is compact.

	\subsection{Long time behavior of the $v$-homogeneous problem: Proof of Theorem \ref{thm:v-homo-convergence}}\label{sc:3.2}

    With results on the ODE system \eqref{eq:system-1}, in this section we move on to prove Theorem \ref{thm:v-homo-convergence} for the $v$-homogeneous PDE \eqref{eq:reduce-p}. Specifically, we aim to show that when $g_1/V_F<1$, as time evolves a solution of \eqref{eq:reduce-p} converges to the unique steady state given by \eqref{steady-g}
	\begin{equation*}
	p^*(g)=\frac{1}{\sqrt{2\pi c^*}}\exp\left(-\frac{(g-b^*)^2}{2c^*}\right),
	\end{equation*}
	which is a Gaussian with mean $b^*$ and variance $c^*$. Here $(b^*,c^*)$ is the unique steady state of ODE \eqref{eq:system-1}.

    We work with the solution formula in Lemma \ref{lemma-explicit-solu}. First we observe that a general solution will converge to a time-varying Gaussian. Then we find an ODE structure for the solution of \eqref{eq:reduce-p}, which can be viewed as a perturbation of \eqref{eq:system-1}. Such a perturbation can be analyzed as an asymptotically autonomous system \cite{thieme1994asymptotically} (see also Appendix \ref{sc:asyauto}). 
    
	And as aforementioned, we use the notation $p(t,g):=p_0(t,g)$, since \eqref{eq:reduce-p} is a reduced model for \eqref{eq:nonlinear-toy}, 

	We start from the following semi-explicit solution formula of \eqref{eq:reduce-p}, which is the special case $k=0$ in Lemma \ref{lemma-explicit-solu}, given by
\begin{equation}\label{eq:ex-solu}
p(t,g)=(p_{t,0}*G_t)(g).
\end{equation}
Here as in \eqref{shrinkage}, $p_{t,0}$ is a shrinkage of the initial data $p_{\text{init}}(g)$, which is a probability density on $\mathbb{R}$,
\begin{equation}\label{pt0-sec32}
p_{t,0}(y)=e^{t}p_{\text{init}}(e^ty),\quad p(0,g)=p_{\text{init}}(g).
\end{equation}
Here in the case $k=0$, $G_t$ is simply a Gaussian which depends on the solution implicitly,
\begin{equation}
G_t(x):=\frac{1}{\sqrt{2\pi C(t)}}\exp\left(-\frac{(x-B(t))^2}{2C(t)}\right)
\end{equation}
And $B(t)$ and $C(t)$ depend on the firing rate as in \eqref{eq:BC}
\begin{equation}\label{eq:AB}
\begin{aligned}
B(t)&=\int_{0}^{t}e^{-(t-s)}g_{in}(s)ds=\int_{0}^{t}e^{-(t-s)}(g_0+g_1N(s))ds.\\
C(t)&=2\int_{0}^{t}e^{-2(t-s)}a(s)ds=2\int_{0}^{t}e^{-2(t-s)}(a_0+a_1N(s))ds.
\end{aligned}
\end{equation}

To see why we expect that a general solution converges to Gaussian, we can rewrite the solution formula \eqref{eq:ex-solu} as
\begin{equation}\label{eq-superpo-Gaussian}
p(t,g)=\int_{\mathbb{R}}\frac{1}{\sqrt{2\pi C(t)}}\exp\left(-\frac{(g-B(t)-y)^2}{2C(t)}\right)p_{t,0}(y)dy.
\end{equation}
This shows, at each time $t$, the solution can be understood as a superposition of a sequence of Gaussian with a same variance $C(t)$ and different means $B(t)+y$. Furthermore, since $p_{\text{init}}(y)$ is a probability density function, one observes that $p_{t,0}(y)=e^{t}p_{\text{init}}(e^ty)$ converges to the Dirac measure under some mild condition. With suitable uniform in $t$ bounds on $G_t$, one can show
\begin{equation}
\|p(t,\cdot)-G_t(\cdot)\|\rightarrow 0,\quad \text{as $t$ goes to infinity},
\end{equation} in some norm. This indicates that a general solution of \eqref{eq:reduce-p} can converge to a time-varying Gaussian. 

 Now we look into the dynamics of the mean and variance of $G_t$, i.e., $B(t),C(t)$. Differentiate \eqref{eq:AB} with respect to $t$, and one obtains,
\begin{equation}\label{dynamics-BC}
\begin{aligned}
\frac{dB(t)}{dt}&=g_0+g_1N(t)-B(t),\\
\frac{dC(t)}{dt}&=2(a_0+a_1N(t)-C(t)).
\end{aligned}
\end{equation}
This is very similar to the ODE system \eqref{eq:system-1}, except the determination of the firing rate $N(t)$. While in system \eqref{eq:system-1}, derived from an exact Gaussian solution, the firing rate is solely determined by mean $b(t),c(t)$ as a function $N(b(t),c(t))$, which is defined in \eqref{def-N-bc}. While in general case, we need to deal with a superposition of Gaussians with different means \eqref{eq-superpo-Gaussian} as follows,
\begin{align}\notag
N(t)&=\int_{\mathbb{R}}g_+p(t,g)dg=\int_{\mathbb{R}}g_+dg\int_{\mathbb{R}}G_t(g-y)p_{t,0}(y)dy\\
&=\int_{\mathbb{R}}p_{t,0}(y)dy\int_{\mathbb{R}}g_+G_t(g-y)dg\\&=\int_{\mathbb{R}}(N(B(t)+y,C(t)))p_{t,0}(y)dy.\label{v-homo-N}
\end{align} In the last line, we use that $G_t(\cdot-y)$ is a Gaussian with variance $C(t)$ and mean $B(t)+y$. As mentioned, we expect that $p_{t,0}(y)=e^{t}p_{\text{init}}(e^ty)$ converges to the Dirac measure at zero, then from \eqref{v-homo-N} we expect that $N(t)$ will converge to $N(B(t),C(t))$
as time evolves. This inspires us to reformulate \eqref{dynamics-BC} as a perturbation of \eqref{eq:system-1},
	\begin{equation}\label{eq:non-auto}
\begin{aligned}
\frac{dB(t)}{dt}&=g_0+g_1N(B,C)-B(t)+g_1\epsilon(t),\\
\frac{dC(t)}{dt}&=2a_0+2a_1N(B,C)-2C(t)+2a_1\epsilon(t).
\end{aligned}
\end{equation}
Here $\epsilon(t)=N(B(t),C(t))-N(t)$ is the ``error'' from using the ODE system to approximate \eqref{dynamics-BC}, which can be seen as a perturbation, and actually goes to zero as time evolves. Since by Lemma \ref{lemma:N-der} we have $|\frac{\p N}{\p b}|<\frac{1}{V_F}$, we estimate the error as follows,
\begin{align}\notag
|\epsilon(t)|=|N(t)-N(B(t),C(t))|&\leq\int_{\mathbb{R}}|N(B(t)+y,C(t))-N(B(t),C(t))|p_{t,0}(y)dy\\&\leq\frac{1}{V_F}\int_{\mathbb{R}}|y|p_{t,0}(y)dy=\frac{1}{V_F}e^{-t}\int_{\mathbb{R}}|y|p_{\text{init}}(y)dy.
\end{align}
We have assumed that the initial data has a finite first moment in  Theorem \ref{thm:v-homo-convergence}. Therefore the perturbation $\epsilon(t)$ exponentially decays \begin{equation}\label{eq:decay-eps}
    |\epsilon(t)|\leq Ce^{-t},
\end{equation} for some $C>0$. As a perturbation from the ODE \eqref{eq:system-1}, the long time behavior of \eqref{eq:non-auto} can be analyzed in the framework of asymptotically autonomous system in literature \cite{thieme1994asymptotically}. For reader's convenience, we summarize the precise definition and a property on asymptotically autonomous system in the Appendix \ref{sc:asyauto}.

\begin{proof}[Proof of Theorem \ref{thm:v-homo-convergence}]
Now we begin the proof of Theorem \ref{thm:v-homo-convergence}. We first show $(B(t),C(t))$ converges to $(b^*,c^*)$, the unique steady state of \eqref{eq:system-1}. This is from analyzing the system \eqref{eq:non-auto} in the framework of asymptotically autonomous system. By Theorem \ref{thm:auto} in Appendix, we can prove that $(B(t),C(t))$ converge to the unique equilibrium of \eqref{eq:system-1}, if we can show that the solution is uniformly bounded in time. 

For this boundedness of $(B(t),C(t))$, we adapt the proof of Proposition \ref{prop:ode-bound}. Actually, one can calculate the same Liapounov functional as in \eqref{tmp-deri-ode-2} and find that
\begin{align*}\label{tmp-pde}
    \frac{d}{dt}[\frac{1}{2}(B^2+\delta_0 C^2)]&\leq C-C_0(B(t)^2+C(t)^2)+Ce^{-t}(|B(t)|+|C(t)|)\\&\leq C+Ce^{-t}-\frac{C_0}{2}(B(t)^2+C(t)^2),
\end{align*}since by Young's inequality $|B|\leq \frac{1}{\delta}+\delta|B|^2,|C|\leq \frac{1}{\delta}+\delta|C|^2$ for all $\delta>0$. Then we derive that the solution is still uniformly bounded and therefore we deduce from Theorem \ref{thm:auto} that
\begin{equation}
    (B(t),C(t))\rightarrow(b^*,c^*),\quad \text{as $t$ goes to infinity}.
\end{equation}
Then by the triangle inequality, we get
\begin{equation}
||p(t,\cdot)-p^*(\cdot)||_{L^1(\mathbb{R})}\leq\|p^*(\cdot)-G_t(\cdot)\|_{L^1(\mathbb{R})}+\|G_t(\cdot)-p(t,\cdot)\|_{L^1(\mathbb{R})}.
\end{equation}
The first term is the difference between two Gaussians, and converges to zero as time evolves since $(B(t),C(t))$ converges to $(b^*,c^*)$. It remains to deal with the second term. Since $(B(t),C(t))$ converges to $(b^*,c^*)$ we can get uniform bounds on $\|\p_gG_t\|_{L^1(\mathbb{R})}$,
\begin{equation}\label{uni-Gt}
    \|\p_gG_t\|_{L^1(\mathbb{R})}=\int_{\mathbb{R}}\frac{|g-B(t)|}{C(t)}\frac{1}{\sqrt{2\pi C(t)}}\exp\left(-\frac{(g-B(t))^2}{2C(t)}\right)dg\leq C_1<\infty,
\end{equation}
for $t$ large. Then by the formula \eqref{eq:ex-solu}, as a standard estimate for an approximation to the identity, we deduce
\begin{align*}
\|p(t,\cdot)-G_t(\cdot)\|_{L^1({\mathbb{R}})}&=\|(p_{t,0}*G_t)(\cdot)-G_t(\cdot)\|_{L^1(\mathbb{R})}\\&\leq \int_{\mathbb{R}}\int_{\mathbb{R}}|G_t(g-y)-G(g)|p_{t,0}(y)dgdy\leq \int_{\mathbb{R}}\int_{\mathbb{R}}\left(\int_0^{1}|\p_gG_t(g-sy)|ds\right)|y|p_{t,0}(y)dgdy\\&=\int_{\mathbb{R}}\int_0^{1}dsdy|y|p_{t,0}(y)\int_{\mathbb{R}}|\p_gG_t(g-sy)|dg=\left(\int_{\mathbb{R}}|y|p_{t,0}(y)dy\right)\|\p_gG_t\|_{L^1(\mathbb{R})}.
\end{align*}
Finally, by the uniform bound on $\|\p_gG_t\|_{L^1(\mathbb{R})}$ \eqref{uni-Gt} and the definition of $p_{t,0}$ \eqref{pt0-sec32}, we conclude
\begin{align*}
    \|p(t,\cdot)-G_t(\cdot)\|_{L^1({\mathbb{R}})}&\leq \left(\int_{\mathbb{R}}|y|p_{t,0}(y)dy\right)\|\p_gG_t\|_{L^1(\mathbb{R})}\\&\leq C\left(\int_{\mathbb{R}}|y|p_{t,0}(y)dy\right)=Ce^{-t}\int_{\mathbb{R}}|y|p_{\text{init}}(y)dy\leq Ce^{-t}.
\end{align*}

\end{proof}

	\subsection{Long time behavior of the full model: Proof of Theorem \ref{thm:full-convergence}}\label{sc:longtime3.3}

	In this section we characterize the long time behavior of the full model \eqref{eq:nonlinear-toy}. Recall the Fourier expansion in $v$ \eqref{fourier-v}
	\begin{equation*}
\begin{aligned}
 p(t,v,g)&=\frac{1}{V_F}\sum_{k=-\infty}^{+\infty}p_k(t,g)e^{ikv\frac{2\pi}{V_F}},\\
p_k(t,g):&=\int_{0}^{V_F}p(t,v,g)e^{-ikv\frac{2\pi}{V_F}}dv,\quad k\in\mathbb{Z}.
\end{aligned}
	\end{equation*}
	We have shown that the solution $p(t,v,g)$ converges to its zeroth mode $\frac{1}{V_F}p_0(t,g)$ in Theorem \ref{thm:decay-v}. Then in Theorem \ref{thm:v-homo-convergence} we give the long time asymptotics for the reduced model when there is only the zeroth mode. However, the long time behavior of the full model \eqref{eq:nonlinear-toy} is not a straightforward consequence. Here is the difficulty -- we need to control the contributions to the firing rate $N(t)$ from non-zero modes, for which the convergence in Theorem \ref{thm:decay-v} is not sufficient. This point will be elaborated throughout this section.

	In terms of the Fourier modes in $v$, for $t>0$, the firing rate $N(t)$ can be represented as
	\begin{align}\label{tmp-sum-N-0}
	    N(t)&=\int_{0}^{\infty}gp(t,V_F,g)dg=\frac{1}{V_F}\int_{0}^{\infty}g \left(\sum_{k=-\infty}^{+\infty}p_k(t,g)\right)dg.\\
	    &=\sum_{k=-\infty}^{+\infty}\frac{1}{V_F}\int_{0}^{\infty}gp_k(t,g)dg=\sum_{k=-\infty}^{+\infty}N_k(t),\label{tmp-sum-N}
	\end{align} where we define 
	\begin{equation}
	    N_k(t):=\frac{1}{V_F}\int_{0}^{\infty}gp_k(t,g)dg,\quad k\in\mathbb{Z},
	\end{equation} to be the contribution to the firing rate $N(t)$ from the $k$-th mode. Sufficient decay of $N_k(t)$ in $k$ will be shown later which justifies the exchange of sum and integral from \eqref{tmp-sum-N-0} to \eqref{tmp-sum-N}. 
	
	We aim to control those extra $N_k$ terms and to adapt the framework in Section \ref{sc:3.2}.
	
	Let us recall the structure of solutions as in Lemma \ref{lemma-explicit-solu}. We rewrite the formula for each mode \eqref{explicit——k}, explicitly plugging out the decay factor,
	\begin{equation}\label{explicit-k-2}
	    	p_k(t,g)=e^{-k^2\pvf^2D(t)}e^{ik(\frac{2\pi}{V_F})g}(p_{t,k}\ast \bar{G}_{t,k})(g).
	\end{equation} And recall that $\bar{G}_{t,k}$ is a Gaussian multiplied by a phase factor \eqref{def-barG},
	\begin{align}\notag
	    \bar{G}_{t,k}(z)&=\frac{1}{\sqrt{2\pi C(t)}}\exp\left(-\frac{(z-B(t))^2}{2C(t)}+ik\frac{2\pi}{V_F}\myBtwo(t,z)\right)\\
	    &=\frac{1}{\sqrt{2\pi C(t)}}\exp\left(-\frac{(z-B(t))^2}{2C(t)}-ik\pvf(z-B(t))\frac{\int_{0}^{t}e^{s-t}a(s)ds}{\int_{0}^{t}e^{2(s-t)}a(s)ds}-ik\pvf\int_{0}^{t}g_{in}(s)\right),\label{bar-G-full}
	\end{align} where we plug in \eqref{def-B2}, the definition of $\myBtwo$.
	
	Then we examine the dynamics of $B(t),C(t)$, similarly to Section \ref{sc:3.2}. Recall the definition of $B(t),C(t)$ \eqref{eq:BC}
		\begin{equation*}
	\begin{aligned}
		B(t)=\int_{0}^{t}e^{-(t-s)}g_{in}(s)ds=\int_{0}^{t}e^{-(t-s)}(g_0+g_1N(s))ds.\\
	C(t)=2\int_{0}^{t}e^{-2(t-s)}a(s)ds=2\int_{0}^{t}e^{-2(t-s)}(a_0+a_1N(s))ds.
	\end{aligned}
	\end{equation*} And we take derivatives w.r.t $t$ to find that $B(t),C(t)$ satisfies a similar ODE system as in \eqref{dynamics-BC}, which we can rewrite as a perturbation of \eqref{eq:system-1},
	\begin{equation}\label{dynamics-BC-vg}
\begin{aligned}
\frac{dB(t)}{dt}&=g_0+g_1N(t)-B(t)=g_0+g_1N(B,C)-B+g_1\eps(t),\\
\frac{dC(t)}{dt}&=2(a_0+a_1N(t)-C(t))=2(a_0+a_1N(B,C)-C(t))+2a_1\eps(t),
\end{aligned}
\end{equation} where $\eps(t)$ is defined as
\begin{equation}\label{def-eps-vg}
    \eps(t):=N(t)-N(B(t),C(t))=\sum_{k=-\infty}^{+\infty}N_k(t)-N(B(t),C(t)).
\end{equation}Here $N(B(t),C(t))$, defined in \eqref{def-N-bc}, is the firing rate of a Gaussian with mean $B(t)$ and variance $C(t)$.

The estimate of $\eps(t)$ is the key. Here the challenge is that it is difficult to show, at least directly, that $\eps(t)$ converges to zero as time evolves, in contrast to the exponential decay \eqref{eq:decay-eps} in Section \ref{sc:3.2}. The trick, is to first derive a control involving the firing rate $N(B(t),C(t))$, which is stated in the following Proposition \ref{prop:eps-v}. 
\begin{proposition}\label{prop:eps-v}
For the voltage-conductance model \eqref{eq:nonlinear-toy}, $\eps(t)$ defined in \eqref{def-eps-vg} satisfies the following estimate for $t\geq 1$,
    \begin{align}\label{key-eps}
        |\eps(t)|\leq Ce^{-d^*t}\bigl(1+N(B(t),C(t))\bigr),\quad t\geq 1,
    \end{align}where $d^*=\min(1,\pvf^2a_0)>0$.
\end{proposition}
Combine Proposition \ref{prop:eps-v} with the dynamics of $(B(t),C(t))$ in \eqref{dynamics-BC-vg}, we can eventually show $\eps(t)$ goes to zero when $g_1/V_F<1$. This is different from Section \ref{sc:3.2}, where we prove the exponential decay of \eqref{eq:decay-eps} without resorting to the dynamics of $(B(t),C(t))$. Actually, in that case $\eps(t)$ goes to zero even in the case $g_1/V_F\geq1$, when the firing rate $N(t)$ itself diverges to infinity.

\begin{proof}[Proof of Proposition \ref{prop:eps-v}]
We split $|\eps(t)|$ into two parts as follows,
\begin{equation}\label{p5-1-pf}
    |\eps(t)|\leq  |N_0(t)-N(B(t),C(t))|+\sum_{k=-\infty,k\neq0}^{k=+\infty}|N_k(t)|.
\end{equation} 
The estimate of the first term in \eqref{p5-1-pf}, $|N_0(t)-N(B(t),C(t))|$ is similar to that in Section \ref{sc:3.2}. Using the explicit formula \eqref{explicit-k-2}, we write
\begin{align*}
    N_0(t)&=\frac{1}{V_F}\int_{\mathbb{R}}g_+p_0(t,g)dg=\frac{1}{V_F}\int_{\mathbb{R}}g_+\int_{\mathbb{R}}p_{t,0}(y)\bar{G}_{t,0}(g-y)dy\\&=\int_{\mathbb{R}}p_{t,0}(y)dy\int_{\mathbb{R}}\frac{1}{V_F}g_+\bar{G}_{t,0}(g-y)dg\\&=\int_{\mathbb{R}}p_{t,0}(y)N(B(t)+y,C(t))dy.
\end{align*} In the lase line, we use that $G_{t,0}(\cdot-y)$ is a Gaussian with variance $C(t)$ and mean $B(t)+y$ from the formula \eqref{bar-G-full}. On the other hand recall that \eqref{shrinkage} $p_{t,0}(y)=e^tp_{0,0}(e^ty)$ is a shrinkage of $p_{0,0}(y)$, and note that $p_{0,0}$ is a probability density from
\begin{equation*}
    \int_{\mathbb{R}}p_{0,0}(y)dy=\int_{\mathbb{R}}dy\int_{0}^{V_F}p_{\text{init}}(v,y)dv=1.
\end{equation*} Therefore we also expect $p_{t,0}$ goes to the Dirac measure. We proceed to estimate $|N_0(t)-N(B(t),C(t))|$ following the strategy in Section \ref{sc:3.2}: 
\begin{align*}\notag
|N_0(t)-N(B(t),C(t))|&=\left|\int_{\mathbb{R}}(N(B(t)+y,C(t))-N(B(t),C(t))p_{t,0}(y)dy\right|\\&\leq\int_{\mathbb{R}}|N(B(t)+y,C(t))-N(B(t),C(t)|p_{t,0}(y)dy.
\end{align*} Using the estimate on $\frac{\p N}{\p b}$ in Lemma \ref{lemma:N-der}, we get
\begin{align}\notag
    |N_0(t)-N(B(t),C(t))|&\leq\frac{1}{V_F}\int_{\mathbb{R}}|y|p_{t,0}(y)dy=\frac{1}{V_F}e^{-t}\int_{\mathbb{R}}|y|p_{0,0}(y)dy\\&=\frac{1}{V_F}e^{-t}\int_{\mathbb{R}}\int_{0}^{V_F}|y|p_{\text{init}}(v,y)dvdy\leq Ce^{-t}\label{p5-2-pf}
\end{align}

Now it remains to estimate the second term $\sum_{k=-\infty,k\neq0}^{k=+\infty}|N_k(t)|$ in \eqref{p5-1-pf}, and we shall use the decay of non-zero modes. First, by the explicit formula \eqref{explicit-k-2}, we have
\begin{align*}
    V_FN_k(t)&=\int_{\mathbb{R}}g_+p_{k}(t,g)dg=e^{-k^2\pvf^2D(t)}\int_{\mathbb{R}}g_+dg\int_{\mathbb{R}}p_{t,k}(y)e^{ik(\frac{2\pi}{V_F})g}\bar{G}_{t,k}(g-y)dy\\
    &=e^{-k^2\pvf^2D(t)}\int_{\mathbb{R}}dyp_{t,k}(y)\int_{\mathbb{R}}e^{ik(\frac{2\pi}{V_F})g}g_+\bar{G}_{t,k}(g-y)dg.
    \end{align*}Therefore
    \begin{align}\notag
       e^{k^2\pvf^2D(t)} |V_FN_k(t)|&\leq \int_{\mathbb{R}}dy|p_{t,k}(y)|\int_{\mathbb{R}}|e^{ik(\frac{2\pi}{V_F})g}g_+\bar{G}_{t,k}(g-y)|dg\\\notag&=\int_{\mathbb{R}}dy|p_{t,k}(y)|\int_{\mathbb{R}}|g_+\bar{G}_{t,k}(g-y)|dg\\&=\int_{\mathbb{R}}dy|p_{t,k}(y)|\int_{\mathbb{R}}g_+\bar{G}_{t,0}(g-y)dg.\label{p5-tmp1-pf}
    \end{align} In the last line we use $|\bar{G}_{t,k}|=\bar{G}_{t,0}$ from \eqref{bar-G-full}. Now we can rewrite \eqref{p5-tmp1-pf} in terms of $N(b,c)$ as our treatment with $N_0$,
    \begin{align}
       e^{k^2\pvf^2D(t)} |V_FN_k(t)|\leq \int_{\mathbb{R}}V_FN(B(t)+y,C(t))|p_{t,k}(y)|dy.\label{p5-tmp2-pf}
    \end{align} 
    On the other hand, recall again \eqref{shrinkage} $p_{t,k}(y)=e^tp_{0,k}(e^ty)$, therefore $p_{t,k}$ and $p_{0,k}$ have a same $L^1$ norm, which we denote as $c_k$. By the definition of $p_{0,k}$ \eqref{shrinkage}, $c_k$ is less than one:
    \begin{align}\notag
        c_k&:=\int_{\mathbb{R}}|p_{0,k}(y)|dy=\int_{\mathbb{R}}dy\left|e^{-ik\frac{2\pi}{V_F}g}\int_{0}^{V_F}p_{\text{init}}(v,g)e^{-ikv\frac{2\pi}{V_F}}dv.\right|\\&\leq\int_{\mathbb{R}}\int_{0}^{V_F}|p_{\text{init}}(v,g)|dvdg=1.
    \end{align} Then we continue the estimate in \eqref{p5-tmp2-pf}, by comparing with $c_kN(B(t),C(t))$
    \begin{align}\notag
               e^{k^2\pvf^2D(t)} |V_FN_k(t)|&\leq \int_{\mathbb{R}}V_F|(N(B(t)+y,C(t))-N(B(t),C(t)))p_{t,k}(y)|dy+c_kV_FN(B(t),C(t))\\&\leq \int_{\mathbb{R}}|yp_{t,k}(y)|dy+c_kV_FN(B(t),C(t)),\label{p5-tmp3-pf}
    \end{align} where we use Lemma \ref{lemma:N-der} the bound on $\frac{\p N}{\p b}$ again. For the first term in \eqref{p5-tmp3-pf}, we estimate in the same way as for $N_0$ in \eqref{p5-2-pf},
    \begin{align*}
        \int_{\mathbb{R}}|yp_{t,k}(y)|dy&=e^{-t}\int_{\mathbb{R}}|yp_{0,k}(y)|dy\\&\leq e^{-t}\int_{\mathbb{R}}\int_{0}^{V_F}|y|p_{\text{init}}(v,y)dvdy.
    \end{align*} Combine this with \eqref{p5-tmp3-pf}, we deduce
    \begin{align*}
        e^{k^2\pvf^2D(t)} |V_FN_k(t)|&\leq Ce^{-t}+c_kV_FN(B(t),C(t))\\ &\leq C+c_kV_FN(B(t),C(t)).
    \end{align*}
    Therefore we get 
    \begin{equation}\label{p5-p3-1-pf}
        |N_k(t)|\leq Ce^{-k^2\pvf^2D(t)}(1+N(B(t),C(t))).
    \end{equation}
    Sum \eqref{p5-p3-1-pf} up for $k\neq0$, we have
    \begin{align}\notag
        \sum_{k=-\infty,k\neq0}^{k=+\infty}|N_k(t)|&\leq C\sum_{k=-\infty,k\neq0}^{k=+\infty}e^{-k^2\pvf^2D(t)}(1+N(B(t),C(t)))\\
        &\leq C\sum_{k=1}^{k=+\infty}e^{-k\pvf^2D(t)}(1+N(B(t),C(t)).\label{p5-tmp4-pf}
    \end{align} By $D(t)\geq a_0(t-\frac{e^t-1}{e^t+1})$ in Lemma \ref{lemma-D(t)}, we have $D(t)\geq a_0(1-\frac{e-1}{e+1})>0$ for $t\geq 1$. Therefore for $t\geq 1$ we obtain from \eqref{p5-tmp4-pf}
    \begin{align}\notag
        \sum_{k=-\infty,k\neq0}^{k=+\infty}|N_k(t)|&\leq C\frac{e^{-\pvf^2D(t)}}{1-e^{-\pvf^2D(t)}}(1+N(B(t),C(t))\\&\leq C{e^{-\pvf^2D(t)}}(1+N(B(t),C(t)),\quad t\geq 1.\notag
    \end{align}By Lemma \ref{lemma-D(t)} again we deduce
    \begin{align}\label{p5-p3-2-pf}
        \sum_{k=-\infty,k\neq0}^{k=+\infty}|N_k(t)|\leq C{e^{-\pvf^2a_0t}}(1+N(B(t),C(t))),\quad t\geq 1.
    \end{align}
    
    Plugging \eqref{p5-p3-2-pf} and the estimate for $N_0$ in \eqref{p5-2-pf} into \eqref{p5-1-pf}, we finally conclude our estimate on $\eps(t)$,
	\begin{align}\notag
		    |\eps(t)|&\leq  |N_0(t)-N(B(t),C(t))|+\sum_{k=-\infty,k\neq0}^{k=+\infty}|N_k(t)|\\&\leq Ce^{-t}+C{e^{-\pvf^2a_0t}}(1+N(B(t),C(t))\notag
		    \\&\leq Ce^{-d^*t}(1+N(B(t),C(t))),\quad t\geq 1,
	\end{align}where $d^*=\min(1,\pvf^2a_0)$.
		\end{proof}
		Now we begin the proof of Theorem \ref{thm:full-convergence}. First we prove that when $0<g_1/V_F<1$, the solution converges to the unique steady state.
	\begin{proof}[Proof of Theorem \ref{thm:full-convergence}, Case $0<g_1/V_F< 1$]
	WLOG, we consider the case $V_F=1$ and $0<g_1<1$, otherwise we rescale. As in Section \ref{sc:3.2}, we shall first show that $(B(t),C(t))$ is uniformly bounded. Recall the dynamics of $(B(t),C(t))$ in \eqref{dynamics-BC-vg}
	\begin{equation*}
\begin{aligned}
\frac{dB(t)}{dt}&=g_0+g_1N(t)-B(t)=g_0+g_1N(B,C)-B+g_1\eps(t),\\
\frac{dC(t)}{dt}&=2(a_0+a_1N(t)-C(t))=2(a_0+a_1N(B,C)-C(t))+2a_1\eps(t),
\end{aligned}
\end{equation*} 
	 By Proposition \ref{prop:eps-v} and $0\leq N(b,c)\leq b_++\sqrt{c}$ in Lemma \ref{lemma-bound-N}, we deduce for $t\geq 1$
	\begin{align}\notag
	    |\eps(t)|&\leq Ce^{-d^*t}(1+N(B,C))\\&\leq Ce^{-d^*t}(1+|B|+\sqrt{C}).\label{eps-bound-2}
	\end{align} Following the same calculation for Liapounov functional as in \eqref{tmp-deri-ode-2} we obtain
	\begin{align}\label{tmp4-tmp-deri-ode-2}
    \frac{d}{dt}[\frac{1}{2}(B^2+\delta_0 C^2)]&\leq C-C_0(B(t)^2+C(t)^2)+C\eps(t)(|B(t)|+|C(t)|)
\end{align} 
Using \eqref{eps-bound-2} and Young's inequality we get for $t\geq 1$,
\begin{align*}
    \eps(t)(|B(t)|+|C(t)|)&\leq Ce^{-d^*t}(1+|B(t)|+\sqrt{C(t)})(|B(t)|+|C(t)|)\\&\leq Ce^{-d^*t}+Ce^{-d^*t}(B(t)^2+C(t)^2).
\end{align*}Plug this in \eqref{tmp4-tmp-deri-ode-2}, and we have 
\begin{equation}
     \frac{d}{dt}[\frac{1}{2}(B^2+\delta_0 C^2)]\leq C-(C_0-Ce^{-d^*t})(B(t)^2+C(t)^2).
\end{equation} Therefore we derive that $(B(t),C(t))$ is uniformly bounded in time. Then using \eqref{eps-bound-2} again, we get that $\eps(t)$ goes to zero exponentially. Now applying Theorem \ref{thm:auto} to the asymptotically autonomous system \eqref{dynamics-BC-vg}, we conclude that $(B(t),C(t))$ converges to $(b^*,c^*)$ as time evolves.

Now we estimate the distance between the solution and the steady state. By triangle inequality we have
\begin{align*}
\|p(t,v,g)-\frac{1}{V_F}\frac{1}{\sqrt{2\pi c^*}}e^{-\frac{(g-b^*)^2}{2c^*}}\|_{L^1((0,V_F)\times\mathbb{R})}&\leq \|p(t,v,g)-\frac{1}{V_F}p_0(t,g)\|_{L^1((0,V_F)\times\mathbb{R})}\\&+
 \|\frac{1}{V_F}G_{t,0}(g)-\frac{1}{V_F}p_0(t,g)\|_{L^1((0,V_F)\times\mathbb{R})}\\&+ \|\frac{1}{V_F}G_{t,0}(g)-\frac{1}{V_F}\frac{1}{\sqrt{2\pi c^*}}e^{-\frac{(g-b^*)^2}{2c^*}}\|_{L^1((0,V_F)\times\mathbb{R})}.
\end{align*}
The first terms goes to zero by Theorem \ref{thm:decay-v}, the last term goes to zero since we have shown the convergence of $(B(t),C(t))$ to $(b^*,c^*)$. For the middle term, it is the same as in the last step of the proof in Theorem \ref{thm:v-homo-convergence}. Actually recall \eqref{fourier-v} that the initial value of zero order mode
\begin{equation*}
    p_{0,0}(g)=\int_{0}^{V_F}p_{\text{init}}(v,g)dv,\quad 
\end{equation*} is non-negative with unit mass $\int_{\mathbb{R}}p_{0,0}(g)dg=\int_{\mathbb{R}}\int_{0}^{V_F}p_{\text{init}}(v,g)dvdg=1$.
Then we can estimate in the same way as Section \ref{sc:3.2}
\begin{align*}
\|\frac{1}{V_F}G_{t,0}(g)-\frac{1}{V_F}p_0(t,g)\|_{L^1((0,V_F)\times\mathbb{R})}&=\|p(t,\cdot)-G_{t,0}(\cdot)\|_{L^1({\mathbb{R}})}\\
&=\|(p_{t,0}*G_{t,0})(\cdot)-G_{t,0}(\cdot)\|_{L^1(\mathbb{R})}\\
&\leq\|\p_{g}G_{t,0}\|_{L^1(\mathbb{R})} \int_{\mathbb{R}}|y|p_{t,0}(y)dy=e^{-t}\|\p_{g}G_{t,0}\|_{L^1(\mathbb{R})}\int_{\mathbb{R}}|y|p_{0,0}(y)dy\\&\leq e^{-t}C.
\end{align*}
In the last inequality, we use a uniform bound on $\|\p_gG_{t,0}\|_{L^1(\mathbb{R})}$ implied by the convergence of $(B(t),C(t))$, as in the proof of Theorem \ref{thm:v-homo-convergence}.
\end{proof}
	
	Next, we treat the case $g_1/V_F\geq 1$ and prove the firing rate $N(t)$ of \eqref{eq:nonlinear-toy} diverges to infinity as time evolves. The key estimate is also Proposition \ref{prop:eps-v}.
	
	\begin{proof}[Proof of Theorem \ref{thm:full-convergence}, Case $g_1/V_F\geq 1$]
WLOG we consider the case $V_F=1$ and $g_1\geq 1$. By Proposition \ref{prop:eps-v} for $t\geq 1$ we have $|N(t)-N(B,C)|\leq Ce^{-d^*t}(1+N(B(t),C(t)))$, therefore
\begin{equation*}
    N(t)\geq (1-Ce^{-d^*t})N(B(t),C(t))-Ce^{-d^*t},\quad t\geq 1,
\end{equation*} and it suffices to show that $N(B,C)$ goes to infinity. Since $N(B(t),C(t))\geq \frac{1}{V_F}B(t)$ by Lemma \ref{lemma-bound-N}, it suffices to show that $B(t)$ goes to infinity. Recall the dynamics of $B(t)$ \eqref{dynamics-BC-vg},
\begin{align*}
    \frac{dB(t)}{dt}&=g_0+g_1N(B,C)-B(t)+\eps(t)\\&\geq g_0+g_1N(B,C)-B(t)-Ce^{-d^*t}-Ce^{-d^*t}N(B,C).
\end{align*}If $g_1>1$, then there exists $T$ such that for $t\geq T$, $Ce^{-d^*t}<\frac{g_1-1}{2}$ and $Ce^{-d^*t}<\frac{g_0}{2}$, then
\begin{align*}
    \frac{dB(t)}{dt}\geq \frac{1}{2}g_0+\frac{g_1-1}{2}N(B,C)\geq  \frac{1}{2}g_0+\frac{g_1-1}{2}B(t),
\end{align*} which implies $B(t)$ goes exponentially to infinity by Gronwall's inequality.

When $g_1=1$ we shall follow the $\eps$-$\delta$ definition of the limit. For every $0<r<\min(g_1,g_0)$, there exists $T=T(r)$ such that $Ce^{-d^*t}<r$ for $t\geq T$. Therefore for $t\geq T$
\begin{align*}
    \frac{dB(t)}{dt}&\geq (g_0-r)+(1-r)N(B,C)-B(t)\\&\geq (g_0-r)-rB(t).
\end{align*} Then we deduce that there exists $T_1=T_1(r)$ such that for all $t\geq T_1(r)$, $B(t)\geq \frac{1}{2}\frac{g_0-r}{r}$. 

In summary, for every $0<r<\min(g_1,g_0)$ there exists $T_1=T_1(r)$ such that $B(t)\geq \frac{1}{2}\frac{g_0-r}{r}$ for $t\geq T_1(r)$. Therefore by the the $\eps$-$\delta$ definition of the limit, we conclude
\begin{equation}
    \lim_{t\rightarrow\infty}B(t)=+\infty.
\end{equation}
\end{proof}

	\section{The fast conductance limit}\label{sc:4fcl}
	We introduce a timescale ratio parameter $\veps>0$ into the model \eqref{eq:nonlinear-toy}, which is the ratio of the timescale of conductance with respect to the timescale of voltage. Precisely we consider the following model
	\begin{equation}\label{eq:nonlinear-toy-veps}
\p_t\pe+g\p_v\pe=\frac{1}{\veps}(\p_g((-g^{\veps}_{\text{in}}(t)+g)\pe+a^{\veps}(t)\p_g\pe)), \quad v\in(0,V_F),\ g\in\mathbb{R},\ t>0.\end{equation} whose boundary condition, initial value and definitions of $a^{\veps}(t),g_{\text{in}}^{\veps}(t),N^{\veps}(t)$ are the same as in \eqref{nonlinear-toy-bc},\eqref{nonlinear-toy-IC} and \eqref{eq:para-general} for the simplified model \eqref{eq:nonlinear-toy}. Here we just use the superscript $\veps$ to stress the dependence on $\veps$.
	
	In this section, we study the effect of the time ratio parameter $\veps$ and the fast conductance limit $\veps\rightarrow0^+$. We derive a limit model and analyze its behavior. Our motivation is twofold. Physically, the timescale of conductance $g$ is much smaller than that of the voltage $v$, which implies that $\veps$ is very small. This motives the study of the effect of a small $\veps$ and the limit $\veps\rightarrow0^+$. Moreover, the fast conductance limit links the kinetic model \eqref{eq:nonlinear-toy} to a ``macro'' model with describes the voltage $v$ only. This provides an another way of model reduction, in contrast to our analysis in Section \ref{sec:2} and \ref{sc:3longtime}, where we reduce the dynamics to $g$ direction only, motivated by the long time behavior  -- the convergence to $v$-homogeneous problem in Theorem \ref{thm:decay-v}. Such an alternative way of model reduction may provide new insights, especially on periodic solutions, of the original model \eqref{eq:original}. 
	
		Let's derive the fast conductance limit model formally first. Consider $\veps\rightarrow0^+$ in \eqref{eq:nonlinear-toy-veps}. We collect the $O(\frac{1}{\veps})$ terms and get the following ``quasi-steady'' equation in $g$ direction
\begin{equation}\label{quasi-steady}
\p_g[(g-g_{\text{in}}(t))p-a(t)\p_gp]=0,\quad t>0,v\in(0,V_F),g\in\mathbb{R}.
\end{equation}Solving \eqref{quasi-steady} we deduce that for every fixed $v$, the profile in the $g$ direction is a Gaussian whose mean is $g_{\text{in}}(t)$ and variance is $a(t)$. Therefore for some $\rho(t,v)$ we writes 
\begin{equation}\label{der-fcl-1}
    p(t,v,g)=\rho(t,v)\mathcal{G}(g;g_{\text{in}}(t),a(t)).
\end{equation} Here we use the following notation for Gaussian: $\mathcal{G}(g;g_{\text{in}}(t),a(t))=\frac{1}{\sqrt{2\pi a(t)}}e^{-\frac{(g-g_{\text{in}}(t))^2}{2a(t)}}$. Actually, $\rho(t,v)$ is the marginal density in $v$, and therefore a probability density on $[0,V_F]$, which can be checked by integrating \eqref{der-fcl-1} over $[0,V_F]\times\mathbb{R}$.

Then we plug \eqref{der-fcl-1} into $O(1)$ terms of \eqref{eq:nonlinear-toy-veps}, and obtain
\begin{equation}\label{deri-formal-fcl-1}
    \p_t(\rho(t,v)\mathcal{G}(g;g_{\text{in}}(t),a(t)))+\left(g\mathcal{G}(g;g_{\text{in}}(t),a(t))\right)\p_v\rho=0.
\end{equation}
Integrating \eqref{deri-formal-fcl-1} in $g$ on $\mathbb{R}$, using the zeroth and first moment of $\mathcal{G}$ 

$$\int_{\mathbb{R}}\mathcal{G}(g;g_{\text{in}}(t),a(t))dg=1,\quad\int_{\mathbb{R}}g\mathcal{G}(g;g_{\text{in}}(t),a(t))dg=g_{\text{in}}(t),$$ we deduce a transport equation for $\rho$
\begin{equation}\label{eq:fcl}
        \p_t\rho+g_{\text{in}}(t)\p_v\rho=0,\quad v\in(0,V_F),t>0.
\end{equation}

For the boundary condition, we similarly integrate the boundary condition \eqref{nonlinear-toy-bc} in $g$  and get
\begin{equation}\label{eq:fcl-bc}
    \rho(t,0)=\rho(t,V_F),\quad t>0.
\end{equation}

Finally we derive a formula of $N(t)$ in terms of $\rho$, by which we can represent $g_{\text{in}}(t)$ in \eqref{eq:fcl}. Plugging \eqref{der-fcl-1} into the definition of firing rate \eqref{eq:nonliner-toy-fire}, we obtain
\begin{align}\notag
N(t)&=\int_{0}^{+\infty}gp(t,V_F,g)dg\\\notag&=\rho(t,V_F)\int_{0}^{+\infty}g\mathcal{G}(g;g_{\text{in}}(t),a(t))dg\\&=\rho(t,V_F)\mathscr{N}(g_{\text{in}}(t),a(t)),
\end{align}where the function $\mathscr{N}(b,c)$ is defined as 
\begin{equation}\label{def-fcl-N}
    \mathscr{N}(b,c):=\int_{0}^{+\infty}g\mathcal{G}(g;b,c)dg=\int_{0}^{+\infty}g\frac{1}{\sqrt{2\pi c}}e^{-\frac{(g-b)^2}{2c}}dg,\quad b\in\mathbb{R},c>0.
\end{equation}
Since $g_{\text{in}}(t)$ and $a(t)$ depend on $N(t)$ \eqref{eq:para-general}, we actually have derived a nonlinear equation for $N(t)$
\begin{equation}\label{fcl-N-g+}
    N(t)=\rho(t,V_F)\mathscr{N}(g_0+g_1N(t),a_0+a_1N(t)).
\end{equation}
Now we have derived the fast conductance limit model, summarized as 
\begin{equation}\label{eq:fcl-main}
    \begin{cases} \p_t\rho+g_{\text{in}}(t)\p_v\rho=0,\quad v\in(0,V_F),t>0.\\
        \rho(t,0)=\rho(t,V_F),\quad t>0.\\
     N(t)=\rho(t,V_F)\mathscr{N}(g_0+g_1N(t),a_0+a_1N(t)),\quad 0\leq N(t)<\infty,
    \end{cases}
\end{equation}where $g_{\myin}(t)=g_0+g_1N(t)$ as in \eqref{eq:para-general}.

 The fast conductance limit model \eqref{eq:fcl-main} is a nonlinear transport equation in $v$. A key feature is that $N(t)$ solves a nonlinear equation \eqref{fcl-N-g+}, which depends on the boundary value $\rho(t,V_F)$. When $0\leq \rho(t,V_F)<1/g_1$, \eqref{fcl-N-g+} has a unique solution. But if $\rho(t,V_F)\geq 1/g_1$, then there is no $N(t)\in[0,+\infty)$ satisfying \eqref{fcl-N-g+}, which we interpret as the blow-up of $N(t)$. Precisely we have the following lemma on the nonlinear equation \eqref{fcl-N-g+}.
\begin{lemma}\label{lemma:fcl-N}
Consider the following equation in $N^*$ with a parameter $\bar{\rho}\geq 0$,
\begin{equation}\label{eq-lm-N}
    N^*=\bar{\rho}\mathscr{N}(g_0+g_1N^*,a_0+a_1N^*),
\end{equation} where $\mathscr{N}$ is defined in \eqref{def-fcl-N}. 
When $\bar{\rho}\geq 1/{g_1}$, there is no solution $N^*$ for \eqref{eq-lm-N}. When $0\leq\bar{\rho}<{1}/{g_1}$, there is a unique solution $N^*$ for \eqref{eq-lm-N}. Moreover, denote the solution in the latter case as $N^*(\bar{\rho})$, and then we have
\begin{equation}
    \lim_{\bar{\rho}\rightarrow ({1}/{g_1})^{-}}N^*(\bar{\rho})=+\infty.
\end{equation}
\end{lemma}

In view of Lemma \ref{lemma:fcl-N} we shall add a ``physical'' restriction on $\rho(t,V_F)$
\begin{equation}\label{fcl-bound-vf}
    0\leq \rho(t,V_F)< \frac{1}{g_1},
\end{equation} which is crucial for the well-posedness of \eqref{eq:fcl-main}. By Lemma \ref{lemma:fcl-N}, under \eqref{fcl-bound-vf} we can uniquely solve $N(t)$ from \eqref{fcl-N-g+}, otherwise the firing rate blows up. Before we give the proof of Lemma \ref{lemma:fcl-N}, let us present the arrangements of this section.

	In Section \ref{sc:4.1deri}, with some essential assumptions on the solution, we give a rigorous derivation of the fast conductance limit model \eqref{eq:fcl-main}. Through semi-explicit formulas similar to Lemma \ref{lemma-explicit-solu}, we can see clearly the effect of $\veps$ as well as what happens to each Fourier mode when $\veps$ goes to zero. 
	
	Then in Section \ref{sc:4.2analysis} we give a sharp characterization of the long time behavior of the fast conductance limit model \eqref{eq:fcl-main}: depending on the $L^{\infty}$ norm of the initial value, either the firing rate $N(t)$ blows up in finite time or the solution global exists. In the latter case, the solution is periodic  in time.

	\begin{proof}[Proof of Lemma \ref{lemma:fcl-N}]
	The nonlinear function $\mathscr{N}(b,c)$ defined in \eqref{def-fcl-N} has been studied in the context of the ODE \eqref{eq:system-1} in Section \ref{sc:ode3.1}. Actually $\mathscr{N}(b,c)=V_FN(b,c)$, where $N(b,c)$ is defined in \eqref{def-N-bc}. Then the equation \eqref{eq-lm-N} becomes finding a $N^*$ such that
	\begin{equation}\label{eq-lm-N2}
	    N^*=\bar{\rho}V_FN(g_0+g_1N^*,a_0+a_1N^*)\geq 0.
	\end{equation}
	By Lemma \ref{lemma-bound-N}, $V_FN(b,c)\geq b$
	we deduce from \eqref{eq-lm-N2}
	\begin{equation}
	    N^*\geq g_0\bar{\rho}+g_1\bar{\rho} N^*.
	\end{equation} Therefore if $ g_1\bar{\rho} \geq 1$, there is no solution $N^*\geq 0$. While when $ g_1\bar{\rho}<1$,
	\begin{equation}
	    N^*\geq \frac{1}{1-g_1\bar{\rho} }g_0\bar{\rho},
	\end{equation} which goes to $+\infty$ as $\bar{\rho}$ goes to $({1}/{g_1})^{-}$. Indeed, this argument is in analogy to Corollary \ref{cor:ode-steady}.

	It remains to check the existence and uniqueness of $N^*$ when $g_1\bar{\rho} <1$. We shall again use the result from the ODE system \eqref{eq:system-1}. Let $b^*=g_0+g_1N^*,c^*=a_0+a_1N^*$ then \eqref{eq-lm-N2} is equivalent to find $(b^*,c^*)$ such that 
	\begin{equation}
	    \begin{cases}
	        g_0+(g_1\bar{\rho}V_F)N(b^*,c^*)-b^*=0,\\
	        a_0+(a_1\bar{\rho}V_F)N(b^*,c^*)-c^*=0,
	    \end{cases}
	\end{equation} which is equivalent to finding a steady state of the ODE system \eqref{eq:system-1}, with $g_1$ replaced by $g_1\bar{\rho}V_F$ and $a_1$ replaced by $a_1\bar{\rho}V_F$.  By Proposition \ref{prop:ode-steady}, we get the existence and uniqueness by checking  $(g_1\bar{\rho}V_F)/V_F=g_1\bar{\rho}<1$.
	\end{proof}
	
	If we define $\mathscr{N}(b,c):=b$ instead of \eqref{def-fcl-N}, i.e., extending the integral in $g$ from $\mathbb{R}^+$ to $\mathbb{R}$, then the equation \eqref{eq-lm-N} becomes
	\begin{equation*}
	    N^*=g_0\bar{\rho}+g_1\bar{\rho} N^*.
	\end{equation*} Clearly in this case, we need $0\leq \bar{\rho}<1/g_1$ to get a non-negative firing rate. Moreover, $N^*(\bar{\rho})=\frac{g_0\bar{\rho}}{1-g_1\bar{\rho}}$ goes to infinity as $\bar{\rho}$ goes to $({1}/{g_1})^{-}$. Lemma \ref{lemma:fcl-N} extends these facts to the nonlinear function \eqref{def-fcl-N}.  

	\subsection{Convergence to the fast conductance limit}\label{sc:4.1deri}

	In this section we derive the fast conductance limit model \eqref{eq:fcl-main} rigorously, under some essential assumptions on the solution.
	
	First let's derive the solution formulas for \eqref{eq:nonlinear-toy-veps}, like Lemma \ref{lemma-explicit-solu} for \eqref{eq:nonlinear-toy}. We still consider the Fourier expansion in $v$:
		\begin{equation}\label{eps-fourier-v}
\begin{aligned}
 \pe(t,v,g)&=\frac{1}{V_F}\sum_{k=-\infty}^{+\infty}\pe_k(t,g)e^{ikv\frac{2\pi}{V_F}},\\
\pe_k(t,g):&=\int_{0}^{V_F}\pe(t,v,g)e^{-ikv\frac{2\pi}{V_F}}dv,\quad k\in\mathbb{Z}.
\end{aligned}
	\end{equation}And similarly, plugging the expansion \eqref{eps-fourier-v} in \eqref{eq:nonlinear-toy-veps}, we get that each $\pe_k$ satisfies
	\begin{equation}\label{eps-eq:fourier-each-k}
\veps\p_t \pe_k+i\veps k\frac{2\pi}{V_F}g\pe_k=\p_g\bigl((-g_{\text{in}}^{\veps}(t)+g)p_k+a^{\veps}(t)\p_g\pe_k\bigr),\quad g\in\mathbb{R},\ t>0.
	\end{equation}
	
	By the following change of variable \begin{equation}\label{change-of-v}
	    \tau=t/\veps,\quad u=v/\veps,\quad U_F:=V_F/\veps,
	\end{equation}we reduce \eqref{eq:nonlinear-toy-veps} to the case $\veps=1$ \eqref{eq:nonlinear-toy}, in new time variable $\tau$ and voltage variable $u$. Then adapting Lemma \ref{lemma-explicit-solu}, we get solution formulas for \eqref{eq:nonlinear-toy-veps} in Lemma \ref{lemma-eps-explicit-solu} below.
	
	\begin{lemma}\label{lemma-eps-explicit-solu}
The solution $\pe_k(t,g)$ of the equation \eqref{eps-eq:fourier-each-k}, the $k$-th mode of the Fourier expansion in $v$ \eqref{eps-fourier-v} for \eqref{eq:nonlinear-toy-veps}, is given by
    	\begin{equation}\label{eps-explicit——k}
	\pe_k(t,g)=e^{i\veps k(\frac{2\pi}{V_F})g}(\pe_{t,k}\ast G^{\veps}_{t,k})(g).
	\end{equation} 
	Here $\pe_{t,k}$ is a shrinkage of $\pe_{0,k}$, which is the initial data for the $k$-th Fourier mode multiplied a shift in frequency:
	\begin{equation}\label{eps-shrinkage}
	\pe_{t,k}(y):=e^{t/\veps}\pe_{0,k}(e^{t/\veps}y),\quad \pe_{0,k}(g):=e^{-i\veps k\frac{2\pi}{V_F}g}\int_{0}^{V_F}p_{\text{init}}^{\veps}(v,g)e^{-ikv\frac{2\pi}{V_F}}dv.
	\end{equation} 
	And $G_{t,k}^{\veps}$ is a modified Gaussian with a phase factor and a decay factor, given by
	\begin{equation}\label{eps-eq:modified-Gaussian}
	G_{t,k}^{\veps}(z)=\frac{1}{\sqrt{2\pi C^{\veps}(t)}}\exp\left(-\frac{(z-B^{\veps}(t))^2}{2C^{\veps}(t)}\right)\exp\left(ik\veps\frac{2\pi}{V_F}\myBtwo^{\veps}(t,z)\right)\exp\left(- k^2\veps^2(\frac{2\pi}{V_F})^2D^{\veps}(t)\right).
	\end{equation}Here the mean $B^{\veps}(t)$ and the variance $C^{\veps}(t)$ are given by
	\begin{equation}\label{eps-eq:BC}
	\begin{aligned}
		B^{\veps}(t)=\int_{0}^{t/\veps}e^{-(t/\veps-\tilde{\tau})}g^{\veps}_{\text{in}}(\veps \tilde{\tau})d\tilde{\tau}=\int_{0}^{t}\frac{1}{\veps}e^{-(t-s)/\veps}g^{\veps}_{\text{in}}(s)ds.\\
	C^{\veps}(t)=2\int_{0}^{t/\veps}e^{-2(t/\veps-\tilde{\tau})}a^{\veps}(\veps \tilde{\tau})d\tilde{\tau}=\int_{0}^{t}\frac{2}{\veps}e^{-2(t-s)/\veps}a^{\veps}(s)ds.
	\end{aligned}
	\end{equation}
	Moreover $\myBtwo^{\veps}(t,z)$ and $D^{\veps}(t)$ are given by
	\begin{equation}\label{eps-def-B2}
	\myBtwo^{\veps}(t,z)=-(z-B^{\veps}(t))\frac{\int_{0}^{t}e^{(s-t)/\veps}a^{\veps}(s)ds}{\int_{0}^{t}e^{2(s-t)/\veps}a^{\veps}(s)ds}-\frac{1}{\veps}\int_{0}^{t}g_{\text{in}}^{\veps}(s)ds,
	\end{equation} and
	\begin{equation}\label{eps-decay-D(t)}
	D^{\veps}(t)=\frac{1}{\veps}\int_{0}^{t}a^{\veps}(s)ds-\frac{(\frac{1}{\veps}\int_{0}^{t}e^{(s-t)/\veps}a^{\veps}(s)ds)^2}{\frac{1}{\veps}\int_{0}^{t}e^{2(s-t)/\eps}a^{\veps}(s)ds}\geq 0.
	\end{equation}
\end{lemma}

With Lemma \ref{lemma-eps-explicit-solu}, we can derive the fast conductance limit model under two assumptions. First we need an assumption on initial data as follows.
\begin{assumption}\label{as:init}
(i) For different $\veps>0$ the initial data $p^{\veps}_{\init}$ is the same, i.e., 
\begin{equation}
    p^{\veps}_{\init}(v,g)=p_{\init}(v,g),\quad v\in[0,V_F],g\in\mathbb{R},\quad\veps>0.
\end{equation}
(ii) 
Let \begin{equation}\label{def-init-k}
    p_{k,\init}(g):=\int_{0}^{V_F}p_{\init}(v,g)e^{-ikv\frac{2\pi}{V_F}}dv
\end{equation} be the $k$-th Fourier coefficient in $v$ of the initial data, then we assume
\begin{equation}
    \sum_{k=-\infty}^{+\infty}\|(1+|g|)p_{k,\init}\|_{L^1(\mathbb{R})}<\infty.
\end{equation}
\end{assumption}
Moreover, we impose an assumption on the solution, that is, we assume the firing rate $N^{\veps}{(t)}$ has a limit in $C[0,T]$ for some $T>0$, as stated in the following.
\begin{assumption}\label{as:N}
For some $T>0$, and $N(t)$ in $C[0,T]$, the firing rate $N^{\veps}(t)\geq0$ has the following limit
\begin{equation}
    N^{\veps}(t)\rightarrow N(t),\quad \text{in }C[0,T],
\end{equation} 
as $\veps\rightarrow0^+$.
\end{assumption}
Assumption \ref{as:N} implies that $N^{\veps}(t)$ is uniformly bounded and the limit $N(t)\geq 0$. And as a consequence, we have well-defined limits for $g_{\text{in}}^{\veps}(t)$ and $a^{\veps}(t)$ 
\begin{equation}\label{limit-gin-a}
    \begin{aligned}
    g_{\text{in}}^{\veps}(t)=g_0+g_1N^{\veps}(t)\rightarrow g_0+g_1N(t)=g_{\text{in}}(t),\\
    a^{\veps}(t)=a_0+a_1N^{\veps}(t)\rightarrow a_0+a_1N(t)=a(t),\\
    \end{aligned}
\end{equation} in $C[0,T]$ as $\veps$ goes to zero.

Assumption \ref{as:init}-(i) can be relaxed to convergence of $p_{\init}^{\veps}$ in the norm corresponding to Assumption \ref{as:init}-(ii). Here we assume the initial data  for different $\veps>0$ is the same for simplicity and clarity. Assumption \ref{as:init}-(ii) gives the control on the regularity in $v$ and the first moment in $g$. 

Assumption \ref{as:N} is more essential, since it is imposed on the solution rather than on initial data. Actually we shall show in Theorem \ref{thm:longtime-fcl} that for the limit model, the firing rate can blow up in finite time.
Therefore Assumption \ref{as:N} does not always hold, since the limit firing rate may not be well-defined.

Now we can state the rigorous result on the fast conductance limit.
\begin{proposition}\label{Prop:fcl-limit} With assumptions \ref{as:init} and \ref{as:N}, as $\veps$ goes to zero, the solution $p^{\veps}(t,v,g)$ of \eqref{eq:nonlinear-toy-veps} converges to a solution of the fast conductance limit model \eqref{eq:fcl-main} in the following sense. For $T>0$ in Assumption \ref{as:N} and every $0<T_0<T$, as $\veps\rightarrow0^+$,
\begin{equation}
    p^{\veps}(t,v,g)\rightarrow \rho(t,v)\frac{1}{\sqrt{2\pi a(t)}}\exp\left(-\frac{(g-g_{\myin}(t))^2}{2a(t)}\right),\quad \text{in  }L^{\infty}((T_0,T);L^1((0,V_F)\times\mathbb{R})).
\end{equation}
Here $\rho(t,v)$ is a solution of the fast conductance limit model \eqref{eq:fcl-main} and it satisfies the bound \eqref{fcl-bound-vf}. 
\end{proposition}
\begin{proof}[Proof of Proposition \ref{Prop:fcl-limit}]
The idea is to pass the limit for each term in Lemma \ref{lemma-eps-explicit-solu}.

First we consider the limit of $\pe_k(t,g)$. Recall \eqref{eps-shrinkage} $\pe_{t,k}(y):=e^{t/\veps}\pe_{0,k}(e^{t/\veps}y)$. Note that
\begin{equation*}
    \pe_{0,k}(g)=e^{-i\veps k\frac{2\pi}{V_F}g}\int_{0}^{V_F}p_{\text{init}}^{\veps}(v,g)e^{-ikv\frac{2\pi}{V_F}}dv\rightarrow \int_{0}^{V_F}p_{\text{init}}(v,g)e^{-ikv\frac{2\pi}{V_F}}dv, 
\end{equation*} in $L^1(\mathbb{R})$ as $\veps\rightarrow0^+$. Also note that for a function $h(x)$ in $L^{1}(\mathbb{R})$, its shrinkage $e^{t/\veps}h(x/\veps)$ goes to $\left(\int_{\mathbb{R}}h(y)dy\right)\delta(x)$, the Dirac measure at $x=0$ multiplied by the integral of $h$, as $\veps\rightarrow0^+$ (e.g. in the sense of distribution). Therefore for $\pe_{t,k}$, we have
\begin{equation}
    \pe_{t,k} \rightarrow c_k\delta(0),\quad t>0,
\end{equation} as $\veps$ goes to $0^+$. Here $c_k$ is defined as 
\begin{equation}\label{def-ck}
c_k:=  \int_{\mathbb{R}}\int_{0}^{V_F}p_{\text{init}}(v,g)e^{-ikv\frac{2\pi}{V_F}}dv,\quad k\in\mathbb{Z}.
\end{equation} Precisely, we shall use that for a $L^1$ function $f$, as $\veps$ goes to $0^+$
\begin{equation}\label{limit-2-pek}
    \pe_{t,k}\ast f \rightarrow c_kf,\quad \text{in }L^1(\mathbb{R}), 
\end{equation} and that this convergence is uniform for a family of $f$ with a uniform $W^{1,1}(\mathbb{R})$ bound. Recall \eqref{eps-shrinkage} $\pe_{t,k}(y):=e^{t/\veps}\pe_{0,k}(e^{t/\veps}y)$, the ``shrinkage factor'' $e^{t/\veps}$ is increasing w.r.t $t$, therefore this convergence is uniform in $(T_0,T)$ for $T_0>0$.

Now we examine the limit for $G_{t,k}^{\veps}$. First we note that as $\veps$ approaches zero, the integral of $\frac{1}{\veps}e^{-(t-s)/\veps}$ becomes more and more localized at $s=t$. Then using Assumption \ref{as:N} on the uniform convergence of $N^{\veps}$, we deduce 
\begin{equation}\label{limit-B}
    B^{\veps}(t)=\int_{0}^{t}\frac{1}{\veps}e^{-(t-s)/\veps}g^{\veps}_{\text{in}}(s)ds\rightarrow g_{\text{in}}(t)=g_0+g_1N(t),
\end{equation} as $\veps\rightarrow 0^+$.
Similarly
\begin{equation}\label{limit-C}
    	C^{\veps}(t)=\int_{0}^{t}\frac{2}{\veps}e^{-2(t-s)/\veps}a^{\veps}_{\text{in}}(s)ds\rightarrow a(t)=a_0+a_1N(t),
\end{equation}as $\veps\rightarrow 0^+$. Next we consider $\exp\left(i\veps k\frac{2\pi}{V_F}\myBtwo^{\veps}(t,z)\right)$. From \eqref{eps-def-B2} we get
\begin{align*}
    \veps \myBtwo^{\veps}(t,z)=-\veps(z-B^{\veps}(t))\frac{\int_{0}^{t}e^{(s-t)/\veps}a^{\veps}(s)ds}{\int_{0}^{t}e^{2(s-t)/\veps}a^{\veps}(s)ds}-\int_{0}^{t}g_{\text{in}}^{\veps}(s)ds.
\end{align*}With Assumption \ref{as:N}, the first term goes to zero, and the second term goes to $-\int_{0}^{t}g_{\text{in}}(s)ds$. Therefore
\begin{equation}\label{limit-B2}
    \lim_{\veps\rightarrow 0^+}\exp\left(i\veps k\frac{2\pi}{V_F}\myBtwo^{\veps}(t,z)\right)=\exp\left(-ik\frac{2\pi}{V_F}\int_{0}^{t}g_{\text{in}}(s)ds\right).
\end{equation} Finally we look at the decay factor $\exp\left(-k^2\veps^2(\frac{2\pi}{V_F})^2D^{\veps}(t)\right)$. By the expression \eqref{eps-decay-D(t)},
\begin{equation}\label{order-D-eps}
    \veps^2D^{\veps}(t)=\veps\int_{0}^{t}a^{\veps}(s)ds-\veps\frac{(\int_{0}^{t}e^{(s-t)/\veps}a^{\veps}(s)ds)^2}{\int_{0}^{t}e^{2(s-t)/\veps}a^{\veps}(s)ds},
\end{equation}which goes to zero as $\veps$ goes to $0^+$ thanks to Assumption \ref{as:N}. Therefore we have the following limit
\begin{equation}\label{limit-D}
    \lim_{\veps\rightarrow 0^+}\exp\left(-k^2\veps^2(\frac{2\pi}{V_F})^2D^{\veps}(t)\right)=1.
\end{equation} Combine \eqref{limit-B},\eqref{limit-C},\eqref{limit-B2} and \eqref{limit-D}, we get the limit of $G_{t,k}^{\veps}$
\begin{equation}\label{limit-Gt}
    G_{t,k}^{\veps}(g)\rightarrow \exp\left(-ik\frac{2\pi}{V_F}\int_{0}^{t}g_{\text{in}}(s)ds\right)\frac{1}{\sqrt{2\pi a(t)}}\exp\left(-\frac{(g-g_{\text{in}(t)})^2}{2a(t)}\right)=:G_{t,k}^0,\quad \text{as $\veps\rightarrow 0^+$}.
\end{equation} For a fixed $k$, thanks to the explicit expression, the convergence of $G_{t,k}^{\veps}$ is uniformly in $L^1(\mathbb{R})$ for $t$ in $(T_0,T)$. Moreover, by checking the formula for the derivatives, we can similarly deduce the convergence in $W^{1,1}(\mathbb{R})$.

Combining the $W^{1,1}$ convergence of $G_{t,k}^{\veps}$ \eqref{limit-Gt} with \eqref{limit-2-pek}, also noting that  $e^{i\veps k(\frac{2\pi}{V_F})g}\rightarrow1$ as $\veps\rightarrow0^+$, we deduce from the solution formula \eqref{eps-explicit——k}
\begin{align}\notag
    \pe_k(t,g)&=e^{i\veps k(\frac{2\pi}{V_F})g}(\pe_{t,k}\ast G^{\veps}_{t,k})\\\notag &=e^{i\veps k(\frac{2\pi}{V_F})g}(\pe_{t,k}\ast G^{0}_{t,k}+\pe_{t,k}\ast(G^{\veps}_{t,k}-G^{0}_{t,k}))\\&\rightarrow c_k\exp\left(-ik\frac{2\pi}{V_F}\int_{0}^{t}g_{\text{in}}(s)ds\right)\frac{1}{\sqrt{2\pi a(t)}}\exp\left(-\frac{(g-g_{\text{in}}(t))^2}{2a(t)}\right),\quad \text{as $\veps\rightarrow 0^+$},\label{limit-pk}
\end{align} in $L^1(\mathbb{R})$ and uniformly for $t$ in $(T_0,T)$. Apply the limit \eqref{limit-pk} to each Fourier mode, and exchange the sum and limit, which is ensured by Assumption \ref{as:init}-(ii) and $\|p_k^{\veps}(t,\cdot)\|_{L^{1}(\mathbb{R})}\leq\|p_{k,\init}(g)\|_{L^{1}(\mathbb{R})}$ from \eqref{eps-explicit——k}, and then we get
\begin{equation}
            \pe(t,v,g)\rightarrow p(t,v,g),\quad \text{as $\veps\rightarrow 0^+$},\quad \text{in  }L^{\infty}((T_0,T);L^1((0,V_F)\times\mathbb{R})),
\end{equation} where 
\begin{equation}\label{limit-p}
    \begin{aligned}
p(t,v,g):=\left(\frac{1}{V_F}\sum_{k=-\infty}^{+\infty}c_k\exp\left(ik\frac{2\pi}{V_F}(v-\int_0^tg_{\text{in}}(s)ds)\right)\right)\frac{1}{\sqrt{2\pi a(t)}}\exp\left(-\frac{(g-g_{\text{in}}(t))^2}{2a(t)}\right),
    \end{aligned}
\end{equation} where $c_k$ is defined in \eqref{def-ck} and $g_{\text{in}}(t),a(t)$ are still given by
\begin{equation*}
    g_{\text{in}}(t)=g_0+g_1N(t),\quad a(t)=a_0+a_1N(t),
\end{equation*} as in \eqref{limit-gin-a}. We denote the marginal density in $v$ direction $\rho(t,v)$, as in our previous derivation,
\begin{equation}\label{def-macro}
    \rho(t,v):=\frac{1}{V_F}\sum_{k=-\infty}^{+\infty}c_k\exp\left(ik\frac{2\pi}{V_F}(v-\int_0^tg_{\text{in}}(s)ds)\right),
\end{equation} then 
\begin{equation}
    p(t,v,g)=\rho(t,v)\frac{1}{\sqrt{2\pi a(t)}}\exp\left(-\frac{(g-g_{\text{in}}(t))^2}{2a(t)}\right).
\end{equation} 

Now we check that $\rho(t,v)$ indeed satisfies the fast conductance limit system \eqref{eq:fcl-main}. First we note that Assumption \ref{as:init}-(ii) ensure that
\begin{equation}
    \sum_{k=-\infty}^{+\infty}|c_k|\leq     \sum_{k=-\infty}^{+\infty}\|p_{k,\init}\|_{L^1(\mathbb{R})}<\infty.
\end{equation}Therefore $\rho(t,v)$ defined in \eqref{def-macro} is continuous, satisfies the boundary condition \eqref{eq:fcl-bc} and is a (weak) solution of the transport equation \eqref{eq:fcl}. We also note that $\rho(t,v)=\int_{\mathbb{R}}p(t,v,g)dg$ is a probability density on $(0,V_F)$ for any fixed $t$.

It remains to check the limit of the firing rate $N(t)$. Recall \eqref{eq:nonliner-toy-fire} for $\veps>0$ the firing rate is defined by
\begin{align}\label{fcl-deri-final}
    N^{\veps}(t)=\int_{0}^{+\infty}gp^{\veps}(t,V_F,g)dg.
\end{align} We need to take limit in \eqref{fcl-deri-final} and exchange the integral and limit in the right hand side. Thanks to Assumption \ref{as:init}-(ii), this can be justified by the control on firing rate similarly to \eqref{p5-tmp3-pf} in the proof of Proposition \ref{prop:eps-v}. Then we deduce
\begin{equation}
    N(t)=\rho(t,V_F)\int_0^{+\infty}g\frac{1}{\sqrt{2\pi a(t)}}\exp\left(-\frac{(g-g_{\text{in}}(t))^2}{2a(t)}\right)dg,
\end{equation} which is \eqref{fcl-N-g+}. Since we know by the Assumption \ref{as:N} that $N(t)\in C[0,T]$, by Lemma \ref{lemma:fcl-N} we deduce that $\rho$ satisfies the bound \eqref{fcl-bound-vf}, i.e., $0\leq \rho(t,V_F)< 1/g_1$.

\end{proof}

	\subsection{Periodic solutions versus blow up}\label{sc:4.2analysis}
	Now we analyze the fast conductance limit model \eqref{eq:fcl-main}. First we supply \eqref{eq:fcl-main} with an initial data
	\begin{equation}
	    \rho(0,v)=\rho_{\text{init}}(v),\quad v\in[0,V_F].
	\end{equation}The initial data $\rho_{\init}$ is also a probability density on $[0,V_F]$. We assume that the initial value $\rho_{\text{init}}$ is continuous and compatible with the boundary condition $\rho_{\text{init}}(0)=\rho_{\text{init}}(V_F)$. For the initial firing rate to be well-defined, in view of Lemma \ref{lemma:fcl-N} we need $\rho_{\init}(V_F)<1/g_1$.
	
	Depending on the $L^{\infty}$ norm of the initial value $\rho_{\init}(v)$, the solution of the fast conductance limit model \eqref{eq:fcl-main} either blows up in finite time or globally exists. In the latter case, the solution is periodic in time. These characterizations are given in the following theorem.
	\begin{theorem}\label{thm:longtime-fcl}
Suppose the
	initial data $\rho_{\init}$ is a continuous probability density on $[0,V_F]$. Moreover, it is compatible with the boundary condition: $\rho_{\init}(0)=\rho_{\init}(V_F)$ and satisfies $\rho_{\init}(V_F)<{1}/{g_1}$. Then we have the following results on the fast conductance limit system 
	\eqref{eq:fcl-main}.
	\begin{enumerate}
	    \item If $\max_{v\in[0,V_F]}\rho_{\init}(v)<{1}/{g_1}$, then the solution $\rho$ globally exists and is periodic in time.
	    \item Otherwise if $\max_{v\in[0,V_F]}\rho_{\init}(v)\geq {1}/{g_1}$. Then the firing rate $N(t)$ blows up in finite time. Precisely, there exists $T^*>0$ such that the solution exists on $(0,T^*)$ but
	    \begin{equation}
	        \lim_{t\rightarrow (T^*)^{-}}N(t)=+\infty.
	    \end{equation}
	    Moreover, we have the following upper bound for the blow-up time
	    \begin{equation}
	       0< T^*\leq V_F/g_0.
	    \end{equation}
	\end{enumerate}
	\end{theorem}
	\begin{proof}
	Suppose the solution exists, by a change of variable in time $\tau=\int_0^tg_{\text{in}}(s)ds$, i.e., $d\tau=g_{\text{in}}(t)dt$, we reduce \eqref{eq:fcl-main} to a simple linear equation for $n(\tau,v):=\rho(t,v)$
	\begin{equation}\label{eq:linear-transport}
	\begin{aligned}
		    \p_{\tau}n+\p_vn&=0,\quad \tau>0,v\in(0,V_F),\\
		    n(\tau,0)&=n(\tau,V_F),\quad \tau>0,
	\end{aligned}
	\end{equation} whose solution is
	\begin{equation}\label{linear-transport}
	    n(\tau,v)=\rho_{\text{init}}(v-\tau),
	\end{equation} where $v-\tau$ should be understood in mod $V_F$ sense (or considering the periodic extension of $\rho_{\init}$).
	
	In view of Lemma \ref{lemma:fcl-N}, as long as $n(\tau,V_F)<1/g_1$, we can construct the firing rate $N(\tau)$ from the solution \eqref{linear-transport} of the linear transport equation \eqref{eq:linear-transport}. Then we can construct the solution $\rho$, by changing back from time variable $\tau$ to $t$. 

	In Case 1, we always have $n(\tau,V_F)\leq \max_{v\in[0,V_F]}\rho_{\init}(v)<1/g_1$. Therefore we can construct the firing rate $N(\tau)$ for all $\tau>0$. Moreover, in time variable $\tau$ the solution $n(\tau,v)$ is periodic with the period $V_F$, which implies that $t=\int_0^{\tau}\frac{1}{g_{\myin}(\tilde{\tau})}d\tilde{\tau}$ goes to infinity as $\tau$ goes to infinity. Therefore changing back to the time variable $t$, we get a global solution $p(t,v)$ with the period $T:=\int_0^{V_F}\frac{1}{g_{\myin}(\tau)}d\tau>0$. Here we use $N(\tau),g_{\myin}(\tau)$ for $N$ and $g_{\myin}$ in the time variable $\tau$.
	
	In Case 2, since $\rho_{\init}(V_F)<1/g_1$, we can find a unique $v^*\in(0,V_F)$ such that
\begin{equation*}
    \rho_{\init}(v^*)=1/g_1,\quad \rho_{\init}(v)<1/g_1,\,  \forall v\in(v^*,V_F].
\end{equation*}	
 Then we can solve \eqref{eq:linear-transport} for $\tau<\tau^*:=V_F-v^*$ but at time $\tau^*$ the firing rate can not be defined. Transforming back to time $t$, we get a solution of \eqref{eq:fcl} on the time interval $(0,T^*)$, where
 \begin{equation}
      T^*:=\int_0^{V_F-v^*}\frac{1}{g_{\myin}(\tau)}d\tau\leq \frac{V_F-v^*}{g_0}<\frac{V_F}{g_0}<\infty.
 \end{equation}
 Since by continuity, $\lim_{t\rightarrow (T^*)^{-}}\rho(t,V_F)=1/g_1$, from Lemma \ref{lemma:fcl-N} we deduce that $\lim_{t\rightarrow (T^*)^{-}}N(t)=+\infty$.
	
	\end{proof}
	
	Theorem \ref{thm:longtime-fcl} gives the long time behavior for different initial data under a fixed parameter $g_1$. We can also reformulate it by fixing the initial data and let $g_1$, which reflects the strength of the nonlinearity, vary. 
	\begin{corollary}\label{cor-fcl-fixinit}
	For a fixed initial data $\rho_{\init}(v)$ which is a continuous probability density on $[0,V_F]$ and is compatible with the boundary condition \eqref{eq:fcl-bc} $\rho_{\init}(0)=\rho_{\init}(V_F)$. Then there exists a  threshold 
	$$g_1^*:=1/\left(\max_{v\in[0,V_F]}\rho_{\init}(v)\right).$$
	For $g_1\geq g_1^*$, the solution of \eqref{eq:fcl-main} blows up in finite time. While for $0<g_1<g_1^*$, the solution globally exists and is periodic.
	\end{corollary}
	Here if $\rho_{\init}(V_F)\geq 1/g_1$, we say that the solution blows up at time $t=0$. Corollary \ref{cor-fcl-fixinit} is a direct consequence of Theorem \ref{thm:longtime-fcl}.
	
	We can also show that if $g_1$, the strength of the nonlinearity, is too large, then every solution, with a probability density initial data, blows up in finite time.
	\begin{corollary}\label{cor-fcl-vf-g1}
	If $g_1\geq V_F$, then any solution of \eqref{eq:fcl-main}, with an initial data satisfying conditions in Corollary \ref{cor-fcl-fixinit}, blows up in finite time.
	\end{corollary}
	\begin{proof}[Proof of Corollary \ref{cor-fcl-vf-g1}]
	Since $\rho_{\text{init}}$ is a probability density and continuous on $[0,V_F]$,
	\begin{equation}
	    \max_{v\in[0,V_F]}\rho_{\text{init}}(v)\geq \frac{\int_0^{V_F}\rho_{\text{init}}(v)dv}{V_F}=\frac{1}{V_F}\geq \frac{1}{g_1}.
	\end{equation} Then the result follows from Theorem \ref{thm:longtime-fcl}.
	\end{proof}

	Corollary \ref{cor-fcl-vf-g1} is sharp in the following sense: if $g_1<V_F$, then $\rho_{\infty}(v)\equiv \frac{1}{V_F}$, the density of uniform distribution on $[0,V_F]$, is a steady state with a finite firing rate, therefore a global solution. Moreover, this threshold is consistent with Proposition \ref{prop:g1>=1}, which can be extended to $\veps>0$ similarly using the change of variable \eqref{change-of-v}. Indeed, one can show for the voltage-conductance simplified model \eqref{eq:nonlinear-toy-veps} that the firing rate $N(t)$ goes to infinity as $t$ goes to infinity for any solution when $g_1\geq V_F$.
	
	In view of the limit process $\veps\rightarrow0^+$, we speculate that the finite time blow up of the limit model corresponds to that $N^{\veps}(t)$ goes to infinity as $\veps$ goes to $0^+$. A numerical simulation by adapting the scheme in \cite{caceres2011numerical} is given in Figure \ref{fig:my_label}. We plot the firing rate $N^{\veps}$ for various $\veps$. As $\veps$ goes to zero, the firing rate of the voltage-conductance simplified model \eqref{eq:nonlinear-toy-veps} becomes larger and larger. And for the limit model \eqref{eq:fcl-main}, formally $\veps=0$, the firing rate blows up in finite time. This may reflect an intuition given in \cite{perthame2013voltage}, which says that blow-up happens at a longer timescale than the timescale of the kinetic model. 
	\begin{figure}[H]
	    \centering
	    \includegraphics[width=0.7\linewidth]{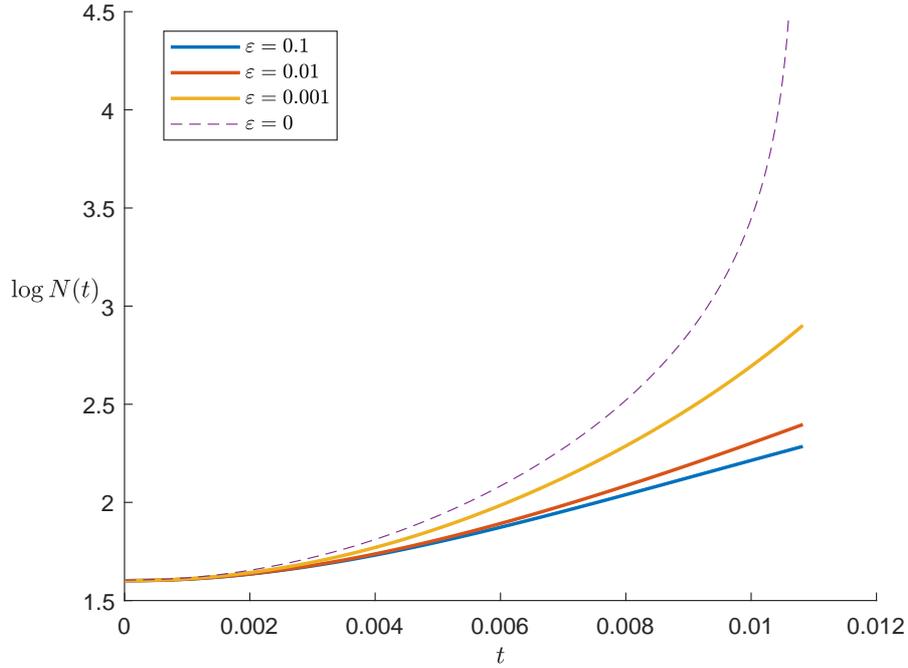}
	    \caption{Firing rate $N(t)$ in a log scale with respect to time for  the voltage-conductance simplified model \eqref{eq:nonlinear-toy-veps} with different $\veps$ versus the fast conductance limit model \eqref{eq:fcl-main} for $\veps=0$. Parameters: $g_0=10,
		g_1=1,
		,a_0=2, a_1=0.1$. In this case, the limit solution blows up.}
	    \label{fig:my_label}
	\end{figure}
On the contrary, if we assume $N^{\veps}(t)$ uniformly converges to some $N(t)$ on $[0,T]$, then the limit solution exists on $[0,T]$ as shown in Proposition \ref{Prop:fcl-limit}. From Theorem \ref{thm:decay-v} we know that when $\veps>0$ the solution of the kinetic voltage-conductance simplified model \eqref{eq:nonlinear-toy} converges to the homogeneous $v$ problem \eqref{eq:reduce-p}. However, Theorem \ref{thm:longtime-fcl} implies that this convergence does not hold in the fast conductance limit model \eqref{eq:fcl-main} since all solutions are periodic. This apparent contradiction can be reconciled by \eqref{eps-decay-D(t)} and \eqref{order-D-eps} in Section \ref{sc:4.1deri}, from which we can observe that the ``decay factor'' for non-constant modes is like $\exp(-\veps k^2 t)$, which vanishes as $\veps$ goes to zero. To illustrate this, we perform numerical simulations by adapting schemes developed in \cite{caceres2011numerical}. In Figure \ref{fig:finiteeps}, we plot the firing rate $N(t)$ for different $\veps$, including the limit case $\veps=0$. For the fast conductance limit model the firing rate is periodic, which is consistent with Theorem \ref{thm:longtime-fcl}. While when $\veps>0$, the solution shows damped oscillations, which last longer and longer as $\veps$ goes to zero.

We remark that results in this section can be directly extended to more general velocity fields $gf(v)$, using the change of variable in Section \ref{sc:gen-vel}.

\begin{figure}[H]
	\centering
	\includegraphics[width=1.0\linewidth]{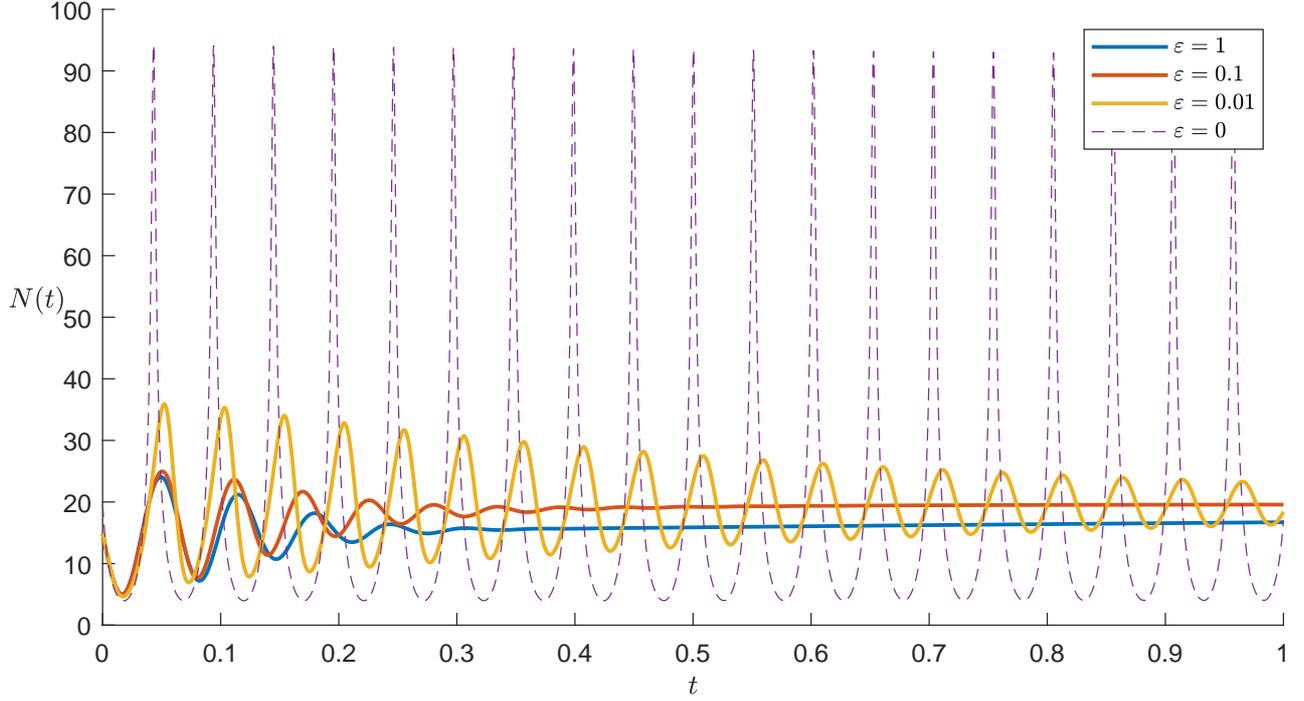}
	\caption{Firing rate $N(t)$ with respect to time for different $\veps>0$ of the voltage-conductance simplified model \eqref{eq:nonlinear-toy-veps} versus the fast conductance limit model \eqref{eq:fcl-main} for $\veps=0$. Parameters: $g_0=10,
		g_1=0.5,
		,a_0=2, a_1=0.1$. In this case, the limit solution is periodic.}
	\label{fig:finiteeps}
\end{figure}

\section*{Acknowledgements}
JAC was supported by the Advanced Grant Nonlocal-CPD (Nonlocal PDEs for Complex Particle Dynamics: Phase Transitions, Patterns and Synchronization) of the European Research Council Executive Agency (ERC) under the European Union's Horizon 2020 research and innovation programme (grant agreement No. 883363). ZZ is supported by the National Key R\&D Program of China, Project Number 2021YFA1001200, and the NSFC, grant Number 12171013. XD is partially supported by The Elite Program of
Computational and Applied Mathematics for PhD Candidates in Peking University.

\begin{appendices}
	\section{Derivation of the solution formula in Lemma \ref{lemma-explicit-solu}}\label{app:deri}
	In this appendix we give detailed calculations for Lemma \ref{lemma-explicit-solu}. Lemma \ref{lemma-explicit-solu} can be seen as an extension of the solution formula for the Fokker Planck equation associated with the OU process \cite{risken1996fokker}.
	
	Recall \eqref{eq:fourier-each-k},
		\begin{equation*}
	\p_t p_k+ik\frac{2\pi}{V_F}gp_k=\p_g\bigl((-g_{\text{in}}(t)+g)p_k+a(t)\p_gp_k\bigr),\quad g\in\mathbb{R},\ t>0.
	\end{equation*}
	
	To be concise, in the following we drop the subscript $k$ and introduce $\mu:=k\frac{2\pi}{V_F}$. Then \eqref{eq:fourier-each-k} becomes
	\begin{equation}\label{eq:mode-k-lambda}
			\p_t p+i\mu gp=\p_g\bigl((-g_{\myin}(t)+g)p+a(t)\p_gp\bigr),\quad g\in\mathbb{R},\ t>0,
		\end{equation}
		We consider the Fourier transform in $g$
		\begin{equation*}
		    \hat{p}(t,\xi):=\mathcal{F}_{g}(p(t,g)):=\frac{1}{\sqrt{2\pi}}\int_{-\infty}^{+\infty}e^{-ig\xi}p(t,g)dg.
		\end{equation*}
		Then we have
		\begin{equation*}
		\widehat{gp}=i\p_{\xi}\hat{p},\quad \widehat{\p_gp}=i\xi \hat{p},\quad \frac{1}{\sqrt{2\pi}}\widehat{f_1\ast f_2}=\hat{f_1}\hat{f_2}.
		\end{equation*}
		In terms of the Fourier transform $\hat{p}$, \eqref{eq:mode-k-lambda} becomes,
		\begin{align*}
\p_t\hat{p}-\mu \p_{\xi}\hat{p}=-ig_{\myin}(t){\xi}\hat{p}-\xi\p_{\xi}\hat{p}-a(t)\xi^2\hat{p},
		\end{align*} which simplifies to 
		\begin{equation}\label{eq:mode-k-lambda-four-app}
\p_t\hat{p}+(\xi-\mu) \p_{\xi}\hat{p}=-(a(t)\xi^2+ig_{\myin}(t)\xi)\hat{p},
		\end{equation}
		This is a first order equation in $\xi$ whose characteristic is given by
		\begin{equation}\label{eq:char-app}
\frac{d\xi(t)}{dt}=\xi(t)-\mu,
		\end{equation}whose solution is
		\begin{equation}\label{eq:char-solu-app}
		    \xi(s)-\mu=e^{s-t}(\xi(t)-\mu).
		\end{equation}
		
		Solving the equation \eqref{eq:mode-k-lambda-four-app} along characteristic \eqref{eq:char-app}, we get
		\begin{equation}\label{app-1-tmp1}
		    \hat{p}(t,\xi(t))=\hat{p}(0,\xi(0))\exp\left(-\int_0^t[a(s)\xi(s)^2+ig_{\myin}(s)\xi(s)]ds\right).
		\end{equation}
		In view of \eqref{eq:char-solu-app}, by identities $x=x-\mu+\mu,x^2=(x-\mu)^2+2\mu(x-\mu)+\mu^2$, where $x=\xi(s)$ in the following calculation, we get
		\begin{equation*}
		\begin{aligned}\int_0^t[a(s)\xi(s)^2+ig_{\myin}(s)\xi(s)]ds&=(\xi(t)-\mu)^2[\int_0^te^{2(s-t)}a(s)ds]\\&+(\xi(t)-\mu)[2\mu\int_0^te^{(s-t)}a(s)ds+i\int_0^te^{s-t}g_{\myin}(s)ds]\\&+\mu^2\int_0^ta(s)ds+i\mu\int_0^tg_{\myin}(s)ds.
		\end{aligned}
		\end{equation*}
		Therefore \eqref{app-1-tmp1} becomes
		\begin{equation}
		    \hat{p}(t,\xi)=\hat{p}(0,e^{-t}(\xi-\mu)+\mu)e^{-H_t(\xi-\mu)},
		\end{equation}
		where $H_t(x)$ is the following quadratic function
		\begin{equation}\label{def-Ht}
		\begin{aligned}
	 H_t(x):=x^2[\int_0^te^{2(s-t)}&a(s)ds]+x[2\mu\int_0^te^{(s-t)}a(s)ds+i\int_0^te^{s-t}g_{\myin}(s)ds]\\&+\mu^2\int_0^ta(s)ds+i\mu\int_0^tg_{\myin}(s)ds.
		\end{aligned}
		\end{equation}
		By the inverse Fourier transform $\mathcal{F}^{-1}$, we have
		\begin{align}\notag
		    p(t,g)=&\mathcal{F}^{-1}_{\xi}[{\hat{p}(0,e^{-t}(\xi-\mu)+\mu)e^{-H_t(\xi-\mu)}}]\\\notag
		    &=e^{i\mu g}\mathcal{F}^{-1}_{\xi}[{\hat{p}(0,e^{-t}\xi+\mu)e^{-H_t(\xi)}}]\\
		    &=e^{i\mu g}\mathcal{F}^{-1}_{\xi}[\hat{p}(0,e^{-t}\xi+\mu)]*(\frac{1}{\sqrt{2\pi}}\mathcal{F}^{-1}_{\xi}[e^{-H_t(\xi)}]).\label{tmp:deri-formula-vg}
		    \end{align}
		    For the first inverse transform, we calculate
		 \begin{align}\notag
		   \mathcal{F}^{-1}_{\xi}[\hat{p}(0,e^{-t}\xi+\mu)]&= \mathcal{F}^{-1}_{\xi}[\hat{p}(0,e^{-t}\xi)]e^{-i\mu e^t g}\\&=e^tp(0,e^tg)e^{-i\mu e^t g}.\label{fourier-init}
		 \end{align}
		 For the second, we rewrite $e^{-H_t(\xi)}$ in \eqref{def-Ht} as
		 \begin{equation}
		     e^{-H_t(\xi)}=A(t)e^{-\frac{1}{2}C(t)\xi^2-\mu(2C_0(t)+iB(t))\xi},
		 \end{equation}
		 where
		 \begin{equation}
		 \begin{aligned}
		 A(t)&:=e^{-\int_0^ta(s)ds-i\mu\int_0^tg_{\myin}(s)ds},\quad C(t):=2\int_0^te^{2(s-t)}a(s)ds, \\C_0(t)&:=\mu\int_0^te^{(s-t)}a(s)ds,\quad B(t):=\int_0^te^{s-t}g_{\myin}(s)ds. 
		 \end{aligned}
		 \end{equation}
		 Then the inverse Fourier transform of $e^{-H_t(\xi)}$ reads
		 \begin{align*}
		 	     \mathcal{F}^{-1}_{\xi}[e^{-H_t(\xi)}]&=\frac{1}{\sqrt{2\pi}}\int_{-\infty}^{+\infty}e^{ig\xi}e^{-H_t(\xi)}d\xi\\
		 	     &=A(t)\left[\frac{1}{\sqrt{2\pi}}\int_{-\infty}^{+\infty}\exp\left(-\frac{1}{2}C\xi^2-(2C_0+iB-ig)\xi\right)d\xi\right]\\
		 	     &=A(t)\left[\frac{1}{\sqrt{2\pi}}\int_{-\infty}^{+\infty}\exp\left(-\frac{1}{2}C(\xi+\frac{2C_0+iB-ig}{C})^2+\frac{(2C_0+iB-ig)^2}{2C}\right)d\xi\right]\\
		 	     &=A(t)\frac{1}{\sqrt{C}}\exp\left(\frac{2C_0^2}{C}+i\frac{2C_0(B-g)}{C}-\frac{(B-g)^2}{2C}\right).
		 \end{align*}
		 Together with \eqref{tmp:deri-formula-vg} and \eqref{fourier-init}, we get
		 \begin{align}\notag
		     p(t,g)&=e^{i\mu g}\mathcal{F}^{-1}_{\xi}[\hat{p}(0,e^{-t}\xi+\mu)]*(\frac{1}{\sqrt{2\pi}}\mathcal{F}^{-1}_{\xi}[e^{-H_t(\xi)}])\\&=e^{i\mu g}[e^tp(0,e^t\cdot)e^{-i\mu e^t \cdot}]\ast(\frac{1}{\sqrt{2\pi}}\mathcal{F}^{-1}_{\xi}[e^{-H_t(\xi)}]).
		 \end{align}
		 Now we return to the notation $p_k$ and substitute $\mu=\pvf$, we get the formulas in Lemma \ref{lemma-explicit-solu}
    	\begin{equation*}
	p_k(t,g)=e^{ik(\frac{2\pi}{V_F})g}(p_{t,k}\ast G_{t,k})(g).
	\end{equation*} 
	Here $p_{t,k}$ is a shrinkage of $p_{0,k}$, which is the initial data for the $k$-th Fourier mode multiplied a shift in frequency,
	\begin{equation*}
	p_{t,k}(y):=e^{t}p_{0,k}(e^ty),\quad p_{0,k}(g):=e^{-ik\frac{2\pi}{V_F}g}\int_{0}^{V_F}p_{\text{init}}(v,g)e^{-ikv\frac{2\pi}{V_F}}dv.
	\end{equation*} 
	And $G_{t,k}=\frac{1}{\sqrt{2\pi}}\mathcal{F}^{-1}_{\xi}[e^{-H_t(\xi)}]$ is a modified Gaussian with a phase factor and a decay factor,
	\begin{equation*}
	G_{t,k}(z)=\frac{1}{\sqrt{2\pi C(t)}}\exp\left(-\frac{(z-B(t))^2}{2C(t)}\right)\exp\left(ik\frac{2\pi}{V_F}\myBtwo(t,z)\right)\exp\left(-k^2(\frac{2\pi}{V_F})^2D(t)\right).
	\end{equation*}Here the mean $B(t)$ and the variance $C(t)$ are given by
	\begin{equation*}
	\begin{aligned}
		B(t)=\int_{0}^{t}e^{-(t-s)}g_{\text{in}}(s)ds=\int_{0}^{t}e^{-(t-s)}(g_0+g_1N(s))ds.\\
	C(t)=2\int_{0}^{t}e^{-2(t-s)}a(s)ds=2\int_{0}^{t}e^{-2(t-s)}(a_0+a_1N(s))ds.
	\end{aligned}
	\end{equation*}
	Moreover $\myBtwo(t,z)$ and $D(t)$ are given by
	\begin{equation*}
	\myBtwo(t,z)=-(z-B(t))\frac{\int_{0}^{t}e^{s-t}a(s)ds}{\int_{0}^{t}e^{2(s-t)}a(s)ds}-\int_{0}^{t}g_{\text{in}}(s)ds,
	\end{equation*} and
	\begin{equation*}
	D(t)=\int_{0}^{t}a(s)ds-\frac{(\int_{0}^{t}e^{s-t}a(s)ds)^2}{\int_{0}^{t}e^{2(s-t)}a(s)ds}
	\end{equation*}

	\section{On the asymptotically autonomous system}\label{sc:asyauto}
	For reader's convenience here we recall the definition of the asymptotically autonomous system and a long time behavior result, taken from the introduction of \cite{thieme1994asymptotically}.
	
	We called an ODE
	\beq\label{auto_1-1}
	\dot{x}=f(t,x),\quad x(t)\in\mathbb{R}^n,
	\eeq asymptotically autonomous if there exists a limit equation
	\beq\label{auto_1-2}
	\dot{y}=g(y),\quad y(t)\in\mathbb{R}^n,
	\eeq such that 
	\beq\label{auto3}
	f(t,x)\rightarrow g(x), \quad t\rightarrow\infty,\quad \text{{locally} uniformly for $x\in\mathbb{R}^n$}.
	\eeq
	In the following we assume $f(t,x)$ and $g(x)$ are continuous function and locally Lipschitz in $x$.
	
	We define the following $\omega$-limit set $\omega(t_0,x_0)$ for a forward bounded solution $x$ of \eqref{auto_1-1} starting at $x_0$ when $t=t_0$,
	\beq
	\omega(t_0,x_0)=\cap_{s>t_0}\overline{\{x(t),t\geq s\}}.
	\eeq
	
	We shall use the following result on the asymptotically autonomous system, which is a direct combination of Theorem 1, 2 and 7 of Markus's \cite{markus1956ii} (which are quoted as Theorem 1.1, 1.2 and 1.3 in 
	\cite{thieme1994asymptotically})
	\begin{theorem_sec}
		\label{thm:auto}
	Suppose the dimension $n=2$ and the system \eqref{auto_1-2} has exactly one equilibrium $y_0$ which is locally asymptotically stable and there is no periodic orbits. Suppose $x$ is forward bounded solution of \eqref{auto_1-1} then 
	\beq
	x(t)\rightarrow y_0,\quad\text{as}\quad t\rightarrow\infty.
	\eeq
	\end{theorem_sec}

	\begin{proof}
		By \cite[Theorem 1.3]{thieme1994asymptotically} in 2D the $\omega$ limit set either contains equilibra or is the union of the periodic orbits of \eqref{auto_1-2}. In our case by the assumption on \eqref{auto_1-2} we have
		\beq
		\omega(t_0,x_0)=\{y_0\}.
		\eeq
		Since $y_0$ is a locally asymptotically stable steady state of \eqref{auto_1-2}, by \cite[Theorem 1.2]{thieme1994asymptotically} we conclude that $x(t)\rightarrow y_0,$ as $t\rightarrow\infty$.
	\end{proof}
	
	To apply Theorem \ref{thm:auto}, besides the study of the limit system \eqref{auto_1-2}, we shall take care to prove the locally uniformly convergence \eqref{auto3} and to prove that the solution is bounded.

\end{appendices}

\bibliography{bibtoy1}

\begin{thebibliography}{10}

\bibitem{abbott1993asynchronous}
L.~F. Abbott and C.~van Vreeswijk.
\newblock Asynchronous states in networks of pulse-coupled oscillators.
\newblock {\em Physical Review E}, 48(2):1483, 1993.

\bibitem{bianchi1995central}
A.~L. Bianchi, M.~Denavit-Saubie, and J.~Champagnat.
\newblock Central control of breathing in mammals: neuronal circuitry, membrane
  properties, and neurotransmitters.
\newblock {\em Physiological reviews}, 75(1):1--45, 1995.

\bibitem{blankenship2010mechanisms}
A.~G. Blankenship and M.~B. Feller.
\newblock Mechanisms underlying spontaneous patterned activity in developing
  neural circuits.
\newblock {\em Nature Reviews Neuroscience}, 11(1):18--29, 2010.

\bibitem{brunel1999fast}
N.~Brunel and V.~Hakim.
\newblock Fast global oscillations in networks of integrate-and-fire neurons
  with low firing rates.
\newblock {\em Neural computation}, 11(7):1621--1671, 1999.

\bibitem{caceres2011analysis}
M.~J. C{\'a}ceres, J.~A. Carrillo, and B.~Perthame.
\newblock Analysis of nonlinear noisy integrate \& fire neuron models: blow-up
  and steady states.
\newblock {\em The Journal of Mathematical Neuroscience}, 1(1):1--33, 2011.

\bibitem{caceres2011numerical}
M.~J. C{\'a}ceres, J.~A. Carrillo, and L.~Tao.
\newblock A numerical solver for a nonlinear fokker--planck equation
  representation of neuronal network dynamics.
\newblock {\em Journal of Computational Physics}, 230(4):1084--1099, 2011.

\bibitem{caceres2020understanding}
M.~J. C{\'a}ceres and A.~Ramos-Lora.
\newblock An understanding of the physical solutions and the blow-up phenomenon
  for nonlinear noisy leaky integrate and fire neuronal models.
\newblock {\em arXiv preprint arXiv:2011.05860}, 2020.

\bibitem{cai2006kinetic}
D.~Cai, L.~Tao, A.~V. Rangan, and D.~W. McLaughlin.
\newblock Kinetic theory for neuronal network dynamics.
\newblock {\em Communications in Mathematical Sciences}, 4(1):97--127, 2006.

\bibitem{cai2004effective}
D.~Cai, L.~Tao, M.~Shelley, and D.~W. McLaughlin.
\newblock An effective kinetic representation of fluctuation-driven neuronal
  networks with application to simple and complex cells in visual cortex.
\newblock {\em Proceedings of the National Academy of Sciences},
  101(20):7757--7762, 2004.

\bibitem{carrillo2013CPDEclassical}
J.~A. Carrillo, M.~d.~M.~González, M.~P. Gualdani, and M.~E. Schonbek.
\newblock Classical solutions for a nonlinear fokker-planck equation arising in
  computational neuroscience.
\newblock {\em Communications in Partial Differential Equations},
  38(3):385--409, 2013.

\bibitem{Antonio_Carrillo_2015}
J.~A. Carrillo, B.~Perthame, D.~Salort, and D.~Smets.
\newblock Qualitative properties of solutions for the noisy integrate and fire
  model in computational neuroscience.
\newblock {\em Nonlinearity}, 28(9):3365--3388, aug 2015.

\bibitem{cormier2020long}
Q.~Cormier, E.~Tanr{\'e}, and R.~Veltz.
\newblock Long time behavior of a mean-field model of interacting neurons.
\newblock {\em Stochastic Processes and their Applications}, 130(5):2553--2595,
  2020.

\bibitem{cormier2021hopf}
Q.~Cormier, E.~Tanr{\'e}, and R.~Veltz.
\newblock Hopf bifurcation in a mean-field model of spiking neurons.
\newblock {\em Electronic Journal of Probability}, 26:1--40, 2021.

\bibitem{delarue2015global}
F.~Delarue, J.~Inglis, S.~Rubenthaler, and E.~Tanr{\'e}.
\newblock Global solvability of a networked integrate-and-fire model of
  mckean--vlasov type.
\newblock {\em The Annals of Applied Probability}, 25(4):2096--2133, 2015.

\bibitem{delarue2015particle}
F.~Delarue, J.~Inglis, S.~Rubenthaler, and E.~Tanr{\'e}.
\newblock Particle systems with a singular mean-field self-excitation.
  application to neuronal networks.
\newblock {\em Stochastic Processes and their Applications}, 125(6):2451--2492,
  2015.

\bibitem{dou:hal-03586715}
X.~Dou, B.~Perthame, D.~Salort, and Z.~Zhou.
\newblock {Bounds and long term convergence for the voltage-conductance kinetic
  system arising in neuroscience}.
\newblock preprint,hal-03586715, Feb. 2022.

\bibitem{fusi1999collective}
S.~Fusi and M.~Mattia.
\newblock Collective behavior of networks with linear (vlsi) integrate-and-fire
  neurons.
\newblock {\em Neural Computation}, 11(3):633--652, 1999.

\bibitem{gerstner2014neuronal}
W.~Gerstner, W.~M. Kistler, R.~Naud, and L.~Paninski.
\newblock {\em Neuronal dynamics: From single neurons to networks and models of
  cognition}.
\newblock Cambridge University Press, 2014.

\bibitem{gray1994synchronous}
C.~M. Gray.
\newblock Synchronous oscillations in neuronal systems: mechanisms and
  functions.
\newblock {\em Journal of computational neuroscience}, 1(1):11--38, 1994.

\bibitem{hambly2019mckean}
B.~Hambly, S.~Ledger, and A.~S{\o}jmark.
\newblock A mckean--vlasov equation with positive feedback and blow-ups.
\newblock {\em The Annals of Applied Probability}, 29(4):2338--2373, 2019.

\bibitem{hirsch2006monotone}
M.~W. Hirsch and H.~Smith.
\newblock Monotone dynamical systems.
\newblock In {\em Handbook of differential equations: ordinary differential
  equations}, volume~2, pages 239--357. Elsevier, 2006.

\bibitem{ikeda2021theoretical}
K.~Ikeda, P.~Roux, D.~Salort, and D.~Smets.
\newblock Theoretical study of the emergence of periodic solutions for the
  inhibitory nnlif neuron model with synaptic delay.
\newblock 2021.

\bibitem{kim2021fast}
J.~Kim, B.~Perthame, and D.~Salort.
\newblock Fast voltage dynamics of voltage--conductance models for neural
  networks.
\newblock {\em Bulletin of the Brazilian Mathematical Society, New Series},
  52(1):101--134, 2021.

\bibitem{mna:8775}
J.-G. Liu, Z.~Wang, Y.~Xie, Y.~Zhang, and Z.~Zhou.
\newblock {Investigating the integrate and fire model as the limit of a random
  discharge model: a stochastic analysis perspective}.
\newblock {\em {Mathematical Neuroscience and Applications}}, {Volume 1}, Nov.
  2021.

\bibitem{markus1956ii}
L.~Markus.
\newblock Ii. asymptotically autonomous differential systems.
\newblock {\em Contributions to the Theory of Nonlinear Oscillations}, (36):17,
  1956.

\bibitem{pakdaman2013relaxation}
K.~Pakdaman, B.~Perthame, and D.~Salort.
\newblock Relaxation and self-sustained oscillations in the time elapsed neuron
  network model.
\newblock {\em SIAM Journal on Applied Mathematics}, 73(3):1260--1279, 2013.

\bibitem{perthame2013voltage}
B.~Perthame and D.~Salort.
\newblock On a voltage-conductance kinetic system for integrate and fire neural
  networks.
\newblock {\em Kinetic and Related Models}, 6(4):841--864, 2013.

\bibitem{perthame2018derivation}
B.~Perthame and D.~Salort.
\newblock {Derivation of a voltage density equation from a voltage-conductance
  kinetic model for networks of integrate-and-fire neurons.}
\newblock {\em {Communications in Mathematical Sciences}}, 17(5), 2019.

\bibitem{risken1996fokker}
H.~Risken and T.~Frank.
\newblock {\em The Fokker-Planck Equation: Methods of Solution and
  Applications}, volume~18.
\newblock Springer Science \& Business Media, 1996.

\bibitem{roux2021towards}
P.~Roux and D.~Salort.
\newblock Towards a further understanding of the dynamics in the excitatory
  nnlif neuron model: blow-up and global existence.
\newblock {\em Kinetic \& Related Models}, 14(5):819, 2021.

\bibitem{thieme1994asymptotically}
H.~R. Thieme.
\newblock Asymptotically autonomous differential equations in the plane.
\newblock {\em The Rocky Mountain Journal of Mathematics}, pages 351--380,
  1994.

\bibitem{villani2009hypocoercivity}
C.~Villani.
\newblock Hypocoercivity. 949-951.
\newblock {\em American Mathematical Soc}, 2009.

\end{thebibliography}
\bibliographystyle{abbrv}
\end{document}